\documentclass[11pt,reqno]{article}
\usepackage[utf8]{inputenc}
\usepackage{booktabs} % for much better looking tables
\usepackage{array} % for better arrays (eg matrices) in maths
\usepackage{paralist} % very flexible & customisable lists (eg. enumerate/itemize, etc.)
\usepackage{verbatim} % adds environment for commenting out blocks of text & for better verbatim
\usepackage{amsthm}
\usepackage[table,xcdraw]{xcolor}
\usepackage[titletoc,toc,title]{appendix}
\usepackage{latexsym, amsmath, amssymb, a4, epsfig, color}
\usepackage{blindtext}
\usepackage{graphicx}
\usepackage{caption}
\usepackage{subcaption}
\usepackage[export]{adjustbox}
\usepackage{wrapfig}
\usepackage{apptools}
\usepackage{afterpage}
\usepackage{tabularx}
\usepackage{floatrow}
\usepackage{mathtools}
\usepackage[normalem]{ulem} % to use strikethrough %\sout{text}
\usepackage{algorithm}
\usepackage[noend]{algpseudocode}
\usepackage[T1]{fontenc}
\usepackage[font=normal,labelfont=bf,tableposition=top]{caption}

\DeclareCaptionLabelFormat{andtable}{#1~#2  \&  \tablename~\thetable}
\AtAppendix{\counterwithin{definition}{subsection}}
\AtAppendix{\counterwithin{theorem}{subsection}}
\newcolumntype{C}[1]{>{\centering\arraybackslash}p{#1}}
\newfloatcommand{capbtabbox}{table}[][\FBwidth]

\graphicspath{}

\newtheorem{theorem}{Theorem}[section]
\newtheorem{cor}[theorem]{Corollary}
\newtheorem{lemma}[theorem]{Lemma}

\newtheorem{definition}{Definition}

\newtheorem{remark}{Remark}
\newcommand{\Ga}{\alpha}
\newcommand{\Gb}{\beta}

\newcommand{\Gs}{\sigma}

\newcommand{\GO}{\Omega}

\newcommand{\FF}{\mathbb{F}}
\newcommand{\RR}{\mathbb{R}}
\newcommand{\CC}{\mathbb{C}}
\newcommand{\NN}{\mathbb{N}}

\newcommand{\Om}{\Omega}
\newcommand{\ds}{\displaystyle}
\newcommand{\pf}{\noindent {\sl Proof}. \ }
\newcommand{\p}{\partial}
\newcommand{\pd}[2]{\frac {\p #1}{\p #2}}
\newcommand{\eqnref}[1]{(\ref {#1})}
\renewcommand{\qed}{\hfill $\Box$ \medskip}
\newcommand{\beq}{\begin{equation}}
\newcommand{\eeq}{\end{equation}}
\def\ep{\varepsilon}

\newcommand{\KstarOmega}{\mathcal{K}_{\partial\Omega}^{*}}
\newcommand{\KOmega}{\mathcal{K}_{\partial\Omega}}
\newcommand{\Kcal}{\mathcal{K}}
\newcommand{\Scal}{\mathcal{S}}
\newcommand{\Dcal}{\mathcal{D}}

\newcommand{\la}{\langle}
\newcommand{\ra}{\rangle}
\numberwithin{equation}{section}
\numberwithin{figure}{section}

\begin{document}

\newcommand{\TheTitle}{Analytical shape recovery of a conductivity inclusion based on Faber polynomials}
\newcommand{\TheAuthors}{D. Choi, J. Kim and M. Lim}
\title{{\TheTitle}
\thanks{The third author is the corresponding author. This work was supported by the National Research Foundation of Korea (NRF) grant funded by the Korean government (MSIT) No. 2019R1F1A1062782.}}

\author{
Doosung Choi\thanks{\footnotesize Department of Mathematical Sciences, Korea Advanced Institute of Science and Technology, Daejeon 34141, Korea ({7john@kaist.ac.kr}, {kjb2474@kaist.ac.kr},  {mklim@kaist.ac.kr})}\and Junbeom Kim\footnotemark[2] \and Mikyoung Lim\footnotemark[2]}
\date{\today}
\maketitle

\begin{abstract}
A conductivity inclusion, inserted in a homogeneous background, induces a perturbation in the background potential. This perturbation admits a multipole expansion whose coefficients are the so-called generalized polarization tensors (GPTs). GPTs can be obtained from multistatic measurements. As a modification of GPTs, the Faber polynomial polarization tensors (FPTs) were recently introduced in two dimensions. In this study, we design two novel analytical non-iterative methods for recovering the shape of a simply connected inclusion from GPTs by employing the concept of FPTs. First, we derive an explicit expression for the coefficients of the exterior conformal mapping associated with an inclusion in a simple form in terms of GPTs, which allows us to accurately reconstruct the shape of an inclusion with extreme or near-extreme conductivity. Secondly, we provide an explicit asymptotic formula in terms of GPTs for the shape of an inclusion with arbitrary conductivity by considering the inclusion as a perturbation of its equivalent ellipse. With this formula, one can non-iteratively approximate an inclusion of general shape with arbitrary conductivity, including a straight or asymmetric shape. Numerical experiments demonstrate the validity of the proposed analytical approaches.
\end{abstract}



%\tableofcontents

\section{Introduction} 
We consider the imaging problem of an elastic or electrical inclusion in two dimensions. Let $\Om$ be a simply connected domain containing the origin with $\mathcal{C}^2$-boundary.
We assume that the background $\RR^2\setminus\overline{\Om}$ and the inclusion $\Omega$ have constant isotropic conductivities $1$ and $\sigma$, respectively, with $0<\sigma \neq 1 < \infty$. 
We consider the conductivity transmission problem:
\beq\label{cond_eqn0}
\begin{cases}
\ds \Delta u =0\quad&\mbox{in } \RR^2\setminus\partial{\Om},\\[1.5mm]
\ds u\big|^+ =u\big|^-\quad&\mbox{on }\p\Om,\\[1.5mm]
\ds\pd{u}{\nu}\Big|^+ =\sigma\pd{u}{\nu}\Big|^- \quad&\mbox{on }\p\Om,\\[1.5mm]
\ds u(x)  =H(x)+O({|x|^{-1}})\quad&\mbox{as } |x| \to \infty,
\end{cases}
\eeq
where $H$ is a given entire harmonic function, $\nu$ is the outward unit normal vector to $\p\Om$, and the symbols $+$ and $-$ indicate the limit from the exterior and interior of $\Om$, respectively. 
The inclusion $\Om$ induces a change in the background potential, and this perturbation, $u-H$, admits a multipole expansion whose coefficients are the so-called generalized polarization tensors (GPTs).
The purpose of this paper is to design two novel analytical approaches to reconstruct the shape of an inclusion from GPTs.

GPTs are complex-valued tensors which generalize the classical polarization tensors (PTs) \cite{Ammari:2013:MSM, Polya:1951:IIM}. One can acquire the values of GPTs from multistatic measurements \cite{Ammari:2007:PMT},  where a high signal-to-noise ratio is required to get high-order terms \cite{Ammari:2014:TIU}. The concept of GPTs has been a fundamental building block in imaging problems; see, for example, \cite{Ammari:2013:MSM, Ammari:2014:GPT, Ammari:2007:PMT, Choi:2018:CEP}. 
The uniqueness of the inverse problem of determining inclusions from GPTs is known \cite{Ammari:2003:PGP}. 
The concept of GPTs has also been used in a variety of interesting contexts, such as invisible cloaking \cite{Ammari:2013:ENCa, Choi:2018:GME} and plasmonic resonance \cite{Ammari:2016:SPR, Feng:2018:CGV}. Furthermore, a recent series of studies reported the super-resolution of a nanoscale object that overcomes diffraction limits using GPTs \cite{Ammari:2018:MNF, Ammari:2018:RFD}.
The spectrum of the NP operator has recently drawn significant attention in relation to plasmonic resonances
\cite{Ammari:2013:STN, Kang:2017:SRN,  Yu:2017:SDA}.

A powerful approach for recovering the shape of a simply connected conductivity inclusion in two dimensions has been to use a complex analytic formulation for the conductivity transmission problem; for example, see \cite{Ammari:2018:MNF, Ammari:2007:PMT, Choi:2018:CEP, Kang:2015:CCM}.
The Riemann mapping theorem ensures that there exists uniquely a conformal mapping from a region outside a disk to the region outside the inclusion. 
Analytic expressions for the coefficients of this exterior conformal mapping were obtained in terms of GPTs and applied to accurate shape recovery \cite{Choi:2018:CEP, Kang:2015:CCM}, given that the inclusion is insulating.
For such a case, the multipole expansion of $u-H$ admits an extension up to $\p \Om$ on which the flux of $u$ is prescribed to be zero. Using this extension, one can directly express the coefficients of the conformal mapping in terms of those of the multipole expansion ({\it i.e.}, GPTs), and the conformal mapping determines the shape of the inclusion.
One can also expect similar results for the perfectly conducting case by considering a harmonic conjugate of $u$. 
However, for an inclusion with arbitrary conductivity, the boundary value of $u$ is then not explicit anymore and it is a challenging question to generalize the analytical formulas in \cite{Choi:2018:CEP, Kang:2015:CCM} to the arbitrary conductivity case. 
In this paper, we provide two asymptotic answers to this question and, as direct applications, design non-iterative methods to recover the shape of an inclusion.

First, we modify the expression of the conformal mapping obtained in \cite{Choi:2018:CEP, Kang:2015:CCM} to be applicable to both insulating and perfectly conducting cases. We then validate that this improved formula approximately holds also for the near-extreme conductivity case. It allows us to accurately reconstruct the shape of an inclusion with extreme or near-extreme conductivity, as shown in numerical examples in section \ref{sec:numerical}.

Secondly, we derive an explicit asymptotic formula for the shape of an inclusion which allows us to approximate an inclusion of arbitrary conductivity with general shape, including a straight or asymmetric shape.
This result is strongly related to the approach of asymptotic analysis on the boundary integral formulation for the conductivity transmission problem, which holds independently of the value of the conductivity, to capture the shape of a conductivity anomaly \cite{Ammari:2014:GPT, Ammari:2010:CIP,  Ammari:2012:GPT, Ammari:2018:IAD, Khelifi:2014:BVP}. 
For a target given by a small perturbation of a disk, the shape perturbation was asymptotically expressed in terms of GPTs \cite{Ammari:2010:CIP}.
For an inclusion of general shape, one can alternatively find an equivalent ellipse that admits the same values of PTs. Iterative optimization methods have been further developed to capture shape details by using higher-order GPTs as well as first-order terms \cite{Ammari:2014:GPT, Ammari:2012:GPT, Ammari:2018:IAD}.
Our result in this paper significantly improves the result obtained in \cite{Ammari:2010:CIP} by regarding an inclusion as a perturbation of its equivalent ellipse instead of a perturbation of a disk; a straight or asymmetric shape can now be well recovered, as shown in numerical examples in section 6. In the derivation, we use the asymptotic integral representation of the shape derivative of GPTs in \cite{Ammari:2010:CIP}.
While this integral formula is too complicated in Cartesian coordinates to be expressed as an explicit analytic form for a perturbation of an ellipse, we overcome this difficulty by using the curvilinear orthogonal coordinates associated with the exterior conformal mapping and the related series solution method introduced in \cite{Jung:2018:NSS}.
As a result, we design an analytic shape recovery method (see Theorem \ref{et} in section \ref{arbtcond}) that is non-iterative, differently from those in the previous literature \cite{Ammari:2014:GPT, Ammari:2012:GPT, Ammari:2018:IAD}.

As the main tool, we employ the concept of Faber polynomial polarization tensors (FPTs), recently introduced in \cite{Choi:2018:GME}, which are linear combinations of GPTs with the coefficients determined by Faber polynomials (see section \ref{sec:matrix} for the definition and properties of FPTs). The Faber polynomials were first introduced by G. Faber \cite{Faber:1903:UPE} and have been successfully applied in various areas, including numerical approximation \cite{Curtiss:1966:SDP, Ellacott:1983:CFS}, interpolation theory \cite{Chui:1992:FSA, Curtiss:1964:HIF} and material science \cite{Gao:2004:FSM, Luo:2009:FSM}.
For any simply connected region in the complex plane, the Faber polynomials are defined in association with the exterior conformal mapping and they form a basis for analytic functions \cite{Duren:1983:UF}. 
A series solution method for the two-dimensional conductivity transmission problem was developed using the Faber polynomials \cite{Jung:2018:NSS}, and this result was successfully applied to estimate the decay rate of eigenvalues of the Neumann-Poincar\'{e} operator \cite{Jung:2019:DEE}. 
The authors of the present paper recently introduced FPTs in \cite{Choi:2018:GME}. We analyze the relations among GPTs, FPTs, and the exterior conformal mapping to derive the main results in this paper.

The remainder of this paper is organized as follows. Section 2 describes the layer potential technique, GPTs, and the shape recovery formula for a perturbed disk. Section 3 is devoted to the Faber polynomials, FPTs, and the matrix formulation for the transmission problem.
We provide analytic shape recovery methods in section \ref{sec:ERF} and section \ref{arbtcond}.
 Numerical results are presented in section 6. The paper ends with a conclusion in section 7.

%%%%%%%%%%%%%%%%%%%%%
\section{Generalized Polarization Tensors (GPTs)}
We first review the boundary integral formulation for the conductivity transmission problem and describe the definition and some essential properties of GPTs. We also review an asymptotic formula for the shape derivative of GPTs and its application to recovering the shape of a perturbed disk.

\subsection{Boundary integral formulation for the transmission problem}
For a density function $\varphi\in L^2(\p\Om)$, we define the single- and double-layer potentials  associated with $\p\Om$ as
\begin{align*}
\ds&\Scal_{\p\Om}[\varphi](x)= \frac{1}{2\pi}\int_{\p \Om} \ln|x-y| \varphi(y)\, d\sigma(y),\quad x\in\RR^2,\\[1.5mm]
\ds&\Dcal_{\p\Om}[\varphi](x)= \frac{1}{2\pi}\int_{\p \Om} \frac{\p}{\p \nu_y}\ln|x-y| \varphi(y)\, d\sigma(y),\quad x\in\RR^2\setminus \p \Om.
\end{align*}
They satisfy the jump relations \cite{Verchota:1984:LPR}:
\begin{align*}
\ds\frac{\partial}{\partial\nu}\Scal_{\p\Om}[\varphi]\Big|^{\pm}&=\left(\pm\frac{1}{2}I+\Kcal_{\p\Om}^*\right)[\varphi]\quad\text{on }\partial\Omega,\\[1.5mm]
\ds	\Dcal_{\p\Om}[\varphi]\Big|^{\pm}&=\left(\mp\frac{1}{2}I+\Kcal_{\p\Om} \right)[\varphi]\quad\text{on }\partial\Omega
\end{align*}
with
$$
\ds \Kcal_{\p\Om}[\varphi](x)= \frac{1}{2\pi} \ p.v.\int_{\partial\Omega} \frac{\left\la y-x,\nu_y\right\ra}{|x-y|^2}\varphi(y)\, d\sigma(y),\quad x\in\p \Om,
$$
and its $L^2$-adjoint
$$
\ds \Kcal_{\p\Om}^*[\varphi](x)= \frac{1}{2\pi} \ p.v.\int_{\partial\Omega} \frac{\left\la x-y,\nu_x\right\ra}{|x-y|^2}\varphi(y)\, d\sigma(y),\quad x\in\p \Om.
$$
The symbol $p.v.$ stands for the Cauchy principal value. We call $\Kcal_{\p\Om}$ and $\Kcal_{\p\Om}^*$ the Neumann--Poincar\'{e} (NP) operators associated with $\Om$. 

One can express the solution to \eqnref{cond_eqn0} as
\beq\label{u:int:formulation}
u(x)=H(x)+\Scal_{\p\Om}[\varphi](x),\quad x\in\RR^2,
\eeq
with
\beq\label{eqn:integral}
\varphi=(\lambda I-\Kcal_{\p\Om}^*)^{-1}\left[\nu\cdot \nabla H\right], \quad \lambda = \frac{\sigma+1}{2(\sigma-1)}.
\eeq
The operator $\lambda I-\Kcal_{\p\Om}^*$ is invertible on $L^2_0(\p\Om)$ for $|\lambda|\geq1/2$ \cite{Escauriaza:1992:RTW, Escauriaza:1993:RPS,Kellogg:2012:FPT,Verchota:1984:LPR}. We refer to \cite{Ammari:2007:PMT} for more properties of the NP operator and to \cite{Helsing:2013:SIE, Helsing:2017:CSN} for numerical computation of \eqnref{eqn:integral}.

\subsection{Definition and properties of GPTs}
We identify $x=(x_1,x_2) \in\RR^2$ with $z=x_1+ix_2\in\CC$. Following \cite{Ammari:2013:MSM}, we define the (complex contracted) GPTs:
\begin{definition}[GPTs]
For each natural number $k$, we set $P_k(z)=z^k$. For $m,n\in\NN$, we define
\begin{align}
\ds\NN_{mn}^{(1)}(\Om, \lambda)&=\int_{\p\Om} P_n(z) (\lambda I-\Kcal^*_{\p\Om})^{-1}\left[\pd{ P_m }{\nu} \right](z) \,d\sigma(z),\label{def:GPT1}\\
\ds\NN_{mn}^{(2)}(\Om, \lambda)&=\int_{\p\Om} P_n(z) (\lambda I-\Kcal^*_{\p\Om})^{-1}\left[\pd{\overline{P_m}}{\nu}\right](z) \,d\sigma(z). \label{def:GPT2}
\end{align}
If not specified otherwise, we set $\lambda=\frac{\sigma+1}{2(\sigma-1)}$. 
\end{definition}

The background potential $H$ is a real-valued entire harmonic function so that it satisfies the expansion
\beq\label{expan:H}
H(z) = \sum_{m=1}^\infty \left(\alpha_m z^m+\overline{\alpha_m z^m}\right)
\eeq
for some complex coefficients $\alpha_m$. One can then derive from \eqnref{u:int:formulation} and the definition of GPTs that the solution to \eqnref{cond_eqn0} admits the multipole expansion (see \cite{Ammari:2013:MSM})
\begin{align}\label{CP2}
u(z) - H(z)= &-\sum_{n=1}^\infty \sum_{m=1}^\infty\frac{1}{4\pi n}\left( \alpha_m \NN_{mn}^{(1)}+\overline{\alpha_m}\, \NN_{mn}^{(2)}\right){z^{-n}}\notag\\
&-\sum_{n=1}^\infty \sum_{m=1}^\infty\frac{1}{4\pi n}\left(\overline{\alpha_m}\,\overline{ \NN_{mn}^{(1)}}+\alpha_m \overline{\NN_{mn}^{(2)}}\right)\overline{z^{-n}},\quad|z|\gg1.
\end{align}
Hence, GPTs quantitatively express the perturbation of a background potential function due to the presence of an inclusion. As the following theorem asserts, the full information of GPTs uniquely determine the geometry and conductivity of an inclusion.
\begin{theorem}[\cite{Ammari:2003:PGP}]\label{uniquenessGPT}
Let $\Om_1$ and $\Om_2$ be inclusions with conductivities $\sigma_1$ and $\sigma_2$, respectively. Set $\lambda_1=\frac{\sigma_1+1}{2(\sigma_1-1)}$ and $\lambda_2=\frac{\sigma_2+1}{2(\sigma_2-1)}$. If
$
\NN_{mn}^{(j)} (\Om_1, \lambda_1) = \NN_{mn}^{(j)} (\Om_2, \lambda_2)$ for all $m,n\in\NN$, $j=1,2$, then $\Om_1=\Om_2$ and $\sigma_1=\sigma_2$.
\end{theorem}
One of the essential properties of GPTs is symmetricity (for the derivation see \cite[Proposition 11.2]{Ammari:2013:MSM}):
\begin{lemma} \label{NNproperty}
For all $m,n\in\NN$, it holds that 
\beq\label{eqn:GPT:symmetric}
\NN_{mn}^{(1)} = \NN_{nm}^{(1)},\quad\overline{\NN_{mn}^{(2)}} = \NN_{nm}^{(2)}.
\eeq
 In other words, $\NN^{(1)}$ is symmetric and $\NN^{(2)}$ is Hermitian.
\end{lemma}

\subsection{Shape derivative of GPTs}\label{sec:shapederi}

For a multi-index $\alpha=(\alpha_1,\alpha_2)\in \NN^2$, we set $x^\alpha=x_1^{\alpha_1}x_2^{\alpha_2}$ and $|\alpha|=|\alpha_1|+|\alpha_2|$.
We consider an alternative real-valued form of GPTs for multi-indices $\alpha$ and $\beta$:
$$
M_{\alpha\beta}(\Om,\lambda) = \int_{\p \GO} y^\Gb\left(\lambda
I - \Kcal^*_{\p\GO}\right)^{-1}\left[\pd{ x^\Ga}{\nu}\right](y) \, d\Gs (y).
$$
One can easily find that $\NN_{mn}^{(1)}$ and $\NN_{mn}^{(2)}$ are linear combinations of $M_{\alpha\beta}$ with $|\alpha|=m$ and $|\beta|=n$. 
More precisely, 
\begin{align*}
\mathbb{N}_{mn}^{(1)}(\Om,\lambda)&=\sum_{\alpha, \beta} a^{(1)}_\alpha b^{(1)}_\beta M_{\alpha \beta} (\Om, \lambda),\\
\mathbb{N}_{mn}^{(2)}(\Om,\lambda)&=\sum_{\alpha, \beta} a^{(2)}_\alpha b^{(2)}_\beta M_{\alpha \beta} (\Om, \lambda)
\end{align*}
with the multi-indexed coefficients $a^{(1)}_\alpha, b^{(1)}_\beta, a^{(2)}_\alpha, b^{(2)}_\beta $ satisfying $\sum_\alpha a^{(1)}_\alpha x^\alpha=z^m$, $\sum_{\beta}b^{(1)}_\beta x^\beta=z^n$, $\sum_\alpha a^{(2)}_\alpha x^\alpha=\overline{z^m}$ and $\sum_{\beta}b^{(2)}_\beta x^\beta=z^n$, respectively.

We have the following lemma assuming that $\Om$ is a simply connected domain with $\mathcal{C}^2$-boundary given by a small perturbation of $\Om_0$, {\it i.e.}, 
\beq\label{Om:deform}
\p \Om = \left\{ z + \ep f(z) \nu_0(z) : \, z\in \p \Om_0 \right\}
\eeq
with a real-valued function $f\in \mathcal{C}^1(\p \Om_0)$, 
where $\nu_0$ is outward unit normal to $\p \Om_0$.
\begin{lemma}[\cite{Ammari:2012:GPT}] \label{lemma:shapederivative} 
Let $a_\alpha$ and $b_\beta$ be two multi-indexed sequences such that $H=\sum_\alpha a_\alpha x^\alpha$ and $F=\sum_\beta b_\beta x^\beta$ are harmonic polynomials. For $\Om$ satisfying \eqnref{Om:deform} and $\lambda=\frac{\sigma+1}{2(\sigma-1)}$, it holds that
\begin{align}\notag
&\sum_{\alpha, \beta} a_\alpha b_\beta M_{\alpha \beta} (\Om, \lambda) - \sum_{\alpha, \beta} a_\alpha b_\beta M_{\alpha \beta} (\Om_0, \lambda) \\ \label{lemma:shapederivative:eqn}
=&\; \ep (\sigma-1) \int_{\p \Om_0} f(x) \left(\pd{{u}}{\nu}\Big|^-\pd{v}{\nu}\Big|^- +\frac{1}{\sigma}\pd{{u}}{T}\Big|^- \pd{v}{T}\Big|^- \right)(x) \,d \sigma(x) + O(\ep^2),
\end{align}
where $T$ is the positively oriented unit tangent vector on $\p \Om$ and $u$, $v$ are the solutions to
\beq\label{cond_eqn1}
\begin{cases}
\ds\Delta u=0\quad&\mbox{in } \RR^2 \setminus \p \Om_0,\\
\ds u\big|^+=u\big|^-\quad&\mbox{on }\p \Om_0, \\
\ds \pd{u}{\nu}\Big|^+=\sigma\pd{u}{\nu}\Big|^-\quad&\mbox{on }\p \Om_0, \\
\ds (u - H)(x) =O({|x|^{-1}})\quad&\mbox{as } |x| \to \infty
\end{cases}
\eeq
and
\beq\label{cond_eqn2}
\begin{cases}
\ds\Delta v=0\quad&\mbox{in } \RR^2 \setminus \p \Om_0,\\
\ds \sigma v\big|^+=v\big|^-\quad&\mbox{on }\p \Om_0, \\
\ds \pd{v}{\nu}\Big|^+=\pd{v}{\nu}\Big|^-\quad&\mbox{on }\p \Om_0, \\
\ds (v - F)(x) =O({|x|^{-1}})\quad&\mbox{as } |x| \to \infty.
\end{cases}
\eeq
\end{lemma}

\subsection{Recovering the shape of a perturbed disk}\label{sec:pertdisk}
We can asymptotically solve \eqnref{lemma:shapederivative:eqn} for $f$ given that $\Om_0$ is a disk; it gives the same formula obtained in \cite{Ammari:2010:CIP}.
Let $\Om_0$ be a disk, namely $D$, centered at the origin with radius $\gamma_D$ for some $\gamma_D>0$. Again, we identify $\RR^2$ with $\CC$. Then, the solutions $u$ and $v$ corresponding to the harmonic functions $H(z)=z^m$ and $F(z)=z^n$ are 
\beq\label{udisk}
u(z)=
\begin{dcases}
\begin{aligned}
\ds & \frac{2}{\sigma+1}\,z^m,\quad && z\in D,\\
\ds & z^m - \frac{\sigma-1}{\sigma+1}\, \gamma_D^{2m}\, \overline{z^{-m}},\quad && z\in \CC\setminus \overline{D}
\end{aligned}
\end{dcases}
\eeq
and
\beq\label{vdisk}
v(z)=
\begin{dcases}
\begin{aligned}
\ds &\frac{2\sigma}{\sigma+1}\, z^n,\quad && z\in D,\\
\ds &z^n - \frac{\sigma-1}{\sigma+1}\, \gamma_D^{2n}\, \overline{z^{-n}},\quad && z\in \CC\setminus \overline{D}.
\end{aligned}
\end{dcases}
\eeq
From Lemma \ref{lemma:shapederivative}, we arrive the following theorem. 
\begin{theorem}[\cite{Ammari:2010:CIP}]\label{dd}
Suppose $D$ is a disk centered at the origin with radius $\gamma_D>0$. Let $\Om$ be a small perturbation of $D$ in the form of \eqnref{Om:deform}. For each $m,n \in\NN$, we have
\begin{align}
&\ep\widehat{f}_{m-n} = \frac{\lambda^2(\sigma-1)}{2\pi mn\gamma_D^{m+n-2}(\sigma+\gamma_D^2)} \left( \NN_{mn}^{(2)} (\Om,\lambda)  - \NN_{mn}^{(2)} (D,\lambda) \right) + O(\epsilon^2), \label{firstdisk} \\
&\ep\widehat{f}_{m+n} = \frac{\lambda^2(\sigma-1)}{2\pi  mn\gamma_D^{m+n-2}(\sigma-\gamma_D^2)} \left( \overline{\NN_{mn}^{(1)}} (\Om,\lambda)  - \overline{\NN_{mn}^{(1)}} (D,\lambda) \right) + O(\epsilon^2), \label{seconddisk}
\end{align}
where $\widehat{f}_k$ denotes the Fourier coefficient of $f$, {\it i.e.}, 
$\widehat{f}_k = \frac{1}{2\pi}\int_{0}^{2\pi} f(\gamma_D e^{i\theta})e^{-ik\theta} d\theta$ 
for each $k\geq0$.
\end{theorem}

For a disk $D$ centered at $a_0\in\CC$ with radius $\gamma_D$, it holds that 
$\NN_{11}^{(2)}(D,\lambda)  =  2\pi \gamma_D^2/\lambda$ and $\NN_{21}^{(2)}(D,\lambda) = 2\overline{a_0}\NN_{11}^{(2)}(D,\lambda)$.
In the numerical simulation in section 6, we first find a disk $D$ satisfying 
$$\NN_{11}^{(2)}(\Om,\lambda) = \NN_{11}^{(2)}(D,\lambda)$$
and $$\NN_{21}^{(2)}(\Om,\lambda) =\NN_{21}^{(2)}(D,\lambda).$$
In other words, we set $D$ to be a disk centered at $a_0$ with radius $\gamma_D$ with
\begin{align}
\gamma_D^2 = \frac{\lambda \, \NN_{11}^{(2)}(\Om,\lambda)}{2\pi}, \quad a_0 = \frac{\NN_{12}^{(2)}(\Om,\lambda)}{2\NN_{11}^{(2)}(\Om,\lambda)}.\label{radius_diskperturb}
\end{align}
Here, we use symmetricity \eqnref{NNproperty} to obtain the second relation.  We then apply the formula \eqnref{firstdisk} to obtain the Fourier coefficient of $f$; \eqnref{seconddisk} works given that $|\sigma-\gamma_D^2|$ is not small. 
\begin{remark}\label{remark:PDR}
In all examples in section 6, the center $a_0$ is zero. 
 In general, we can apply Theorem \ref{dd} after shifting the coordinate frame by $a_0$ (see \cite{Ammari:2014:GPT} for the dependence of GPTs on the coordinate frame translation). 
\end{remark}

\section{Faber polynomial Polarization Tensors (FPTs)}\label{sec:matrix}

As stated earlier, $\Om$ is assumed to be a simply connected planar domain.
According to the Riemann mapping theorem, there uniquely exist a positive number $\gamma$ and a conformal mapping $\Psi$ from $\left\{w\in\CC:|w|>\gamma\right\}$ onto $\CC\setminus\overline{\Om}$ satisfying $\Psi(\infty)=\infty$ and $\Psi'(\infty)=1$. This mapping admits a Laurent series expansion
\beq\label{conformal:Psi}
\Psi(w)=w+a_0+\frac{a_1}{w}+\frac{a_2}{w^2}+\cdots
\eeq 
with some complex coefficients $a_0,a_1,a_2,\dots$; see \cite[Chapter 1.2]{Pommerenke:1992:BBC} for the derivation.
Note that $\p\Om$ can be specified by the image of $\{z\in\CC:|w|=\gamma\}$ under $\Psi$.

\subsection{Faber polynomials and Grunsky coefficients}
The exterior conformal mapping $\Psi$ defines a sequence of the so-called Faber polynomials $\{F_m\}_{m=1}^\infty$ via the relation (see \cite{Duren:1983:UF, Faber:1903:UPE})
\begin{align}\notag
\frac{w\Psi'(w)}{\Psi(w)-z}=\sum_{m=0}^\infty \frac{F_m(z)}{w^{m}},\quad z\in{\overline{\Om}},\ |w|>\gamma.
\end{align}
For each $m$, $F_m$ is a monic polynomial of degree $m$ uniquely determined by $a_0,a_1,\dots,a_{m-1}$ via the recursive relation
\beq\label{Faberrecursion}
F_{m+1} (z) = z F_m (z) - m a_m - \sum_{k=0} ^{m} a_k F_{m-k} (z), \quad m\ge 0.
\eeq
For example, the first three Faber polynomials are
\begin{align}
F_0(z)&=1,\quad F_1(z)=z-a_0,\quad F_2(z)=z^2-2a_0 z+a_0^2-2a_1.\label{Faber:ex1}
%\\
%F_3(z)& = z^3 - 3a_0z^2 + 3(a_0^2-a_1)z - a_0^3+3a_0a_1-3a_2.\label{Faber:ex2}
\end{align}

%We set $\gamma,\Psi(w),F_m(z)$ as in section \ref{sec:Faber}. 
An essential property of the Faber polynomial is that $F_m(\Psi(w))$ has a single positive-order term $w^m$. In other words, it satisfies
\begin{equation*}
F_m(\Psi(w)) = w^m+\sum_{k=1}^{\infty}c_{mk}{w^{-k}},\quad |w|>\gamma.
\end{equation*}
The quantities $c_{mk}$ are called the Grunsky coefficients
and can be uniquely determined by $\Psi$ via the recursive relation
\begin{align*}
\ds &c_{1m} = a_m,\quad c_{m1} = ma_m,\\
\ds &c_{m,k+1} = c_{m+1,k} - a_{m+k} + \sum_{s=1}^{m-1} a_{m-s}c_{sk} - \sum_{s=1}^{k-1} a_{k-s}c_{ms},\quad m,k\geq1.
\end{align*} 
For example, we have
\begin{align*}
&c_{11}=a_1,\quad c_{12}=a_2,\quad c_{13}=a_3,\\
& c_{21}=2a_1,\quad c_{22} = 2a_3+a_1^2, \quad c_{23} = 2a_4+2a_1a_2,\\
& c_{31} = 3a_1,\quad c_{32} = 3a_4+3a_1a_2,\quad c_{33} = 3a_5+3a_1a_3+3a_2^2+a_1^3.
\end{align*}
The Grunsky coefficients satisfy
\beq\label{eqn:Grunsky:sym1}
kc_{mk} =mc_{km}\quad\mbox{for all }m,k\in\NN.
\eeq
It also holds that
\begin{equation}\label{ineq:Grunsky}
\sum_{k=1}^\infty \left| \sum_{m=1}^\infty \sqrt{\frac{k}{m}} \frac{c_{mk}}{\gamma^{m+k}} x_m \right|^2 \leq \sum_{m=1}^\infty |x_m|^2
\end{equation}
for all complex sequences $(x_m)$, where the equality holds if and only if $\Omega$ has a measure of zero.
See \cite{Duren:1983:UF} for the derivation and further details. 
We can symmetrize the Grunsky coefficients as 
 $$g_{mk} = \sqrt{\frac{k}{m}} \frac{c_{mk}}{\gamma^{m+k}}.$$
From \eqnref{eqn:Grunsky:sym1}, it holds that
\beq\label{eqn:g:sym}
g_{mk}=g_{km}\quad\mbox{for all }m,k \in\NN.
\eeq

\subsection{Series expansion of the layer potential operators} \label{subsec:series}
Set $\rho_0=\ln \gamma$.
We can define a curvilinear orthogonal coordinate system $(\rho,\theta)\in [\rho_0, \infty)\times[0,2\pi)$ on the exterior of $\Om$ via the relation
\beq\notag
z=\Psi(e^{\rho+i\theta})\quad\mbox{for }z\in\CC\setminus\Om.
\eeq 
We denote the scale factor as 
$
h(\rho,\theta)=\left|\frac{\partial \Psi}{\partial\rho}\right|=\left|\frac{\partial \Psi}{\partial\theta}\right|.$
The length element on $\p\Om$ is $d\sigma(z)=h(\rho_0,\theta)d\theta$. 
For a function $g(z)=(g\circ\Psi)(e^{\rho+i\theta})$, it holds that
\beq
\frac{\partial g}{\partial \nu}\Big|_{\p \Om}^+(z)=\frac{1}{h(\rho_0,\theta)}\frac{\partial }{\partial \rho}g(\Psi(e^{\rho+i\theta}))\Big|_{\rho\rightarrow\rho_0^+}.\label{eqn:normalderiv}
\eeq	
We also define density basis functions on $\p\Om$ as
\begin{align*}
&\eta_m(z) = |m|^{-\frac{1}{2}}e^{im\theta},\\
&\zeta_m(z)= |m|^{\frac{1}{2}}\frac{e^{im\theta}}{h(\rho_0,\theta)} \quad\mbox{for }m\in\NN.
\end{align*}
For $m=0$, we set $\eta_0(z)=1$ and $\zeta_0(z)=\frac{1}{h(\rho_0,\theta)}$.
\begin{lemma}[\cite{Jung:2018:NSS}]\label{lemma:seriesexpan}
Let $\Om$ be a simply connected planar domain with $\mathcal{C}^2$-boundary. 
For each $m\in\NN$, the layer potential operators associated with $\Om$ satisfy 
\begin{align*}
\Scal_{\p\Om}[\zeta_m](z)&=
\begin{dcases}
-\frac{1}{2\sqrt{m}\gamma^m} F_m(z), \quad &z\in \overline{\Omega},\\
-\frac{1}{2\sqrt{m}\gamma^m} \left( \sum_{k=1}^\infty c_{mk} e^{-k(\rho +i\theta)} + \gamma^{2m}e^{-\rho+i\theta} \right), \quad &z\in \CC\setminus\overline{\Omega}.
\end{dcases}\\[2mm]
\Dcal_{\p\Om}[\eta_m](z)&=
\begin{dcases}
\frac{1}{2\sqrt{m}\gamma^m} F_m(z), \quad &z\in \overline{\Omega},\\
\frac{1}{2\sqrt{m}\gamma^m} \left( \sum_{k=1}^\infty c_{mk} e^{-k(\rho +i\theta)} - \gamma^{2m}e^{-\rho+i\theta} \right), \quad &z\in \CC\setminus\overline{\Omega}.
\end{dcases}
\end{align*}
Furthermore, $\mathcal{K}_{\p\Om}$ and $\mathcal{K}^*_{\p\Om}$ admit the series expansion:
\begin{align*}
&&\Kcal^*_{\p\Om}\left[\zeta_0\right]&=\frac{1}{2}\zeta_0, &\KstarOmega[\zeta_m]&=\frac{1}{2}\sum_{k=1}^{\infty}\sqrt{\frac{k}{m}}\frac{c_{mk}}{\gamma^{m+k}}\, \overline{{\zeta}_{k}},&\\
&&\Kcal_{\p\Om}[1]&=\frac{1}{2}, &\KOmega [\eta_m]&=\frac{1}{2}\sum_{k=1}^{\infty} \sqrt{\frac{k}{m}}\frac{c_{mk}}{\gamma^{m+k}}\, \overline{{\eta}_{k}}.&
\end{align*}
\end{lemma}

Using the jump relations of the layer potential operators and Lemma \ref{lemma:seriesexpan}, one can easily find that 
the solutions $u(z)$ and $v(z)$ to \eqnref{cond_eqn1} and \eqnref{cond_eqn2} with $H(z)=F_m(z)$ and $F(z)=F_n(z)$ satisfy
\begin{align}
\ds u(z) &= \sqrt{m}\gamma^m (1-2\lambda) \Scal_{\p \Om} \left[(\lambda I - \Kcal_{\p \Om} ^*)^{-1} [\zeta_{m}]\right](z), \quad z\in\Om, \label{Su} \\[2mm] 
\ds v(z) &= \sqrt{n}\gamma^n (1+2\lambda) \Dcal_{\p \Om} \big[(\lambda I - \Kcal_{\p \Om})^{-1} [\eta_{n}]\big](z), \quad z\in\Om.\label{Sv}
\end{align}

\subsection{Definition and properties of FPTs}\label{sec:Faber}

The Faber polynomials corresponding to $\Om$ form a basis for analytic functions in a region containing $\Om$ \cite{Duren:1983:UF}. We note that $z^m$ are the Faber polynomials corresponding to a disk centered at the origin. 
Following \cite{Choi:2018:GME}, we define the Faber polynomial polarization tensors by replacing $z^n$ with $F_n$ in \eqnref{def:GPT1} and \eqnref{def:GPT2}. 
\begin{definition}[FPTs]
For $m,n\in\NN$, we define
\begin{align*}
\FF_{mn}^{(1)}(\Om, \lambda)&=\int_{\p\Om}F_n(z)\left(\lambda I-\Kcal^*_{\p\Om}\right)^{-1}\left[\pd{F_m}{\nu}\right](z)\,d\sigma(z),\\
\FF_{mn}^{(2)}(\Om, \lambda)&=\int_{\p\Om}F_n(z)\left(\lambda I-\Kcal^*_{\p\Om}\right)^{-1}\left[\pd{\overline{F_m}}{\nu}\right](z)\,d\sigma(z).
\end{align*}
If not specified otherwise, we set $\lambda=\frac{\sigma+1}{2(\sigma-1)}$. 
\end{definition}

The background potential $H$ is real-valued and harmonic so that it satisfies
\beq\label{expanF:H}
H(z) = \sum_{m=1}^\infty \left(\beta_m F_m(z)+\overline{\beta_m F_m(z)}\right)
\eeq
for some complex coefficients $\beta_m$. The solution $u$ corresponding to $H$ then admits an expansion in terms of FPTs and $\Psi$ (see \cite{Choi:2018:GME} for the derivation): 
\begin{align}\label{eqn:FPT_expan}
u (z)-H(z)= &-\sum_{n=1}^\infty \sum_{m=1}^\infty\frac{1}{4\pi n}\Big( \beta_m \FF_{mn}^{(1)}+\overline{\beta_m}\, \FF_{mn}^{(2)}\Big){w^{-n}}\notag\\
&-\sum_{n=1}^\infty \sum_{m=1}^\infty\frac{1}{4\pi n}\Big( \overline{\beta_m}\, \overline{\FF_{mn}^{(1)}}  + \beta_m  \overline{\FF_{mn}^{(2)}} \Big)\overline{w^{-n}},\quad z=\Psi(w)\in\CC\setminus\overline{\Om}.
\end{align}
We call it the geometric multipole expansion of $u$. 
We highlight that it holds in the entire exterior region $\CC\setminus\overline{\Om}$ while the classical multipole expansion \eqnref{CP2} holds for sufficiently large $z$.

Recall that for each $m\in\NN$, $F_m$ is uniquely determined by $\{ a_j\}_{0\le j \le m-1}$. 
From the definition of GPTs and FPTs, we have the following:
\begin{lemma}\label{lemma:fabercoeffi}
For each $m\in\NN$, the Faber polynomial $F_m(z)$ can be written as
$
F_m(z) = \sum_{n=0}^{m} p_{mn} z^n
$
with some coefficients $p_{mn}$ depending only on $\{ a_j\}_{0\le j \le m-1}$. 
For all $m,k\in\NN$, we have \begin{align*}
\ds\ \FF_{mk}^{(1)} = \sum_{l=1} ^{k} \sum_{n=1} ^{m} p_{kl} \, p_{mn} \, \mathbb{N}_{nl}^{(1)}, 
\qquad \FF_{mk}^{(2)} = \sum_{l=1} ^{k} \sum_{n=1} ^{m} p_{kl} \, \overline{p_{mn}}\, \mathbb{N}_{nl}^{(2)}.
\end{align*}
\end{lemma}
From this lemma with $m,n=1,2$, we obtain
\begin{align}
& \FF_{11}^{(1)}= \mathbb{N}_{11} ^{(1)},\quad  \FF_{11}^{(2)}=  \mathbb{N}_{11} ^{(2)}, \quad \FF_{21}^{(1)}= \mathbb{N}_{21} ^{(1)}- 2a_0 \mathbb{N}_{11} ^{(1)}, \quad \FF_{21}^{(2)}=  \mathbb{N}_{21} ^{(2)} - 2 \overline{a_0} \mathbb{N}_{11} ^{(2)},\label{FPTex1}\\
&\FF_{22}^{(1)} = \mathbb{N}_{22}^{(1)} - 4a_0 \mathbb{N}_{21}^{(1)} + 4a_0^2 \mathbb{N}_{11}^{(1)}, \quad
\FF_{22}^{(2)} = \mathbb{N}_{22}^{(2)} - 4 \text{Re}\left( a_0 \mathbb{N}_{21} ^{(2)} \right) + 4|a_0|^2 \mathbb{N}_{11}^{(2)}.\label{FPTex2}
\end{align}

\subsection{Matrix formulation for the transmission problem}\label{sec:coord}
We consider a semi-infinite matrix given by the Grunsky coefficients,
\begin{equation*}
C:=\big[c_{mk}\big]_{m,k=1}^\infty
=
\begin{bmatrix}
\ c_{11} & c_{12} & c_{13}& \cdots\\[2mm]
\ c_{21} & c_{22} & c_{23} & \cdots\\[2mm]
\ c_{31} & c_{32} & c_{33} & \cdots\\
\ \vdots & \vdots & \vdots & \ddots
\end{bmatrix},
\end{equation*}
and its symmetrization,
\begin{equation*}
G := \big[ g_{mk} \big]_{m,k=1}^\infty
=\begin{bmatrix}
\ g_{11} & g_{12} & g_{13}& \cdots\\[2mm]
\ g_{21} & g_{22} & g_{23} & \cdots\\[2mm]
\ g_{31} & g_{32} & g_{33} & \cdots\\
\ \vdots & \vdots & \vdots & \ddots
\end{bmatrix}.
\end{equation*}
We denote by $I$ the identity matrix and define
\begin{equation}\label{eqn:matrix}
\gamma^{\pm\NN}:=
\begin{bmatrix}
\  \gamma^{\pm1} & 0 & 0 &\cdots  \\[2mm]
\ 0 & \gamma^{\pm2} & 0 & \cdots\\[2mm]
\ 0 & 0 & \gamma^{\pm3} & \cdots \\
\ \vdots & \vdots &\vdots  & \ddots 
\end{bmatrix},
\qquad
\NN^{\pm\frac{1}{2}}:=
\begin{bmatrix}
\ 1  & 0 & 0 &\cdots \\[2mm]
\ 0 & {\sqrt{2}}^{\, \pm 1}  & 0 & \cdots\\[2mm]
\ 0 & 0 & {\sqrt{3}}^{\, \pm 1} & \cdots\\
\ \vdots & \vdots &\vdots  & \ddots 
\end{bmatrix}.
\end{equation}
Similarly, the matrix $\gamma^{\pm 2\NN}$ denotes the diagonal matrix whose $(n,n)$-entries are $\gamma^{\pm 2n}$.

From equation \eqnref{eqn:g:sym} and the definition of $g_{mk}$, $G$ is symmetric and satisfies
\beq\label{GGG}
G= \NN^{-\frac{1}{2}}\gamma^{-\mathbb{N}} C\gamma^{-\mathbb{N}}\NN^{\frac{1}{2}}.
\eeq
We can interpret the matrix $G$ as a linear operator from $l^2(\CC)$ to $l^2(\CC)$ given by
\begin{align}\label{G:l2}
    (x_m) &\longmapsto (y_m)\quad\mbox{with}\quad y_m = \sum_{k=1}^\infty g_{mk}x_k.
\end{align}
Here, $l^2(\CC)$ is the vector space consisting of all complex sequences $(x_m)$ satisfying 
$\sum_{m=1}^\infty |x_m|^2<\infty$. 
The inequality \eqnref{ineq:Grunsky} and the symmetricity of $G$ imply
\begin{align*}
\left\lVert G \right\rVert^2
& = \sup_{\left\lVert (x_k) \right\rVert=1} \sum_{m=1}^\infty \left| \sum_{k=1}^\infty g_{mk} x_k \right|^2 = \sup_{\left\lVert (x_m) \right\rVert=1} \sum_{k=1}^\infty \left| \sum_{m=1}^\infty g_{mk} x_m \right|^2\leq 1.
\end{align*} 
In fact, the above inequality is strict. 
A quasiconformal curve (or quasicircle) is the image of the unit circle under a quasi-conformal mapping of the complex plane. It is known, by Ahlfors \cite{Ahlfors:1963:QR} in 1963, that 
a closed Jordan curve $\Gamma\subset \CC$ is a quasiconformal curve if and only if there exists a finite number $K$ such that
\beq\notag
\min\left(\text{diam}(\Gamma_1),\, \text{diam}(\Gamma_2)\right) \le K|z_1-z_2|\quad\mbox{for any } z_1, z_2 \in \Gamma,
\eeq
where $\Gamma_1$ and $\Gamma_2$ are two arcs of which $\Gamma\setminus \{z_1,z_2\}$ consists. The symbol $\text{diam}(L)$ stands for the diameter of a curve $L$, {\it i.e.}, 
$\text{diam}(L)= \sup\left\{ |z'-z''| : z',z''\in L \right\}.$
Hence, a Lipschitz curve is quasiconformal.
According to \cite[Theorems 9.12-13]{Pommerenke:1975:UF}, it holds that $\left\lVert G \right\rVert_{l^2\rightarrow l^2} \le \kappa$ for some $\kappa\in[0,1)$ if and only if $\p \Om$ is quasiconformal; we recommend \cite{Ahlfors:1963:QR, Beckermann:2018:BOP, Kuehnau:1971:VKG, Springer:1964:FEQ} for more properties of quasiconformality and the Grunsky coefficients.
 Therefore, we have the following lemma.
 \begin{lemma}\label{quasiconformal}
The operator $G$ given by \eqnref{G:l2} satisfies
$$\left\lVert G \right\rVert < 1.$$
\end{lemma}
The operator $4\lambda^2 I - {\gamma^{-2\mathbb{N}} \overline{C}\gamma^{-2\mathbb{N}} C}$ is related to the transmission problem \eqnref{cond_eqn0} as shown in Lemma \ref{invseries}. The following lemma shows its invertibility.
\begin{lemma}\label{thm:nearextreme}
For $|\lambda|\geq\frac{1}{2}$, the semi-infinite matrix $4\lambda^2 I - {\gamma^{-2\mathbb{N}} \overline{C}\gamma^{-2\mathbb{N}} C}$ is invertible in the sense of a linear operator on $l^2(\CC)$ and each entry of its inverse is bounded independently of $\lambda$. 
\end{lemma}

\pf
It is straightforward to see from \eqnref{GGG} that 
\beq\label{matrix:decom:relation}
4\lambda^2 I - \gamma^{-2\mathbb{N}} \overline{C}\gamma^{-2\mathbb{N}} C=\NN^{\frac{1}{2}}\gamma^{-\NN}\left(4\lambda^2 I - \overline{G}G\right)\gamma^{\NN}\NN^{-\frac{1}{2}}.
\eeq
The diagonal matrices $\gamma^{\pm\NN}$ and $\NN^{\pm\frac{1}{2}}$ are invertible. 
From Lemma \ref{quasiconformal}, the operator $4\lambda^2I - \overline{G}G$ is invertible
and satisfies
$$
\left\lVert \left( 4\lambda^2I - \overline{G}G  \right)^{-1} \right\rVert
 = \left\lVert \frac{1}{4\lambda^2} \sum_{n=0}^\infty \left(\frac{1}{4\lambda^2}\overline{G}G\right)^n\right\rVert
 \leq \left(4\lambda^2-\|G\|^2\right)^{-1}\leq \left(1-\|G \|^{2}\right)^{-1}.$$
Since each entry of a matrix is bounded by the operator norm of the matrix, we have the uniform boundedness
$$\left|\left(\left( 4\lambda^2I - \overline{G}G  \right)^{-1} \right)_{mk}\right|\leq \left(1-\|G \|^{2}\right)^{-1}\quad\mbox{for all }m,k.$$
Hence, we have from \eqnref{matrix:decom:relation} that
$$
\left|\left(\left( 4\lambda^2 I - \gamma^{-2\mathbb{N}} \overline{C}\gamma^{-2\mathbb{N}} C \right)^{-1} \right)_{mk}\right| \leq   \sqrt{\frac{m}{k}}\gamma^{k-m}\left(1-\|G \|^{2}\right)^{-1}.
$$
Note that the right side is independent of $\lambda$. 
\qed

\begin{lemma}[\cite{Choi:2018:GME}]\label{invseries}
For each $m$, it holds that
\begin{align*}
\left(\lambda I-\Kcal^*_{\p \Om}\right)^{-1}[\zeta_m]
&=\frac{1}{2}\sum_{k=1}^\infty \sqrt{\frac{k}{m}} \frac{1}{\gamma^{m+k}}\left( a_{mk}\zeta_k+b_{mk}\overline{\zeta_{k}}\right) ,\\
\left(\lambda I-\Kcal_{\p \Om}\right)^{-1}[\eta_m]
&=\frac{1}{2}\sum_{k=1}^\infty \sqrt{\frac{k}{m}}   \frac{1}{\gamma^{m+k}} \left( a_{mk} \eta_k + b_{mk} \overline{\eta_{k}}\right)  ,
\end{align*}
where $a_{mk}$ and $b_{mk}$ are given by
\begin{align*}
A &=[a_{mk}]_{m,k=1}^\infty=8\lambda \gamma^{2\mathbb{N}} \left( 4 \lambda^2 I - \gamma^{-2\mathbb{N}} C \gamma^{-2\mathbb{N}} \overline{C} \right)^{-1},\\
 B& =[b_{mk}]_{m,k=1}^\infty= 4C\left( 4\lambda^2 I - \gamma^{-2\mathbb{N}} \overline{C}\gamma^{-2\mathbb{N}} C \right)^{-1}.
\end{align*}
\end{lemma}  
We can express FPTs in matrix form as follows:
\begin{theorem}[\cite{Choi:2018:GME}]\label{thm:FPT}
For each $m,k$, it holds that
\begin{align*}
\FF_{mk}^{(1)} (\Om, \lambda)
&=  4\pi k c_{mk} + 4\pi k \left( 1 - 4\lambda^2 \right) \left(C\left( 4\lambda^2 I - \gamma^{-2\mathbb{N}} \overline{C}\gamma^{-2\mathbb{N}} C \right)^{-1} \right)_{mk},\\
\FF_{mk}^{(2)}(\Om, \lambda)
&= 8\pi k\lambda \gamma^{2m}\delta_{mk} + 8 \pi k \lambda \gamma^{2m} \left( 1 - 4\lambda^2 \right) \left( \left( 4\lambda^2 I - \gamma^{-2\mathbb{N}} \overline{C}\gamma^{-2\mathbb{N}} C \right)^{-1} \right)_{mk}.
\end{align*}
Here, $\delta_{mk}$ is the Kronecker delta function.
\end{theorem}

For an ellipse given by
 \beq\label{Psi:E}
 \Psi(w)=w+a_0+\frac{a_1}{w},
 \eeq
it holds that (see, for example, \cite{Jung:2018:NSS} for the derivation)
\beq\label{Faber:ellipse}
F_{m}(z) = \left( {\tilde{z}}  + \sqrt{{\tilde{z}}^2 - a_1}\right)^m + \left({\tilde{z}} - \sqrt{{\tilde{z}}^2- a_1} \right)^m\quad\mbox{with }{\tilde{z}} = \frac{z-a_0}{2}
\eeq
and 
\beq\label{FPsi:Ellipse}
F_m\left(\Psi(w)\right)=w^m +\frac{a_1^m}{w^m}.
\eeq
Hence, the Grunsky coefficients are
$$c_{mk} = \delta_{mk} a_1^m$$
and we obtain the following corollary from Theorem \ref{thm:FPT}.
\begin{cor}\label{cor:FPT}
The FPTs for an ellipse, namely $E$, given by \eqnref{Psi:E} are diagonal matrices with
\begin{align*}
\FF_{mm}^{(1)} (E, \lambda)
&= 4\pi m a_1^m \left( 4\lambda^2 \gamma^{4m} - |a_1|^{2m} \right)^{-1} \left( \gamma^{4m} - |a_1|^{2m} \right),\\
\FF_{mm}^{(2)}(E, \lambda)
&= 8 \pi m \lambda \gamma^{2m} \left( 4\lambda^2 \gamma^{4m} - |a_1|^{2m} \right)^{-1} \left( \gamma^{4m} - |a_1|^{2m} \right).
\end{align*}
\end{cor}

\section{Explicit reconstruction formula for the conformal mapping}\label{sec:ERF}
In this section we derive an inversion formula for the conformal mapping corresponding to an inclusion with extreme or near-extreme conductivity by using Theorem \ref{thm:FPT}.

We first derive a formula for the conformal mapping corresponding to an inclusion with extreme conductivity:
\begin{theorem}\label{thm:exact}
Let $\Om$ be a simply connected planar domain with $\mathcal{C}^2$-boundary. Assume that $\Om$ is either insulating or perfectly conducting, {\it i.e.}, $\lambda = \pm\frac{1}{2}$.
Then the coefficients of the exterior conformal mapping corresponding to $\Om$ satisfy
\begin{align*}
&\gamma^2 = \frac{\mathbb{N}_{11}^{(2)}(\Om,\lambda)}{8\pi \lambda}, \quad a_0 = \frac{\mathbb{N}_{12}^{(2)} (\Om,\lambda ) }{2\mathbb{N}_{11} ^{(2)} (\Om,\lambda) }, \\
& a_m =\frac{1}{4\pi m} \sum_{n=1} ^{m} p_{mn}\, {\mathbb{N}_{n1} ^{(1)} (\Om,\lambda)}\quad\mbox{for } m\ge 1.
\end{align*}
Here, $p_{m1}, p_{m2}, \cdots, p_{mm}$ denote the coefficients of the Faber polynomials as in Lemma \ref{lemma:fabercoeffi}. 
Each $a_m$ is uniquely determined by $\mathbb{N}_{11}^{(2)}$, $\mathbb{N}_{12}^{(2)}$ and $\{\mathbb{N}_{n1}^{(1)}\}_{1\le n\le m}$. 
\end{theorem}
\pf
From Lemma \ref{lemma:fabercoeffi} with $k=1$ and \eqnref{FPTex1}, it holds that
\begin{align}
&\FF_{m1}^{(1)}\left(\Om, \lambda \right) = \sum_{n=1} ^{m} p_{mn}\, {\mathbb{N}_{n1} ^{(1)} (\Om,\lambda)} \quad \mbox{for }m\geq 1,\notag \\
&\FF_{11}^{(2)}\left(\Om,\lambda \right) = \mathbb{N}_{11}^{(2)}\left(\Om,\lambda \right),\notag \\
&\mathbb{F}_{21}^{(2)} \left(\Om, \lambda \right) = \mathbb{N}_{21}^{(2)} \left(\Om, \lambda \right) - 2\overline{a_0}\mathbb{N}_{11}^{(2)} \left(\Om, \lambda \right). \label{cases:eqn:thm:exact}
\end{align}
Since $1-4\lambda^2=0$ from the assumption, Theorem \ref{thm:FPT} further implies that
\begin{align*}
\FF_{mk}^{(1)} (\Om, \lambda)
&=  4\pi k c_{mk} ,\\
\FF_{mk}^{(2)}(\Om, \lambda)
&= 8\pi k\lambda \gamma^{2m}\delta_{mk}\quad\mbox{for any }m,k\in\NN.
\end{align*}
In particular, it holds that
\begin{align*}
&\FF_{m1}^{(1)}\left(\Om, \lambda \right) = 4\pi c_{m1} \quad \mbox{for }m\geq 1,\\
&\FF_{11}^{(2)}\left(\Om,\lambda \right) = 8\pi \lambda \gamma^2, \quad \mathbb{F}_{21}^{(2)} \left(\Om, \lambda \right) = 0.
\end{align*}
By using Lemma \ref{NNproperty} and \eqnref{cases:eqn:thm:exact}, we then obtain
\begin{align*}
&\gamma^2 =\frac{\FF_{11}^{(2)}\left(\Om,\lambda \right) }{8\pi \lambda} =\frac{\mathbb{N}_{11}^{(2)}\left(\Om,\lambda \right)}{8\pi \lambda},\\[1mm]
&a_0 = \dfrac{\overline{\mathbb{N}_{21} ^{(2)} \left(\Om, \lambda \right) - \mathbb{F}_{21} ^{(2)} \left(\Om, \lambda \right)} }{2 \mathbb{N}_{11} ^{(2)} \left(\Om, \lambda \right) } = \dfrac{\overline{\mathbb{N}_{21} ^{(2)} \left(\Om, \lambda \right)} }{2 \mathbb{N}_{11} ^{(2)} \left(\Om, \lambda \right) } = \dfrac{\mathbb{N}_{12} ^{(2)} \left(\Om,\lambda \right)}{2 \mathbb{N}_{11} ^{(2)} \left(\Om,\lambda\right) },\\[1mm]
&a_m = \frac{c_{m1}}{m} = \frac{\FF_{m1}^{(1)}\left(\Om, \lambda \right)}{4\pi m} = \frac{1}{4\pi m} \sum_{n=1} ^{m} p_{mn}\, {\mathbb{N}_{n1} ^{(1)} (\Om,\lambda)} \quad \mbox{for }m\geq 1.
\end{align*}
Note that $N_{11}^{(2)}$ is nonzero since $\gamma$ is positive. 
Since $p_{11}=1$, $a_0$ and $a_1$ are determined by $\mathbb{N}_{11}^{(2)}$, $\mathbb{N}_{12}^{(2)}$ and $\mathbb{N}_{11}^{(1)}$.
 For $a_m$ with $m\geq 2$, $\{p_{m1}, p_{m2}, \cdots, p_{mm} \}$ is uniquely determined by $\{ a_0, a_1, \cdots , a_{m-1}\}$. 
By induction, one can obtain $a_m$ from $\mathbb{N}_{11}^{(2)}$, $\mathbb{N}_{12}^{(2)}$ and $\{\mathbb{N}_{n1}^{(1)}\}_{1\le n\le m}$.
 \qed
\begin{remark}
For an insulating inclusion, a recursive relation similar to Theorem \ref{thm:exact} was previously obtained in \cite{Choi:2018:CEP}. In this study, using FPTs, we derive an exact relation that holds for the perfectly conducting case as well as the perfectly insulating case. We also emphasize that the formula in Theorem \ref{thm:exact} is much simpler than that in \cite{Choi:2018:CEP}.
\end{remark}

\smallskip

We now validate that this generalized formula approximately holds also for the near-extreme conductivity case as follows. 
\begin{theorem}[Conformal mapping recovery]\label{rs}
The coefficients of the conformal mapping associated with $\Om$ satisfy
\begin{align*}
\gamma^2 &= \frac{\mathbb{N}_{11}^{(2)}(\Om,\lambda)}{8\pi \lambda}  + O\left(|\lambda|-\frac{1}{2}\right), \\[1mm]
a_0 &= \frac{\mathbb{N}_{12}^{(2)} (\Om,\lambda ) }{2\mathbb{N}_{11} ^{(2)} (\Om,\lambda) } + O\left(|\lambda|-\frac{1}{2}\right), \\[1mm]
a_m &=\frac{1}{4\pi m} \sum_{n=1} ^{m} p_{mn}\, \mathbb{N}_{n1} ^{(1)} (\Om,\lambda) + O\left(|\lambda|-\frac{1}{2}\right)\quad\mbox{for } m\ge 1.
\end{align*}
\end{theorem}
\begin{proof}
From Theorem \ref{thm:FPT}, we have
\begin{align*}
&\FF_{11}^{(2)}\left(\Om,\lambda \right) = 8\pi \lambda \gamma^2 + 8\pi \lambda \gamma^2 \left( \frac{1}{4}-\lambda^2 \right) \left(\left( \lambda^2 I - \frac{\gamma^{-2\mathbb{N}} \overline{C}\gamma^{-2\mathbb{N}} C}{4} \right)^{-1} \right)_{11}  , \\
&\mathbb{F}_{21}^{(2)} \left(\Om, \lambda \right) = 8\pi \lambda \gamma^4\left( \frac{1}{4}-\lambda^2 \right) \left(\left( \lambda^2 I - \frac{\gamma^{-2\mathbb{N}} \overline{C}\gamma^{-2\mathbb{N}} C}{4} \right)^{-1} \right)_{21} , \\
&\FF_{m1}^{(1)}\left(\Om, \lambda \right) = 4\pi c_{m1} + 4\pi \left( \frac{1}{4}-\lambda^2 \right) \left(C\left( \lambda^2 I - \frac{\gamma^{-2\mathbb{N}} \overline{C}\gamma^{-2\mathbb{N}} C}{4} \right)^{-1} \right)_{m1} \quad \mbox{for }m\geq 1.
\end{align*}
We complete the proof by applying this relation to \eqnref{cases:eqn:thm:exact} and by using Lemma \ref{thm:nearextreme}.
\end{proof}

%%%%%%%%%%%%%%%%%%%%%%%%%%%%%%%%%%%%%%%%%%%%%%%%%

\section{Recovering the shape of an inclusion with arbitrary conductivity}\label{arbtcond}
In this section we derive an explicit formula that approximates the shape of an inclusion with arbitrary conductivity by using the concept of FPTs. We regard an inclusion as a perturbation of its equivalent ellipse and apply the asymptotic integral expression for the shape derivative of GPTs obtained in \cite{Ammari:2012:GPT}. As a result, we obtain an analytic reconstruction formula for the shape perturbation function. With this formula, one can non-iteratively approximate an inclusion of arbitrary conductivity with general shape, including a straight or asymmetric shape.

\subsection{Equivalent ellipse and modified GPTs}\label{sec:equiellipse}
For an inclusion of a general shape, we can find an equivalent ellipse that has the same first-order GPTs as shown in \cite{Ammari:2005:RCS, Bruehl:2003:DIT}. The equivalent ellipse provides a good initial guess for the optimization procedure \cite{Ammari:2012:GPT}. 
In the following, we derive an expression formula for an equivalent ellipse by employing the concept of FPTs.

An ellipse, say $E$, has the exterior conformal mapping
\beq\label{eeconformal}
\Psi_E(w) =  w + e_0 +\frac{e_1}{w}\quad \mbox{for }|w|>\gamma_e
\eeq
with some complex coefficients, $e_0$, $e_1$, and $\gamma>0$. 
From Corollary \ref{cor:FPT} and \eqnref{FPTex1}, we have
\begin{align*}
&\mathbb{N}_{11}^{(1)}(E, \lambda) = \FF_{11}^{(1)} (E, \lambda)=  4\pi e_1\left( 4\lambda^2 - \frac{ |e_1|^2}{\gamma_e^{4}} \right)^{-1}\left( 1 - \frac{ |e_1|^2}{\gamma_e^{4}} \right),\\
&\mathbb{N}_{11}^{(2)}(E, \lambda) = \FF_{11}^{(1)} (E, \lambda)= 8\pi \lambda \gamma_e^{2}\left( 4\lambda^2 - \frac{ |e_1|^2}{\gamma_e^{4}} \right)^{-1}\left( 1 - \frac{ |e_1|^2}{\gamma_e^{4}} \right).
\end{align*}
Also, 
we have
\beq\label{eqn:FF:E}
\FF_{m1}^{(1)} (E,\lambda) =\FF_{m1}^{(2)} (E,\lambda) = 0\quad\mbox{for }m\geq2
\eeq
and
\beq\notag
\mathbb{N}_{21} ^{(2)}(E,\lambda) - 2 \overline{e_0} \mathbb{N}_{11} ^{(2)}(E,\lambda) = \FF_{21}^{(2)}(E, \lambda) =0.
\eeq
From these relations and the symmetric properties of GPTs in \eqnref{eqn:GPT:symmetric}, we obtain the following lemma. 

\begin{lemma}[Equivalent ellipse]\label{ee} 
Let $E$ be the ellipse whose exterior conformal mapping is given by \eqnref{eeconformal} with
\begin{align}\label{def:ee}
\gamma_e^2 = \frac{\lambda \, \NN_{11}^{(2)}}{2\pi}   \frac{\bigl|\NN_{11}^{(2)}\bigr|^2 - \bigl|\NN_{11}^{(1)}\bigr|^2}{\bigl|\NN_{11}^{(2)}\bigr|^2 -4\lambda^2 \bigl|\NN_{11}^{(1)}\bigr|^2},\quad
e_0 = \frac{\mathbb{N}_{12}^{(2)}}{2 \mathbb{N}_{11} ^{(2)}},\quad
e_1 =2\lambda\gamma_e^2\,   \frac{ \NN_{11}^{(1)}}{ \NN_{11}^{(2)}}.
\end{align}
Then, $E$ satisfies
\begin{align*}
\mathbb{N}_{11}^{(j)}(E, \lambda)&=\mathbb{N}_{11}^{(j)}(\Om, \lambda) \quad \mbox{for }j=1,2, \\
\mathbb{N}_{21}^{(2)}(E, \lambda)&=\mathbb{N}_{21}^{(2)}(\Om, \lambda).
\end{align*}
\end{lemma}

Assuming $\sigma$ (or $\lambda$) is known, we set an equivalent ellipse $E$ as in Lemma \ref{ee}.
 We then denote by $F_m[E](z)$ the Faber polynomials corresponding to $\Psi_E$. We define the curvilinear coordinate system $(\rho,\theta)$ corresponding to $\Psi_E$ as in subsection \ref{subsec:series}.
For $z=\Psi_E(w)\in\p E$, it holds that
\begin{align}
&h(\rho_e,\theta) = \left| \frac{\p \Psi_E}{\p \rho} \right| = \left|w - \frac{e_1}{w}\right|,\quad\rho_e=\ln(\gamma_e),\label{h:nu:pE1}\\
&\nu_{\p E}(z) = \frac{w - \frac{e_1}{w}}{\left|w - \frac{e_1}{w}\right|},\label{h:nu:pE2}
\end{align}
where $\nu_{\p E}$ is unit outward normal to $\p E$. 

Recall that GPTs are defined in terms of the solutions to the transmission problem corresponding to the functions $z^n$, which are the Faber polynomials corresponding to a disk. We now modify GPTs by employing $F_m[E](z)$ instead of $z^n$ as follows.
\begin{definition}[Modified GPTs using {$F_m[E](z)$}] 
For $m,n\in\NN$, we define
\begin{align}\label{def:EE:1}
\mathbb{E}_{mn}^{(1)}(\Om, \lambda)&=\int_{\p\Om}F_n[E](z)\left(\lambda I-\Kcal^*_{\p\Om}\right)^{-1}\left[\pd{}{\nu}F_m[E]\right](z)\,d\sigma(z), \\ \label{def:EE:2}
\mathbb{E}_{mn}^{(2)} (\Om, \lambda)&=\int_{\p\Om}F_n[E](z)\left(\lambda I-\Kcal^*_{\p\Om}\right)^{-1}\left[\pd{}{\nu}\overline{F_m[E]}\right](z)\,d\sigma(z).
\end{align}
\end{definition}

To recover the shape of an inclusion of arbitrary shape, we will apply the shape derivative approach in Theorem \ref{dd}, in an analytical way in the following subsection. 
Unlike Theorem \ref{dd}, in which the difference of GPTs was used, we now consider the difference of  modified GPTs, that is, 
\beq\label{def:Delta}
\Delta_{mn}^{(j)} := \mathbb{E}_{mn}^{(j)}(\Om,\lambda) - \mathbb{E}_{mn}^{(j)} (E,\lambda)
\eeq
for $m,n\in\NN$ and $j=1,2$. 
As we find the equivalent ellipse $E$ from GPTs as shown in Lemma \ref{ee}, $\mathbb{E}_{mn}^{(j)}(E,\lambda)$ and $\mathbb{E}_{mn}^{(j)}(\Om,\lambda)$ can be obtained by using only $\mathbb{N}_{mn}^{(j)}(\Om,\lambda)$ and so can $\Delta_{mn}^{(j)}$. More precisely, we have the following. 
\begin{lemma}\label{lemm:pE}
Set $p_{mn}[E]$ to be the coefficients of the Taylor series of $F_m[E]$, {\it i.e.}, 
$
F_m[E](z) = \sum_{n=0}^{m} p_{mn}[E] z^n.
$
For each $m\geq2$, we have
\begin{align*}
&\Delta_{m1}^{(1)}= \mathbb{E}_{m1}^{(1)} (\Om,\lambda) = \sum_{n=1}^m p_{mn}[E]\, \NN_{n1}^{(1)} (\Om,\lambda),\\
&\Delta_{m1}^{(2)}= \mathbb{E}_{m1}^{(2)} (\Om,\lambda) = \sum_{n=1}^m \overline{p_{mn}[E]}\, \NN_{n1}^{(2)} (\Om,\lambda).
\end{align*}
\end{lemma}
\proof
It is straightforward to see from \eqnref{def:EE:1} and \eqnref{def:EE:2} that
$$\mathbb{E}_{mn}^{(j)}(E,\lambda)=\FF_{mn}^{(j)}(E,\lambda).$$
From \eqnref{eqn:FF:E}, we have
\beq\label{eqn:EE:E}
\mathbb{E}_{m1}^{(1)} (E,\lambda) =\mathbb{E}_{m1}^{(2)} (E,\lambda) = 0\quad\mbox{for }m\geq2.
\eeq
By using the fact that $F_1[E](z)=z-e_0$, we complete the proof.
\qed

\subsection{Analytic shape recovery}

We now regard an inclusion $\Om$ as a perturbation of its equivalent ellipse $E$, {\it i.e.}, 
\beq\label{def:deform:ellipse}
\p \Om = \left\{ z + \ep f_{\p E} (z) \nu_{\p E}(z) : \, z\in \p E\right\}
\eeq
for some real-valued $\mathcal{C}^1$ function $f_{\p E}$. We set 
\begin{align}\notag
\ds f(\theta) =  \frac{1}{h(\rho_e,\theta)}\left(f_{\p E} \circ\Psi_E\right) (e^{\rho_e+i\theta}).
\end{align}
From \eqnref{h:nu:pE1} and \eqnref{h:nu:pE2}, it holds that
\beq\notag
z+ \ep f_{\p E}(z)\nu_{\p E}(z) = w+e_0+\frac{e_1}{w}+\ep\left(w-\frac{e_1}{w}\right)f(\theta).
\eeq
As $f$ is real-valued, it admits the Fourier series
$$f(\theta)= \widehat{f}_0 + \sum_{k=1}^\infty \left( \widehat{f}_k e^{ik\theta} + \overline{\widehat{f}_k} e^{-ik\theta} \right)=2 \mbox{\text{Re}}\left(\sum_{k=0}^\infty \widehat{f}_k e^{ik\theta} \right).$$

\begin{theorem}[Perturbed ellipse recovery] \label{et}
Let $\Om$ be a simply connected planar domain with $\mathcal{C}^2$-boundary, whose conductivity $\sigma$ is assumed to be known.  
Assume $\NN_{mn}^{(j)}(\Om,\lambda)$ for $ 1\leq m,n\leq M,\ j=1,2$ are given for some $M\in\NN$. 
Then, $\p\Om$ can be approximated as the parametrized curve:
$$w \longmapsto w + e_0 + \frac{e_1}{w} + \ep\left(w - \frac{e_1}{w}\right)2 \mbox{\text{Re}}\left(\sum_{k=0}^{M-1} \widehat{f}_k e^{ik\theta} \right), \qquad |w|=\gamma_e,
$$
where for each $m,n\in\NN$, the Fourier coefficients of $f$ satisfies
\begin{align}
&\ep\widehat{f}_{m-n} = \frac{\gamma_e^{m+n}}{4\pi  \lambda mn}  \left(s_n \overline{t_m}\, \Delta_{mn}^{(1)}  +s_m t_n\, \overline{\Delta_{mn}^{(1)}}  +  s_m s_n\, \Delta_{mn}^{(2)}+ \overline{t_m} t_n\, \overline{\Delta_{mn}^{(2)}}  \right) + O(\ep^2), \label{f2}  \\[1mm]
&\ep\widehat{f}_{m+n} = \frac{\gamma_e^{m+n}}{2\pi mn} \left(  \overline{t_m t_n}\,\Delta_{mn}^{(1)}+ s_m s_n \,  \overline{\Delta_{mn}^{(1)}}    +  s_m \overline{t_n}\, \Delta_{mn}^{(2)} +  s_n\overline{t_m}\, \overline{\Delta_{mn}^{(2)}} \right)  + O(\ep^2) \label{f1}
\end{align}
with 
\beq\label{def:sm:tm}
s_m = \frac{\lambda \gamma_e^{2m} - \frac{|e_1|^{2m}}{2\gamma_e^{2m}}}{\gamma_e^{4m} - |e_1|^{2m}},\quad 
t_m = \frac{e_1^m\left( \lambda -  \frac{1}{2} \right)}{\gamma_e^{4m} - |e_1|^{2m}}.
\eeq
\end{theorem}
\proof

Let $u$ and $v$ satisfy \eqnref{cond_eqn1} and \eqnref{cond_eqn2} with $H(z)=F_m[E](z)$ and $F(z)=F_n[E](z)$, respectively.

At $z = \Psi_E(\gamma_e e^{i\theta}) \in \partial E$, we have by applying \eqnref{eqn:normalderiv} to \eqnref{FPsi:Ellipse} that
\begin{align*}
\frac{\p }{\p \nu} F_m[E]\Big|^- &= \frac{1}{h(\rho_e, \theta)} \frac{\p }{\p \rho}F_m[E]\Big|^- = \frac{m}{h(\rho_e, \theta)} \left(e^{m(\rho_e + i\theta)} - e_1^m e^{-m(\rho_e + i\theta)}\right),\\
\frac{\p }{\p T}F_m[E]\Big|^-& = \frac{1}{h(\rho_e, \theta)} \frac{\p }{\p \theta}F_m[E]\Big|^- = i\frac{m}{h(\rho_e, \theta)} \left( e^{m(\rho_e+ i\theta)} - e_1^m e^{-m(\rho_e + i\theta)}\right).
\end{align*}
From Lemma \ref{invseries} and the fact that $c_{mk} = \delta_{mk} e_1^m$, it follows
\begin{align*}
\left(\lambda I-\Kcal^*_{\p E}\right)^{-1}[\zeta_m] &= \left(\lambda^2 - \frac{|e_1|^{2m}}{4\gamma_e^{4m}}\right)^{-1} \left( \lambda \zeta_m+ \frac{e_1^{m}}{2\gamma_e^{2m}} \overline{\zeta_{m}} \right), \\
\left(\lambda I-\Kcal_{\p E}\right)^{-1}[\eta_n] &= \left(\lambda^2 - \frac{|e_1|^{2n}}{4\gamma_e^{4n}}\right)^{-1} \left( \lambda \eta_n + \frac{e_1^{n}}{2\gamma_e^{2n}} \overline{\eta_{n}} \right).
\end{align*}
We then obtain from \eqnref{Su}, \eqnref{Sv} and Lemma \ref{lemma:seriesexpan} that
\begin{align}
u(z)& =\left(\lambda-\frac{1}{2}\right)\left(\lambda^2 - \frac{|e_1|^{2m}}{4\gamma_e^{4m}}\right)^{-1} \left(\lambda  F_m[E](z) + \frac{e_1^{m}}{2\gamma_e^{2m}} \overline{F_m[E](z)} \right), \label{u}\\[1mm]
v(z)& =\left(\lambda+\frac{1}{2}\right)\left(\lambda^2 - \frac{|e_1|^{2n}}{4\gamma_e^{4n}}\right)^{-1} \left(\lambda F_n[E](z) + \frac{e_1^{n}}{2\gamma_e^{2n}} \overline{F_n[E](z)} \right)\quad\mbox{in }\Om.\label{v} 
\end{align}
From \eqnref{eqn:normalderiv}, it follows that
\begin{align*}
&\pd{}{\nu}u(z)\Big|^-
= \frac{m d_m}{\gamma_e^{m}h(\rho_e, \theta)} \left(\lambda-\frac{1}{2}\right) \left(s_m e^{i m\theta}  - t_m e^{-i m\theta}\right) , \\[2mm]
&\pd{}{\nu}v(z)\Big|^-
= \frac{n d_n}{\gamma_e^{n}h(\rho_e, \theta)} \left(\lambda+\frac{1}{2}\right)  \left(s_n e^{i n\theta}  -  t_n e^{-i n\theta}\right)
\end{align*}
with $s_m$, $t_m$ given by \eqnref{def:sm:tm} and
 $$d_m=\frac{\gamma_e^{4m} - |e_1|^{2m}}{\lambda^2 - \frac{|e_1|^{2m}}{4\gamma_e^{4m}}}.$$
In the same manner, it holds that
\begin{align*}
&\pd{}{T}u(z)\Big|^-
= i\, \frac{m d_m}{\gamma_e^{m}h(\rho_e, \theta)} \left(\lambda-\frac{1}{2}\right) \left(s_m e^{i m\theta}  + t_m e^{-i m\theta} \right) , \\[2mm]
&\pd{}{T}v(z)\Big|^-
= i\, \frac{n d_n}{\gamma_e^{n}h(\rho_e, \theta)} \left(\lambda+\frac{1}{2}\right) \left(s_n e^{i n\theta}  +  t_n e^{-i n\theta}\right) .
\end{align*}

In view of \eqnref{def:EE:1} and \eqnref{def:EE:2}, we get the asymptotic integral expressions for the shape derivative by replacing $\Om_0$ with $E$ in Lemma \ref{lemma:shapederivative}:
\begin{align*}
&\mathbb{E}_{mn}^{(1)}(\Om,\lambda) - \mathbb{E}_{mn}^{(1)} (E,\lambda)\notag\\
&= \epsilon(\sigma-1) \int_{\p E} f_{\p E}(z) \left(\pd{{u}}{\nu}\Big|^-\pd{v}{\nu}\Big|^- +\frac{1}{\sigma}\pd{{u}}{T}\Big|^- \pd{v}{T}\Big|^- \right)(z)\, d \sigma(z) + O(\epsilon^2)
\end{align*}
and 
\begin{align*}
&\mathbb{E}_{mn}^{(2)}(\Om,\lambda) - \mathbb{E}_{mn}^{(2)} (E,\lambda)\notag\\
&= \epsilon(\sigma-1) \int_{\p E} f_{\p E}(z) \left( \pd{\overline{u}}{\nu}\Big|^-\pd{v}{\nu}\Big|^- +\frac{1}{\sigma}\pd{\overline{u}}{T}\Big|^- \pd{v}{T}\Big|^- \right)(z)\, d \sigma(z) + O(\epsilon^2).
\end{align*}
Remind that the length element on $\p\Om$ is $d\sigma(z)=h(\rho_0,\theta)d\theta$. 
From the asymptotic formula for the tangential and normal derivatives of $u$ and $v$, we obtain
\begin{align}
\Delta_{mn}^{(1)}
&= \mathbb{E}_{mn}^{(1)}(\Om,\lambda) - \mathbb{E}_{mn}^{(1)} (E,\lambda)\nonumber\\[1mm]
&= \frac{2 \pi \epsilon m n d_m d_n}{\gamma_e^{m+n}} \left(s_m s_n  \overline{\widehat{f}_{m+n}}  + t_m t_n  \widehat{f}_{m+n}  - 2\lambda \left(s_m t_n  \overline{\widehat{f}_{m-n}}   +  s_n t_m  \widehat{f}_{m-n}\right) \right)  + O(\epsilon^2) \label{1}
\end{align}
and 
\begin{align}
\Delta_{mn}^{(2)}
&=\mathbb{E}_{mn}^{(2)} (\Om,\lambda) - \mathbb{E}_{mn}^{(2)} (E,\lambda)\nonumber \\[1mm]
&= \frac{2 \pi \epsilon m n d_m d_n}{\gamma_e^{m+n}} \left(2\lambda \left(s_m s_n \widehat{f}_{m-n} +  \overline{t_m} t_n \overline{\widehat{f}_{m-n}} \right) - s_m t_n \widehat{f}_{m+n} - s_n \overline{t_m} \overline{\widehat{f}_{m+n}} \right)  + O(\epsilon^2).\label{2}
\end{align}
We now have four equations about the Fourier coefficients of $f$: \eqnref{1}, \eqnref{2}, and their conjugate systems. It is then straightforward to find explicit formulas for four unknowns, $\widehat{f}_{m+n}, \overline{\widehat{f}_{m+n}}, \widehat{f}_{m-n}, \overline{\widehat{f}_{m-n}}$, which are exactly \eqnref{f1} and \eqnref{f2}.

\qed

Let us compute first three terms of the Fourier coefficients of $f$. 
From the definition of the equivalent ellipse, we have 
\beq\label{D11}
\Delta_{11}^{(j)} = \mathbb{E}_{11}^{(j)} (\Om,\lambda) - \mathbb{E}_{11}^{(j)} (E,\lambda) = \NN_{11}^{(j)} (\Om,\lambda) - \NN_{11}^{(j)} (E,\lambda) = 0,\quad j=1,2.
\eeq
From \eqnref{Faber:ex1}, we have
$F_2[E](z) = z^2-2e_0 z+e_0^2-2e_1.$
It then follows from Lemma \ref{lemm:pE} and the definition of $e_0$ (see \eqnref{def:ee}) that
\begin{align*}
\Delta_{21}^{(1)} &= \FF_{21}^{(1)}[E] (\Om,\lambda) =  \NN_{21}^{(1)} (\Om,\lambda) -  2e_0\, \NN_{11}^{(1)} (\Om,\lambda),\\
\Delta_{21}^{(2)} &= \FF_{21}^{(2)}[E] (\Om,\lambda) =  \NN_{21}^{(2)} (\Om,\lambda) -  2\overline{e_0}\, \NN_{11}^{(2)} (\Om,\lambda) = 0.
\end{align*}
By substituting $(m,n)=(1,1)$ and $(m,n)=(2,1)$ into \eqnref{f2}, we obtain
\begin{align*}
\widehat{f}_0& =\widehat{f}_{1-1}= O(\epsilon),\\
\widehat{f}_1 &= \widehat{f}_{2-1}=\frac{\gamma_e^{3}}{8\pi \epsilon \lambda}\left(   s_1 \overline{t_2}\, \Delta_{21}^{(1)}+s_2 t_1\, \overline{\Delta_{21}^{(1)}}\right) + O(\epsilon).
\end{align*}
%$$\widehat{f}_1 = \frac{\gamma_e^{3}}{4\pi}\left(\overline{\Delta_{21}^{(1)}} s_2 t_1 + \Delta_{21}^{(1)} s_1 \overline{t_2}\right), \quad 
%\widehat{f}_3 =  \frac{\gamma_e^{3}}{4\pi} \left(\overline{\Delta_{21}^{(1)}} s_1 s_2  + \Delta_{21}^{(1)} \overline{t_1 t_2}\right).$$
We can apply \eqnref{f2} or \eqnref{f1} to get high-order Fourier coefficients. For example, we obtain $\widehat{f}_{2}$ by substituting $(m,n)=(3,1)$ or $(1,1)$ as follows:
\begin{align*}
&\widehat{f}_{2}=\widehat{f}_{3-1} = \frac{\gamma_e^{4}}{12\pi \epsilon \lambda}  \left(s_1 \overline{t_3}\, \Delta_{31}^{(1)}  +s_3 t_1\, \overline{\Delta_{31}^{(1)}}  +  s_1 s_3\, \Delta_{31}^{(2)} + t_1 \overline{t_3}\, \overline{\Delta_{31}^{(2)}}  \right) + O(\epsilon),  \\[1mm]
&\widehat{f}_{2}=\widehat{f}_{1+1} = \frac{\gamma_e^{2}}{2\pi \epsilon} \left(  \overline{t_1}^2 \,\Delta_{11}^{(1)}+ s_1^2 \,  \overline{\Delta_{11}^{(1)}}  +  s_1 \overline{t_1}\, \Delta_{11}^{(2)} +  s_1\overline{t_1}\, \overline{\Delta_{11}^{(2)}}\right)  + O(\epsilon) = O(\epsilon).
\end{align*}
Here, we use equation \eqnref{D11} in the last equality.

\begin{remark}
% Whether to use \eqnref{f2} or \eqnref{f1} for shape recovery always depends on the shape or conductivity. Due to errors as much as $O(\epsilon)$, the values of both formulas are not exactly the same. 
% By applying \eqnref{f1} instead of \eqnref{f2}, more Fourier coefficients can be obtained even with GPTs of relatively low order. For this reason, you can usually use \eqnref{f1}. However, we recommend \eqnref{f2} if the error term for formula \eqnref{f1} is too large to reconstruct the shape.
 
 In numerical examples in section \ref{sec:numerical}, we use \eqnref{f2} to compute $\widehat{f}_{m-n}$ with $n$ fixed to be $1$.

\end{remark}

\section{Numerical results}\label{sec:numerical}
We present numerical results obtained by the two proposed analytic shape recovery methods (Theorem \ref{rs} and Theorem \ref{et}) and compare them with the method proposed in \cite{Ammari:2010:CIP} ({\it i.e.}, Theorem \ref{dd}).
We call the three methods (Theorem \ref{rs}, Theorem \ref{et} and Theorem \ref{dd}) {\it the conformal mapping recovery}, {\it perturbed ellipse recovery}, and {\it perturbed disk recovery}, respectively.

\begin{figure}[htp!]
\centering
\begin{subfigure}[t]{0.24\textwidth}
\centering
\includegraphics[width=\textwidth]{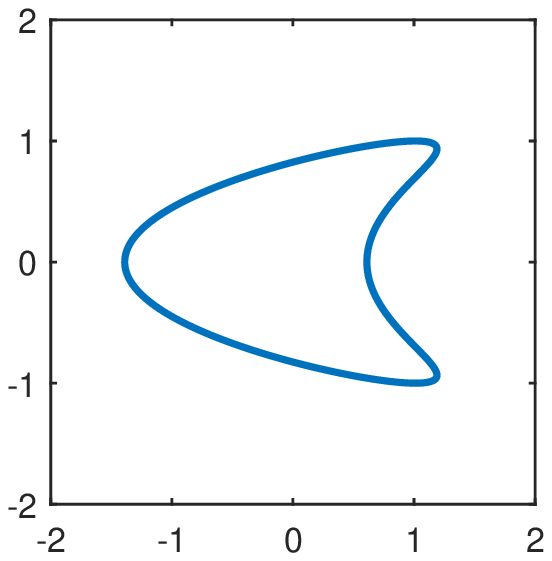}
\end{subfigure}
\begin{subfigure}[t]{0.24\textwidth}
\centering
\includegraphics[width=\textwidth]{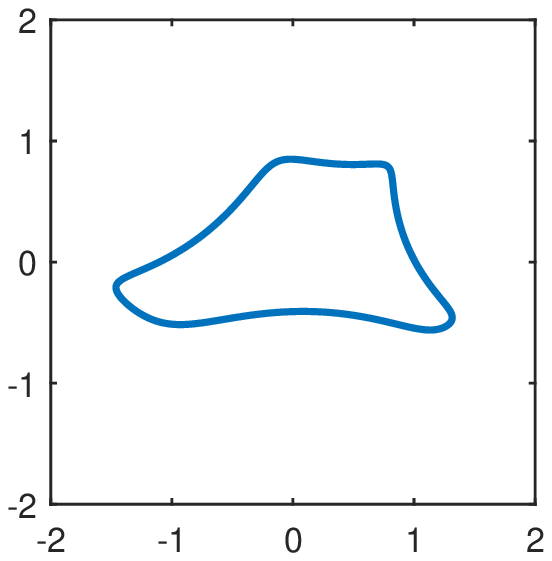}
\end{subfigure}
\begin{subfigure}[t]{0.24\textwidth}
\centering
\includegraphics[width=\textwidth]{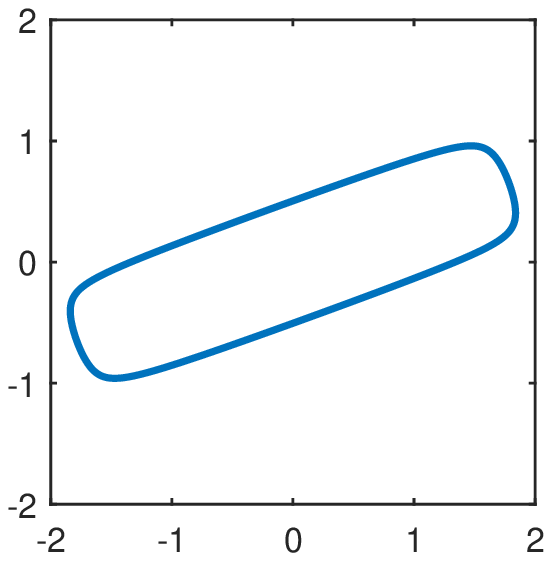}
\end{subfigure}
\begin{subfigure}[t]{0.24\textwidth}
\centering
\includegraphics[width=\textwidth]{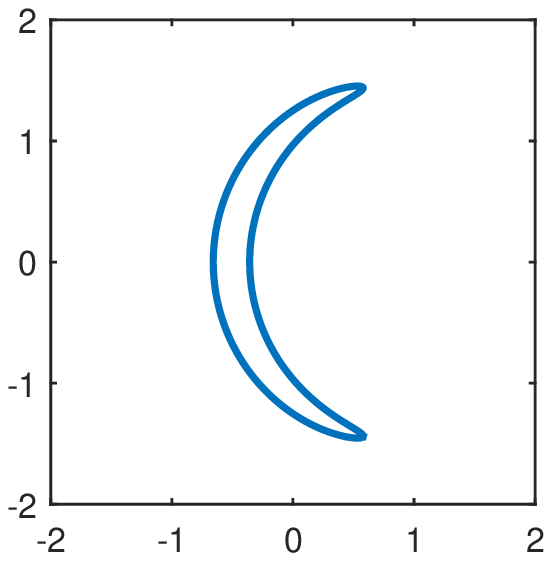}
\end{subfigure}
\caption{Simulation geometry of simply connected inclusions.}
\label{fig}
\end{figure}

We show numerical results of four different shapes; Fig.\;\ref{fig} illustrates the examples. 
All boundaries of the four domains are $\mathcal{C}^2$-curves.
In each simulation, we use GPTs up to some given order, namely $\text{Ord}\in\NN$. In other words,
\beq\label{ordGPTs}
\{\NN_{mn}^{(j)}\}_{j=1,2,\ 1\leq m,n\leq \text{Ord}}.
\eeq
  To acquire the values of GPTs, we use the integral formulation \eqnref{u:int:formulation}. More precisely, we solve \eqnref{eqn:integral} by employing the Nystr\"{o}m discretization method; see \cite[Section 17.1, 17.3]{Ammari:2013:MSM} for numerical codes. We then numerically evaluate the integrals \eqnref{def:GPT1} and \eqnref{def:GPT2}.

For the perturbed ellipse (or disk) recovery, the conductivity $\sigma$ is assumed to be known in each simulation. We compute the coefficients $\widehat{f}_k$ for order up to $k=\text{Ord}-1$ by the formula of 
$\widehat{f}_{m-n}$ with $m=k+1$ and $n=1$ in Theorem \ref{dd} (or Theorem \ref{et}).

For the conformal mapping recovery, we compute the coefficients of the conformal mapping of order up to a given Ord, following the formulas in Theorem \ref{rs}.

In the remaining figures, the gray curve shows the original shape of an example inclusion and the black curve shows the reconstructed shape. 
In each example, the conductivity values is either all smaller than $1$, or all bigger than $1$. 
The conformal mapping recovery shows high accuracy for all examples when the conductivity has either high or low value. The perturbed ellipse (or disk) recovery does not rely much on the value of conductivity. 
The perturbed disk recovery, in which an inclusion is regarded as a perturbation of a disk, 
is not applicable to either an asymmetric or a straight shape (the second and third examples in Fig.\;\ref{fig}).
However, the perturbed ellipse recovery restores successfully the examples of both straight and asymmetric shapes (see Fig.\;\ref{fig:cap_initial}--\ref{fig:straight_initial}). Especially for the straight shape, it shows a very good recovery because the equivalent ellipse reflects the eccentricity of the straight shape; see Fig.\;\ref{fig:straight_initial}. Nevertheless, the perturbed ellipse recovery does not show a good result if the inclusion is heavily concave; see Fig.\;\ref{fig:moon_initial}.

\smallskip
\noindent\textbf{Example 1\,}
In Fig.\;\ref{fig:kite_initial}, we consider the inclusion whose boundary is given by a parametrization: 
$$z(\theta) = 0.311 + \cos\theta - 0.7\cos2\theta + i \sin\theta\quad\mbox{with }\theta\in[0,2\pi).$$

The perturbed disk recovery (first column) and the perturbed ellipse recovery (second column) show similar performance since this example has an equivalent ellipse that is quite similar to a disk. Fig.\;\ref{fig:kite_GPT} shows the results obtained by using GPTs of various orders. 

\smallskip
\smallskip

\noindent\textbf{Example 2\,} In Fig.\;\ref{noise}, we consider an asymmetric-shaped inclusion given by the parametrization:
$$
\Psi(w) = e^{\frac{\pi i}{5}}\left( w + \tfrac{1-2i}{7}w^{-1} + \tfrac{i-1}{6}w^{-2} + \tfrac{i}{20}w^{-3} + \tfrac{1}{20}w^{-4} + \tfrac{i}{20}w^{-5} + \tfrac{i}{50}w^{-6} \right)
$$
with $|w|=1$.
 
We show the simulation results obtained by using noisy information of GPTs with $\mbox{SNR}=\infty$ (no noise), $5$ and $2$, where the signal-to-noise ratio (SNR) is defined to be
$$
\text{SNR} = -10 \log_{10} (Var) \quad \mbox{or} \quad Var = 10^{-\text{SNR}/10}.
$$
Here, $Var$ is the variance of additive white noise compared to the signal power. More precisely, we generate the noise by using the Gaussian distribution as follows:
$$\text{GPT}_\text{noise} = \text{GPT}_\text{origial} + \mathcal{N}(0, Var)$$
with the normal distribution $\mathcal{N}$.
All three shape recovery methods are fairly stable with noise. 

The perturbed ellipse recovery shows much better restoration than the perturbed disk recovery. The conformal mapping recovery shows good results when the conductivity approaches zero or infinity, as expected.

\smallskip
\smallskip

\noindent\textbf{Example 3\,} In Fig.\;\ref{fig:straight_initial}, we consider a straight shape whose boundary is the parametrized curve $z(t) = e^{\frac{\pi i}{9}}z_0(t)$, where $z_0(t)$ is given by
\begin{align*}
z_0(t)
&= 
\sqrt{2 + u^2 - v^2 + 2\sqrt{2}u} - \sqrt{2 + u^2 - v^2 - 2\sqrt{2}u}\\
&\quad + \frac{i}{4}\left(\sqrt{2 - u^2 + v^2 + 2\sqrt{2}v}- \sqrt{2 - u^2 + v^2 - 2\sqrt{2}v}\right)
\end{align*}
with $u=\cos t$, $v=\sin t$,  $t\in[0,2\pi)$.
For this example shape, the associated equivalent ellipse has a small aspect ratio and, thus, the perturbed disk recovery does not show good recovery. However, the perturbed ellipse recovery works well for all conductivity values.

\smallskip
\smallskip

\noindent\textbf{Example 4\,}
In this example (Fig.\;\ref{fig:moon_initial}), we consider a crescent shape given by
$$
z(t) = \frac{5z_0(t) - 20i}{2z_0(t) + 40i},
$$
where $z_0(t)$ is given by
\begin{align*}
z_0(t)
&= 
15\left(\sqrt{2 + u^2 - v^2 + 2\sqrt{2}u} - \sqrt{2 + u^2 - v^2 - 2\sqrt{2}u}\right)\\
&\quad + i \left(\sqrt{2 - u^2 + v^2 + 2\sqrt{2}v}- \sqrt{2 - u^2 + v^2 - 2\sqrt{2}v}\right)
\end{align*}
with $u=\cos t$, $v=\sin t$,  $t\in[0,2\pi)$.
While the perturbed ellipse (or disk) recovery shows worse results than in previous examples, the conformal mapping recovery shows a good reconstruction for the extreme conductivity case.

\section{Conclusion}
This study presents two analytical methods of recovering a simply connected conductivity inclusion in two dimensions from multistatic measurements, based on complex analysis and the concepts of the generalized polarization tensors (GPTs) and the Faber polynomial polarization tensors (FPTs). First, we provide an exact, simple expression of the conformal mapping associated with the inclusion in terms of GPTs given that the inclusion is either insulating or perfectly conducting. This expression allows us to accurately recover the shape of an inclusion with near-extreme conductivity. 
Secondly, we derive an asymptotic formula to approximate an inclusion with arbitrary conductivity by considering the inclusion as a perturbation of its equivalent ellipse. This formula provides us a good recovery for inclusions of general shapes, including straight or asymmetric shapes.

\begin{figure}[t!]
\begin{subfigure}{\linewidth}
\raggedright\qquad
\captionsetup{justification=centering}
\begin{minipage}{0.11\linewidth}
\subcaption*{}
\end{minipage}
\begin{minipage}{0.25\linewidth}
\subcaption*{Perturbed disk\\ recovery}
\end{minipage}\,
\begin{minipage}{0.25\linewidth}
\subcaption*{Perturbed ellipse\\ recovery}
\end{minipage}\,
\begin{minipage}{0.25\linewidth}
\subcaption*{Conformal mapping\\ recovery}
\end{minipage}\,
\end{subfigure}\\
\hskip -1.5cm
\begin{subfigure}{\linewidth}
\raggedright\qquad
\begin{minipage}{0.11\linewidth}
\subcaption*{$\sigma=5$}
\subcaption*{$\text{Ord} = 6$}
\end{minipage}\,
\begin{minipage}{0.25\linewidth}
\includegraphics[width=.95\linewidth,trim=70 30 50 15, clip]{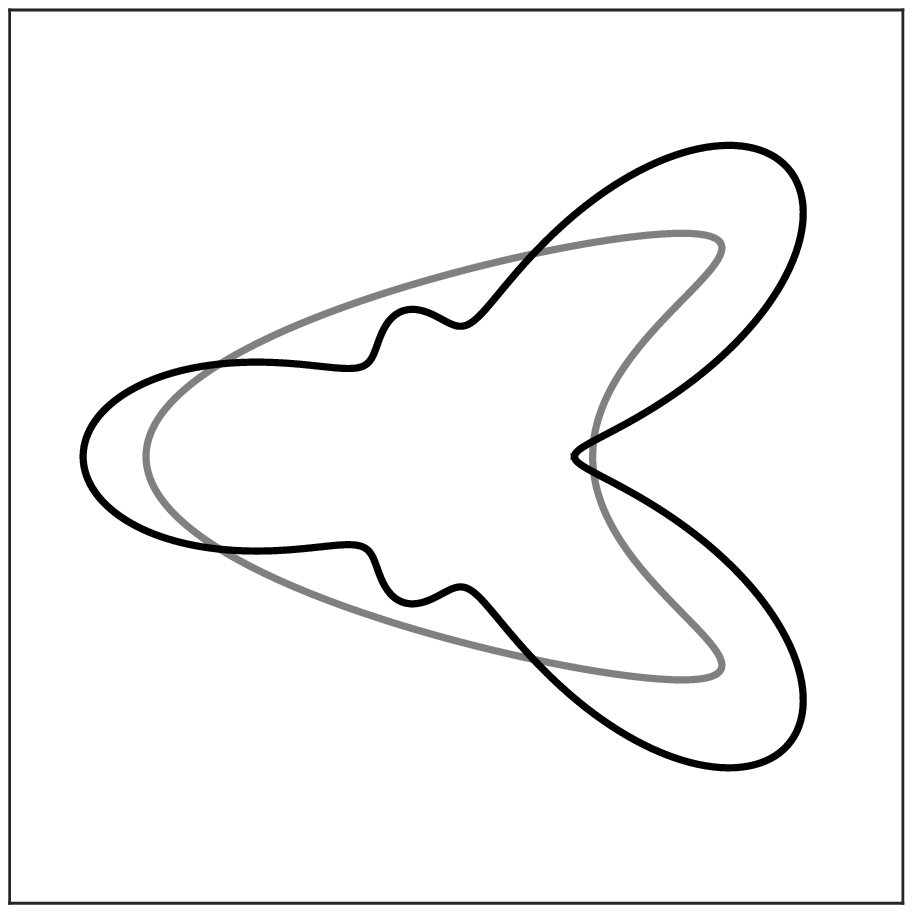}
\end{minipage}\,
\begin{minipage}{0.25\linewidth}
\includegraphics[width=.95\linewidth,trim=70 30 50 15, clip]{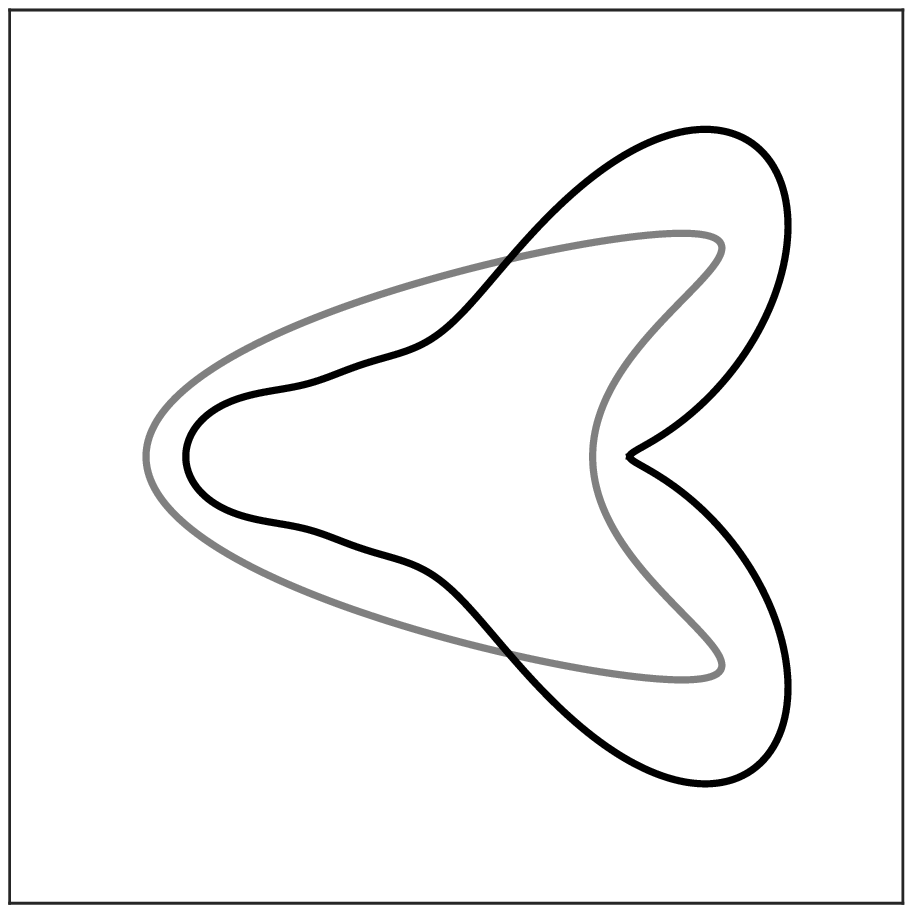}
\end{minipage}\,
\begin{minipage}{0.25\linewidth}
\includegraphics[width=.95\linewidth,trim=70 30 50 15, clip]{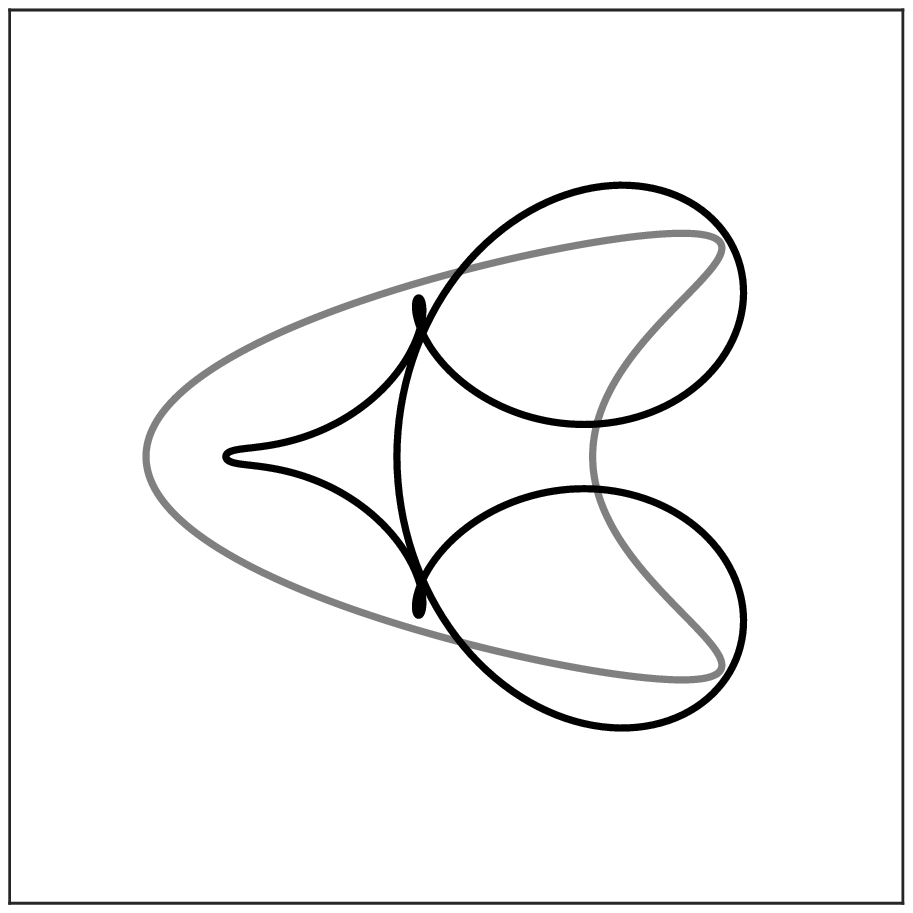}
\end{minipage}\,
\end{subfigure}\\
\hskip -1.5cm
\begin{subfigure}{\linewidth}
\raggedright\qquad
\begin{minipage}{0.11\linewidth}
\subcaption*{$\sigma=50$}
\subcaption*{$\text{Ord} = 6$}
\end{minipage}\,
\begin{minipage}{0.25\linewidth}
\includegraphics[width=.95\linewidth,trim=70 30 50 15, clip]{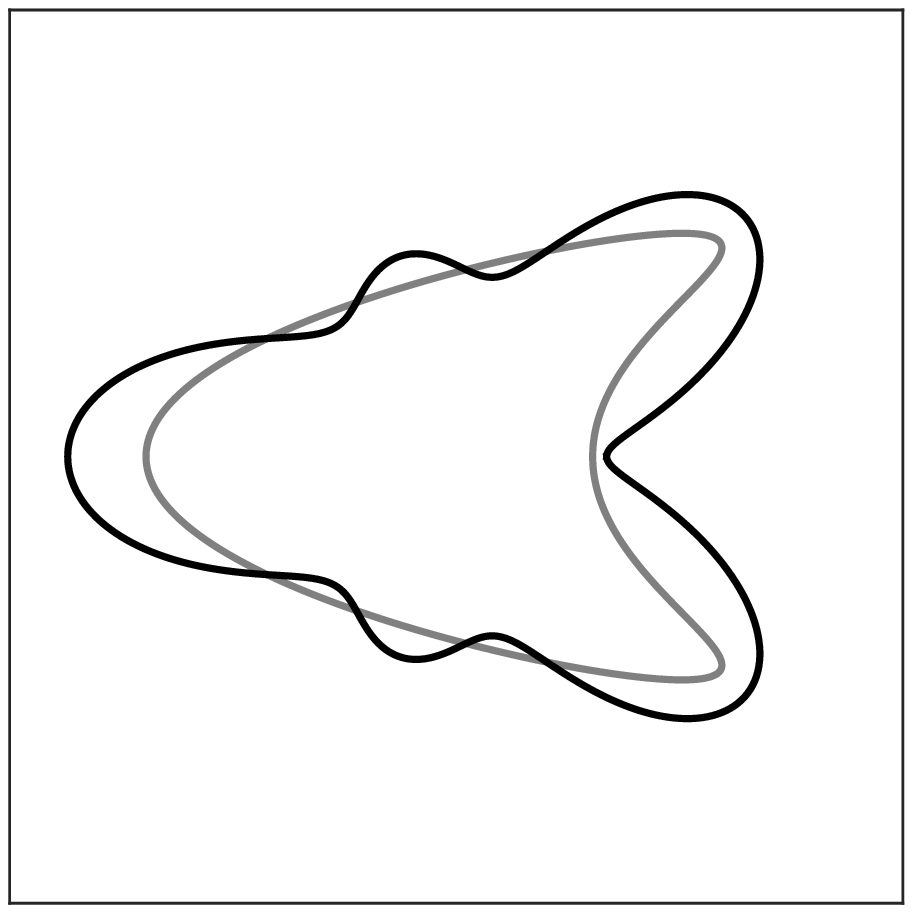}
\end{minipage}\,
\begin{minipage}{0.25\linewidth}
\includegraphics[width=.95\linewidth,trim=70 30 50 15, clip]{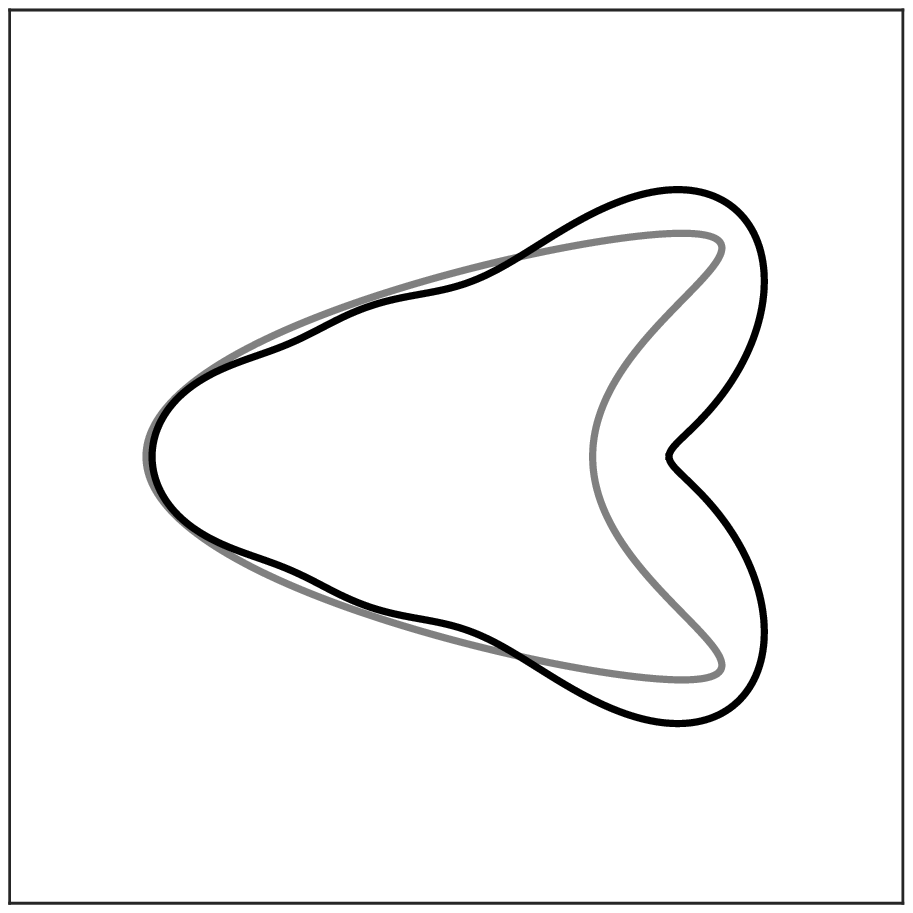}
\end{minipage}\,
\begin{minipage}{0.25\linewidth}
\includegraphics[width=.95\linewidth,trim=70 30 50 15, clip]{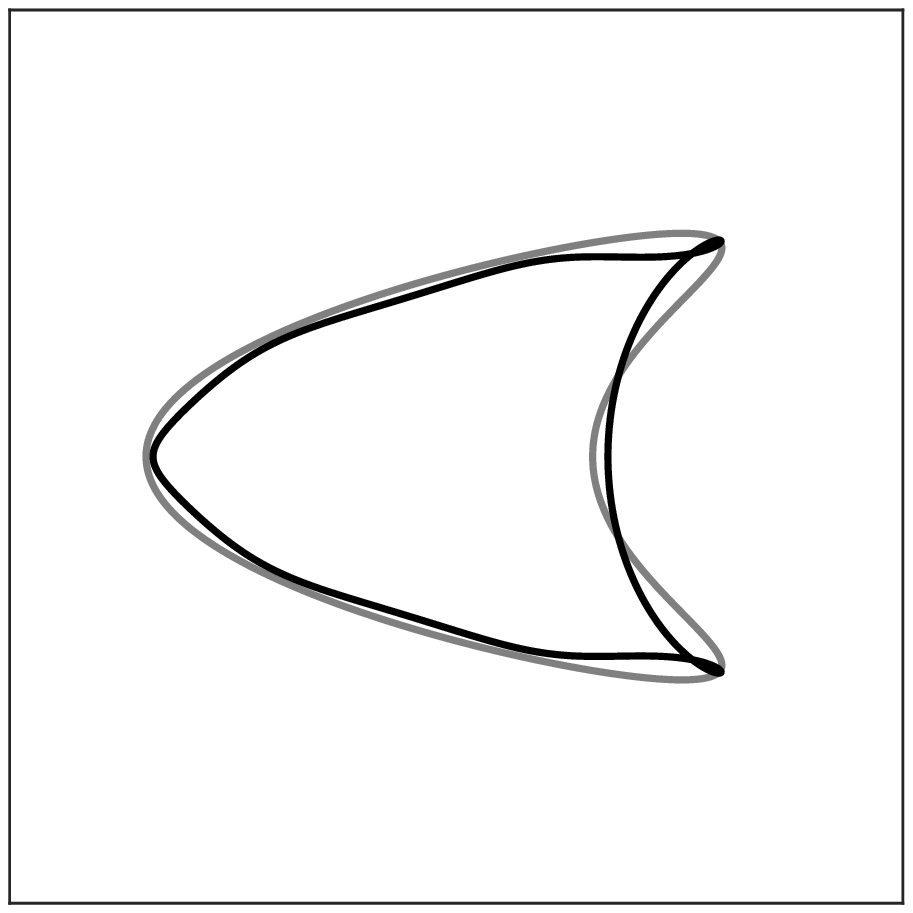}
\end{minipage}\,
\end{subfigure}\\
\hskip -1.5cm
\begin{subfigure}{\linewidth}
\raggedright\qquad
\begin{minipage}{0.11\linewidth}
\subcaption*{$\sigma=\infty$}
\subcaption*{$\text{Ord} = 6$}
\end{minipage}\,
\begin{minipage}{0.25\linewidth}
\includegraphics[width=.95\linewidth,trim=70 30 50 15, clip]{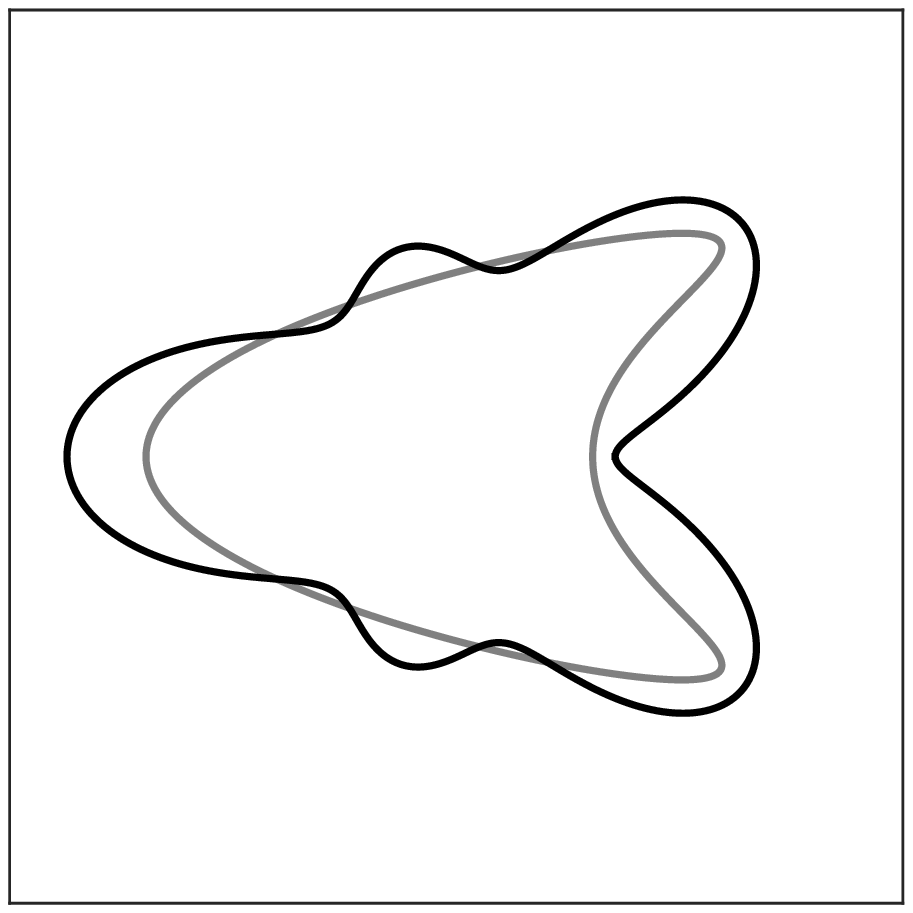}
\end{minipage}\,
\begin{minipage}{0.25\linewidth}
\includegraphics[width=.95\linewidth,trim=70 30 50 15, clip]{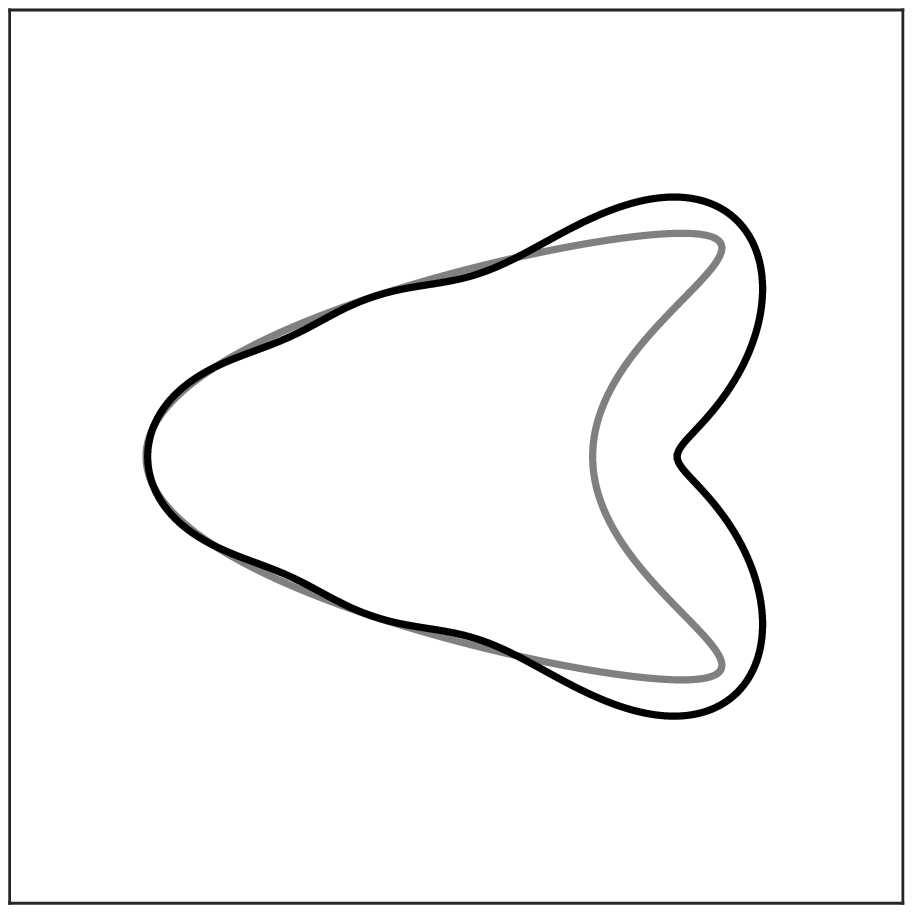}
\end{minipage}\,
\begin{minipage}{0.25\linewidth}
\includegraphics[width=.95\linewidth,trim=70 30 50 15, clip]{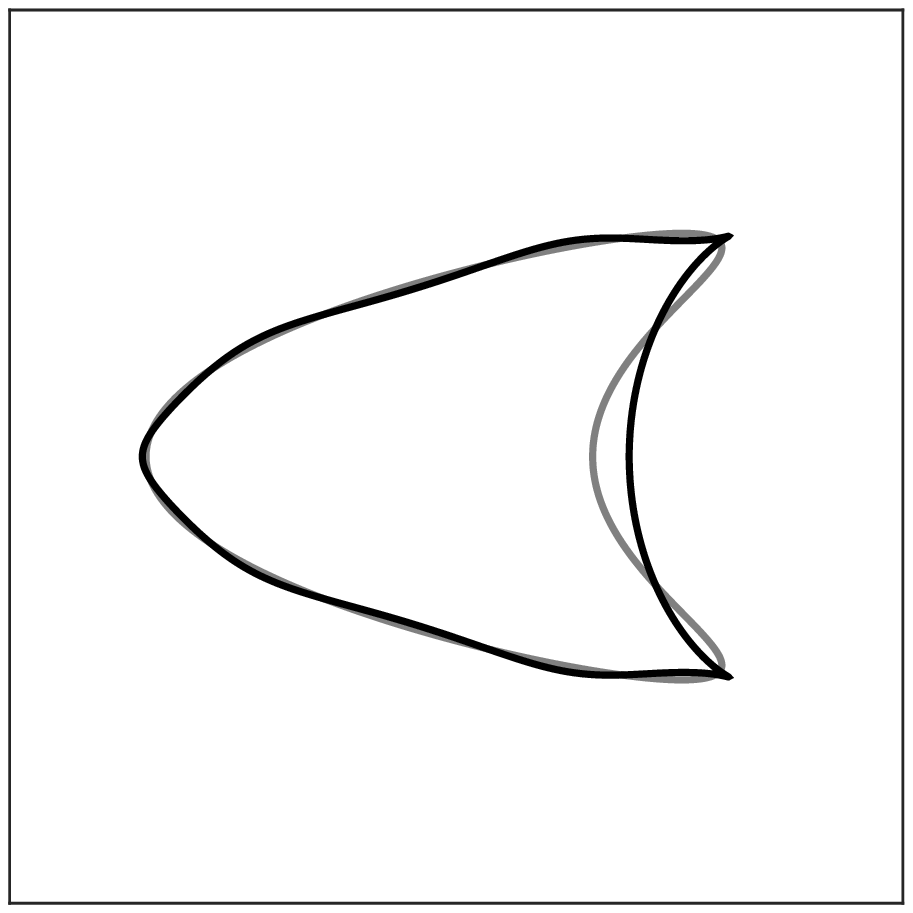}
\end{minipage}\,
\end{subfigure}
\caption{Recovery of a kite-shaped inclusion with various conductivities. 
For the perturbed disk (or ellipse) recovery, $\sigma$ is assumed to be known. The values of GPTs up to $\text{Ord}=6$ are used; see Fig.\;\ref{fig:kite_GPT}  for the reconstruction with $\text{Ord}=2,4$.}
\label{fig:kite_initial}
\end{figure}
\begin{figure}[H]
\begin{subfigure}{\linewidth}
\raggedright\qquad
\captionsetup{justification=centering}
\begin{minipage}{0.11\linewidth}
\subcaption*{}
\end{minipage}\,
\begin{minipage}{0.25\linewidth}
\subcaption*{Perturbed disk\\ recovery}
\end{minipage}\,
\begin{minipage}{0.25\linewidth}
\subcaption*{Perturbed ellipse\\ recovery}
\end{minipage}\,
\begin{minipage}{0.25\linewidth}
\subcaption*{Conformal mapping\\ recovery}
\end{minipage}\,
\end{subfigure}\\
\hskip -1.5cm
\begin{subfigure}{\linewidth}
\raggedright\qquad
\begin{minipage}{0.11\linewidth}
\subcaption*{$\sigma=50$}
\subcaption*{$\text{Ord}= 2$}
\end{minipage}\,
\begin{minipage}{0.25\linewidth}
\includegraphics[width=.95\linewidth,trim=70 30 50 15, clip]{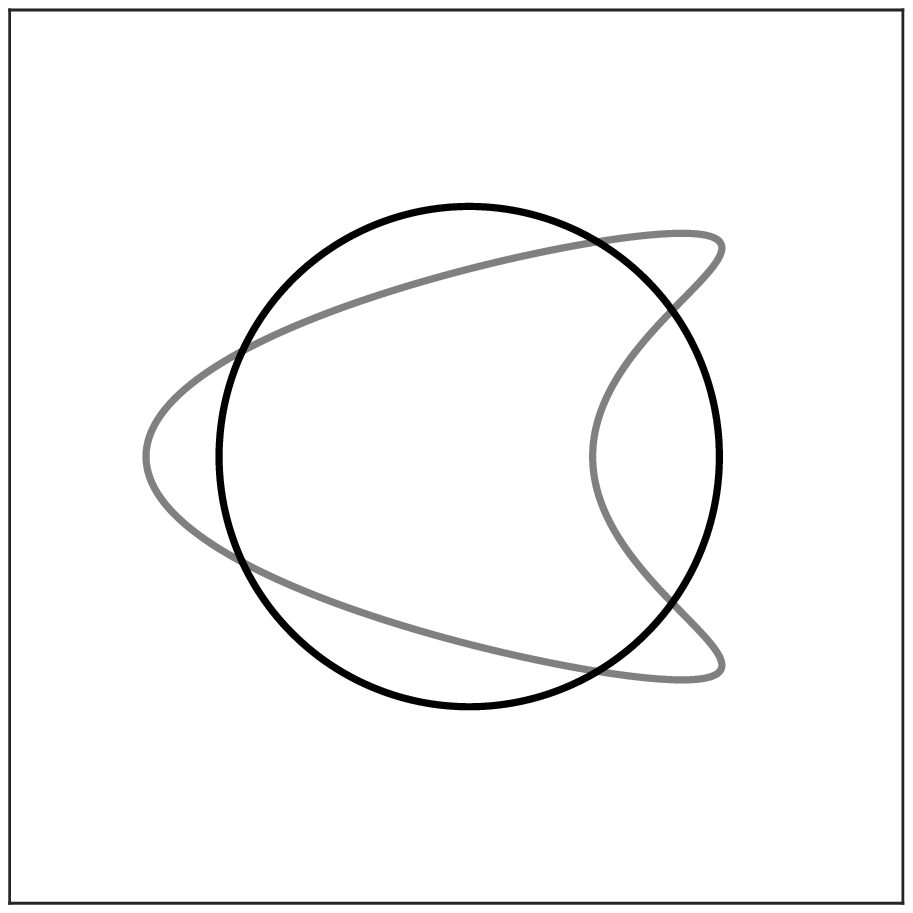}
\end{minipage}\,
\begin{minipage}{0.25\linewidth}
\includegraphics[width=.95\linewidth,trim=70 30 50 15, clip]{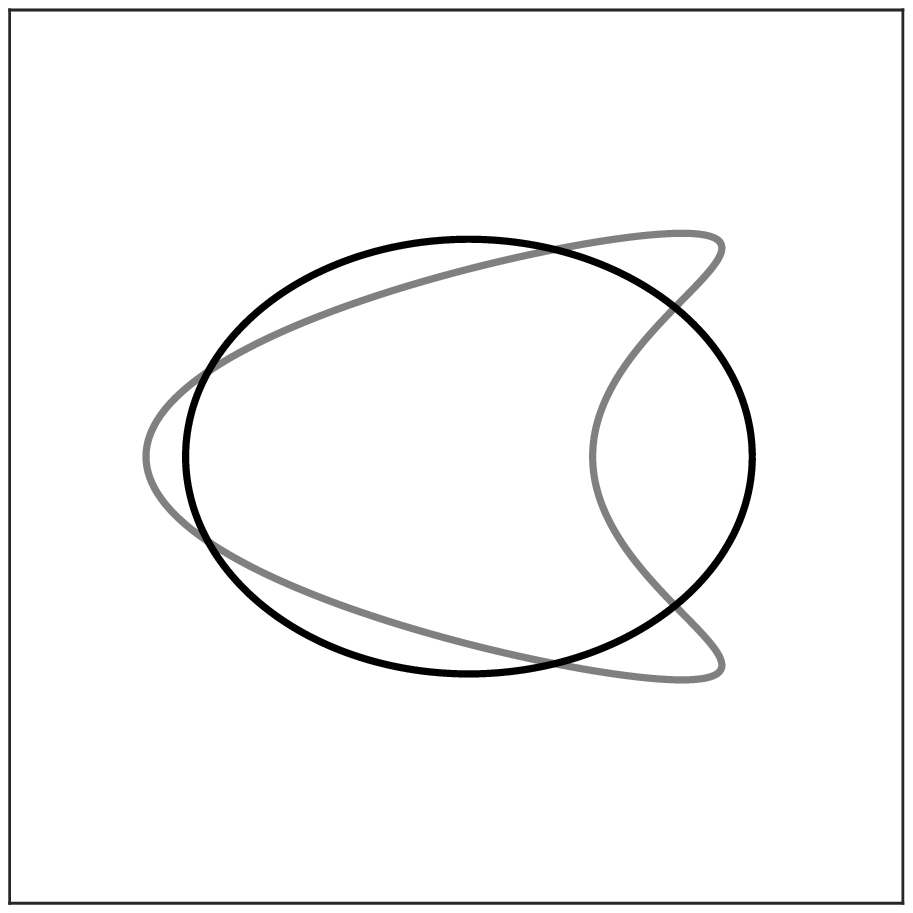}
\end{minipage}\,
\begin{minipage}{0.25\linewidth}
\includegraphics[width=.95\linewidth,trim=70 30 50 15, clip]{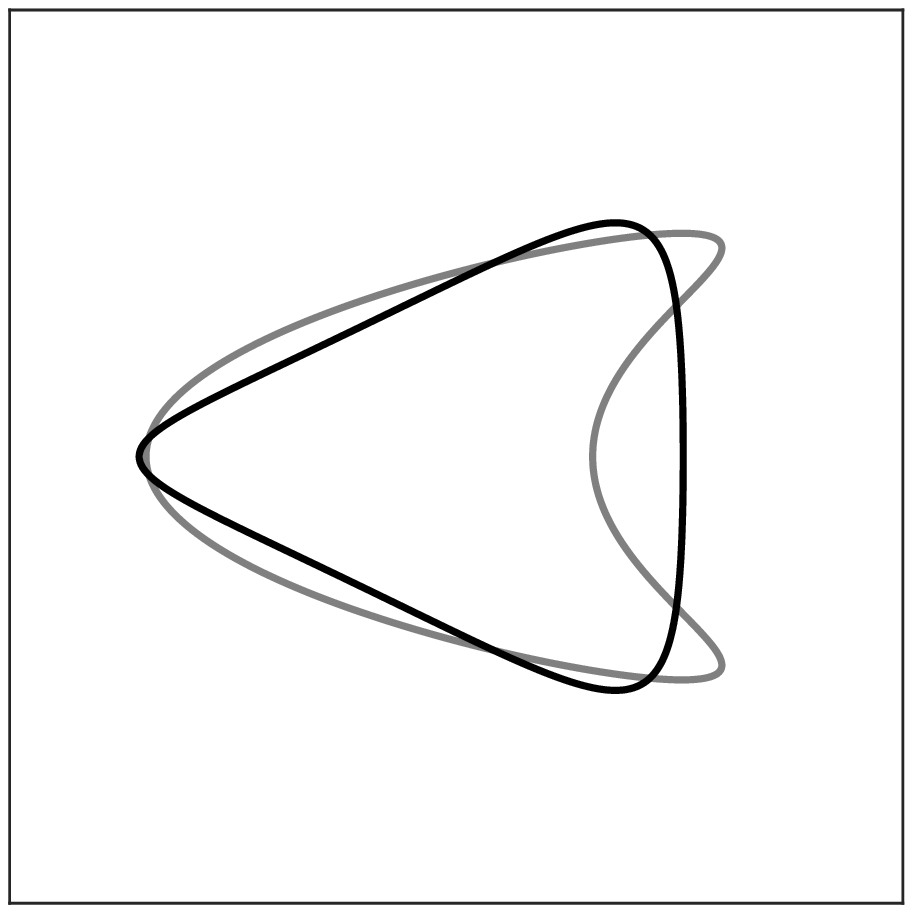}
\end{minipage}\,
\end{subfigure}\\
\hskip -1.5cm
\begin{subfigure}{\linewidth}
\raggedright\qquad
\begin{minipage}{0.11\linewidth}
\subcaption*{$\sigma=50$}
\subcaption*{$\text{Ord}= 4$}
\end{minipage}\,
\begin{minipage}{0.25\linewidth}
\includegraphics[width=.95\linewidth,trim=70 30 50 15, clip]{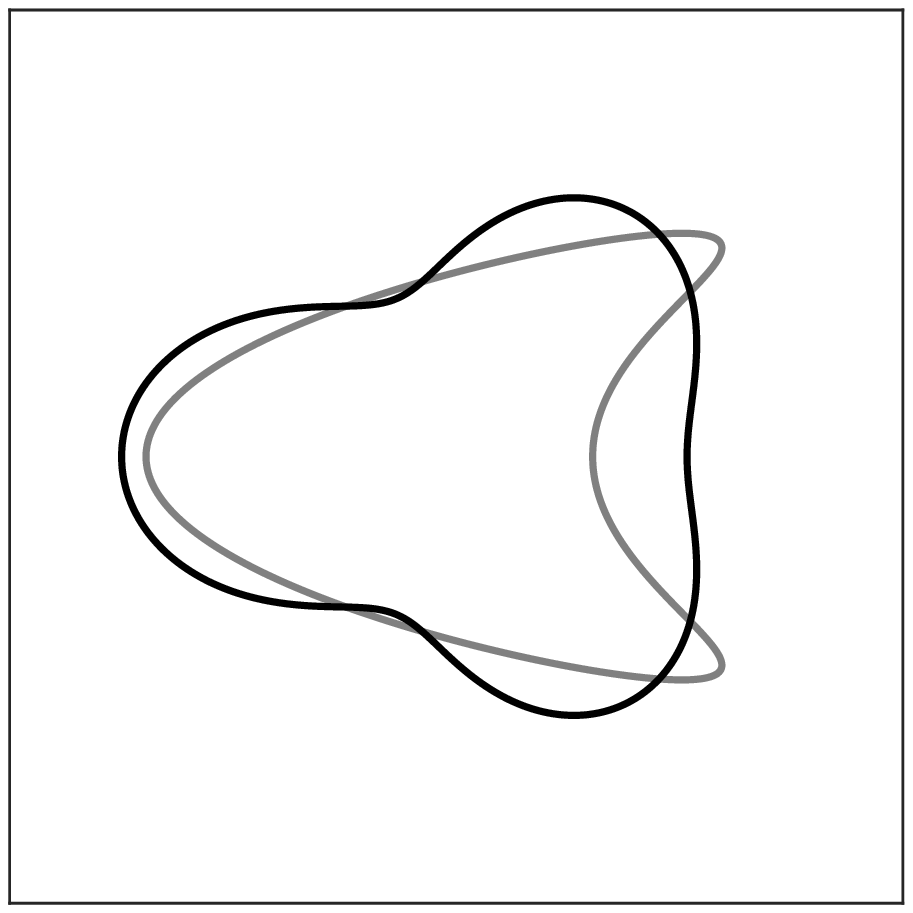}
\end{minipage}\,
\begin{minipage}{0.25\linewidth}
\includegraphics[width=.95\linewidth,trim=70 30 50 15, clip]{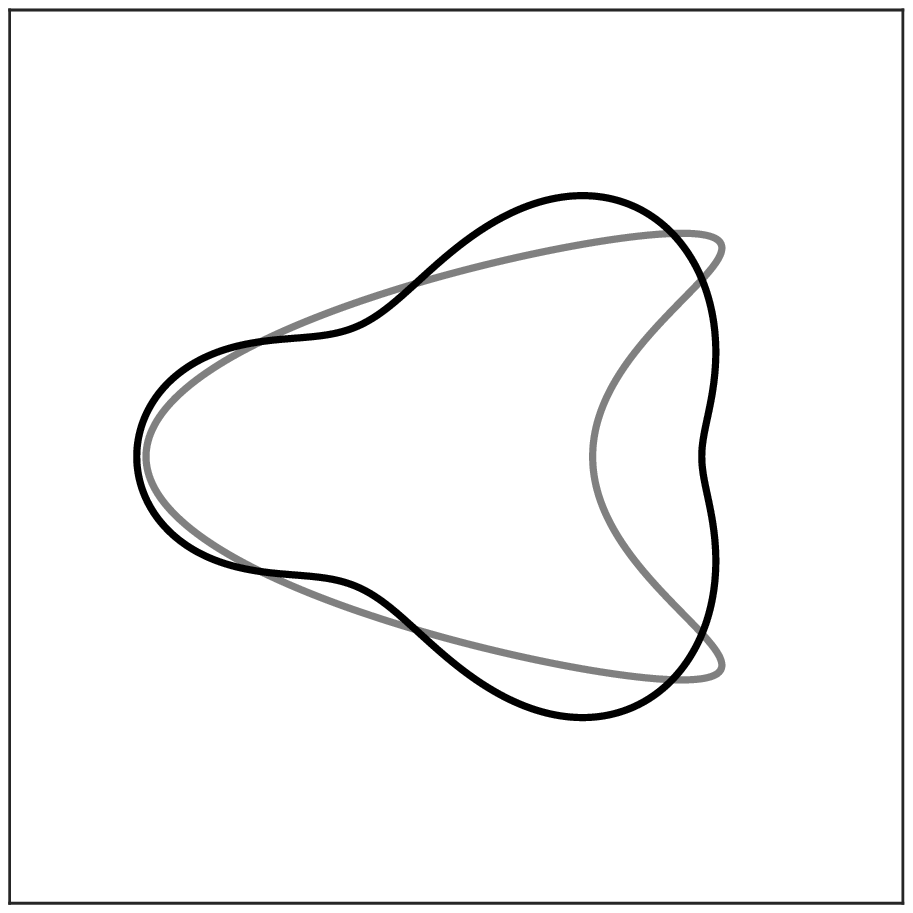}
\end{minipage}\,
\begin{minipage}{0.25\linewidth}
\includegraphics[width=.95\linewidth,trim=70 30 50 15, clip]{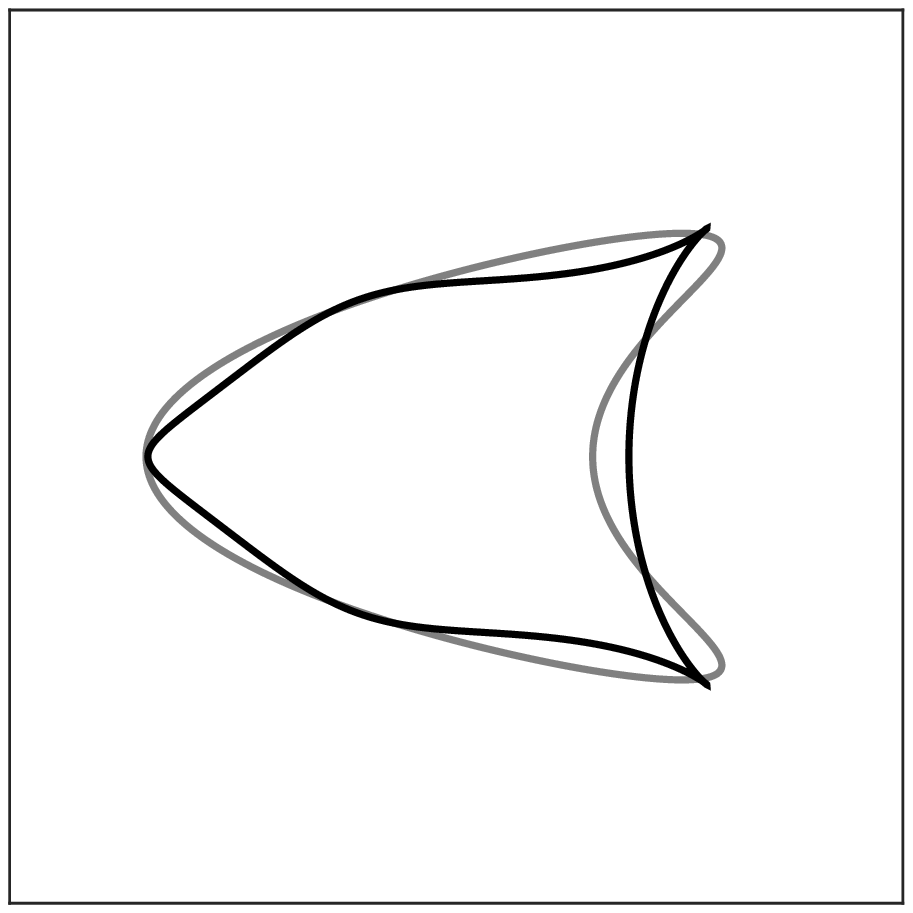}
\end{minipage}\,
\end{subfigure}
\caption{A kite-shaped inclusion with conductivity $\sigma=50$.
GPTs up to $\text{Ord}=2, 4$ are used. 
All three shape recovery methods show better results as Ord increases.
}
\label{fig:kite_GPT}
\end{figure}
\begin{figure}[t!]
\begin{subfigure}{\linewidth}
\raggedright\qquad
\captionsetup{justification=centering}
\begin{minipage}{0.11\linewidth}
\subcaption*{}
\end{minipage}\,
\begin{minipage}{0.25\linewidth}
\subcaption*{Perturbed disk\\ recovery}
\end{minipage}\,
\begin{minipage}{0.25\linewidth}
\subcaption*{Perturbed ellipse\\ recovery}
\end{minipage}\,
\begin{minipage}{0.25\linewidth}
\subcaption*{Conformal mapping\\ recovery}
\end{minipage}\,
\end{subfigure}\\
\hskip -1.5cm
\begin{subfigure}{\linewidth}
\raggedright\qquad
\begin{minipage}{0.11\linewidth}
\subcaption*{$\sigma=1/5$}
\subcaption*{SNR = $\infty$}
\end{minipage}\,
\begin{minipage}{0.25\linewidth}
\includegraphics[width=.95\linewidth,trim=70 30 50 15, clip]{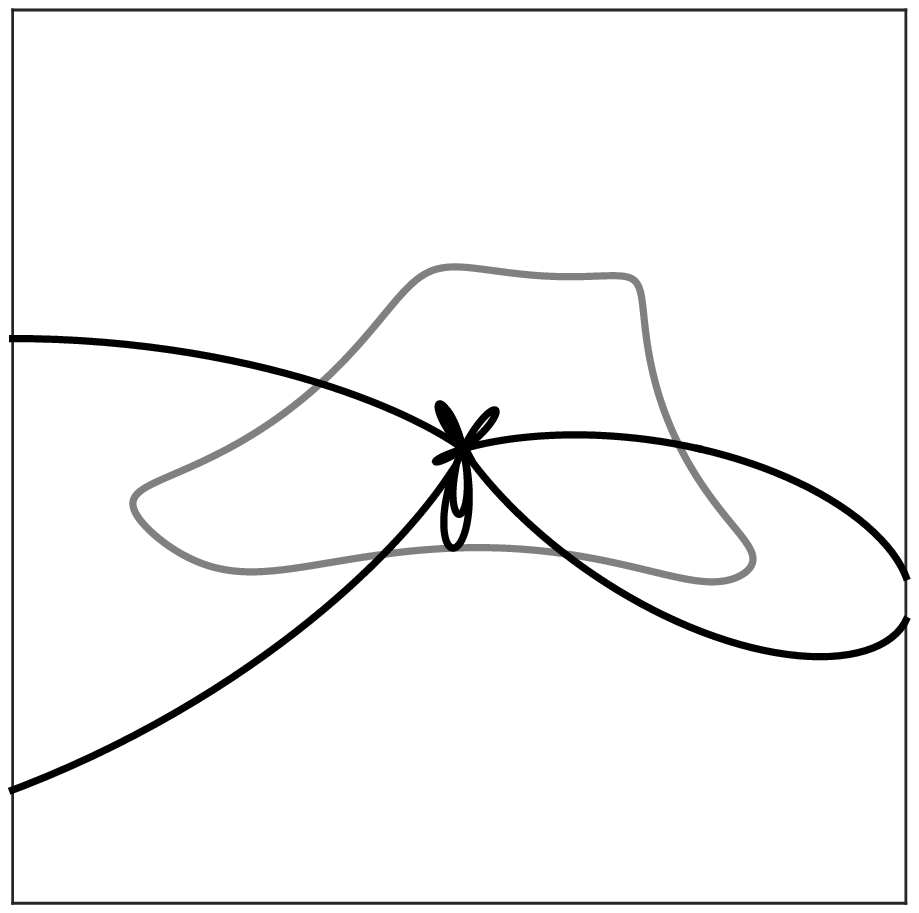}
\end{minipage}\,
\begin{minipage}{0.25\linewidth}
\includegraphics[width=.95\linewidth,trim=70 30 50 15, clip]{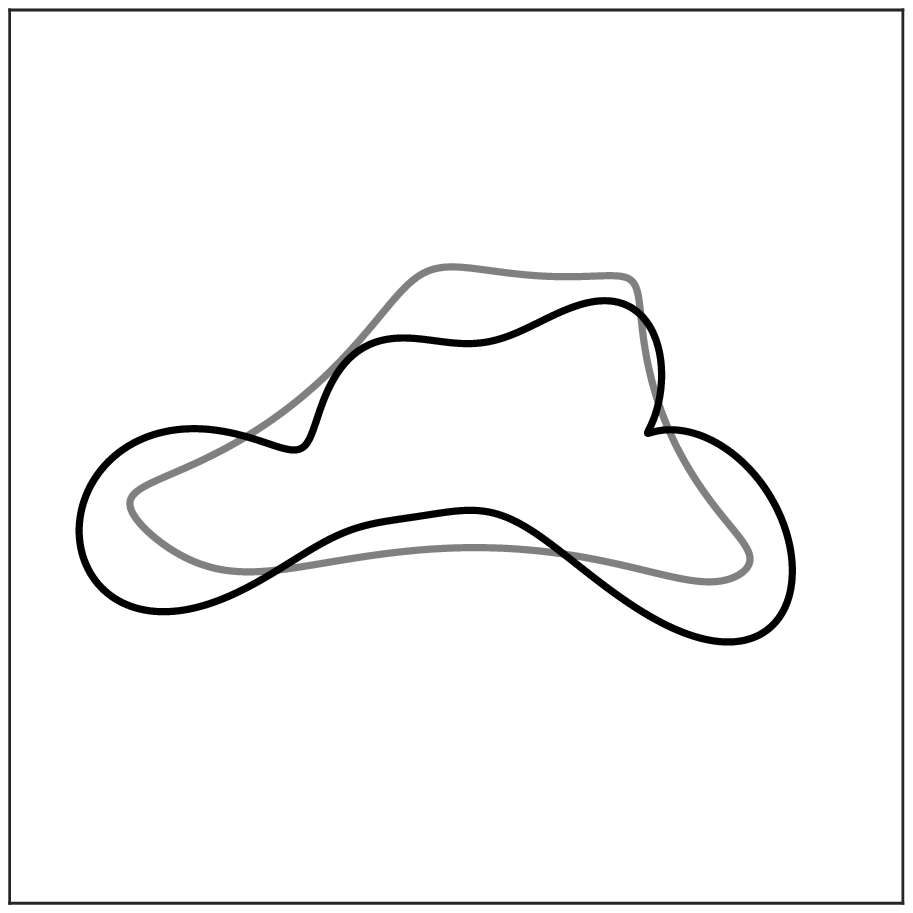}
\end{minipage}\,
\begin{minipage}{0.25\linewidth}
\includegraphics[width=.95\linewidth,trim=70 30 50 15, clip]{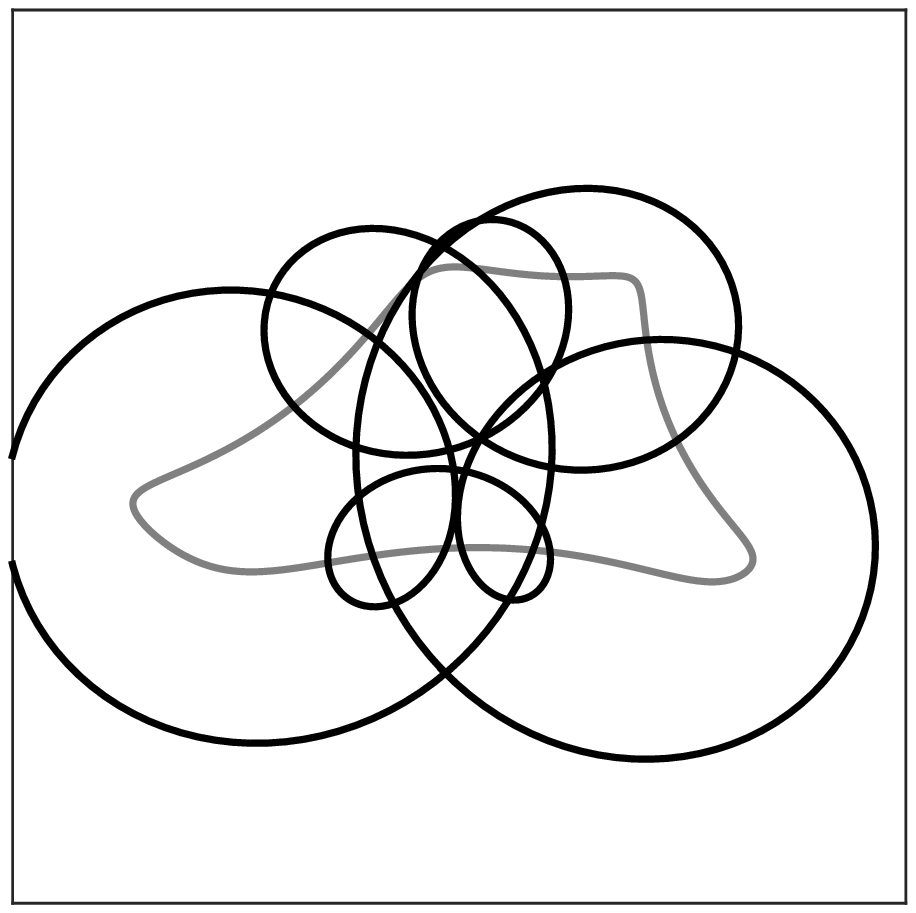}
\end{minipage}\,
\end{subfigure}\\
\hskip -1.5cm
\begin{subfigure}{\linewidth}
\raggedright\qquad
\begin{minipage}{0.11\linewidth}
\subcaption*{$\sigma=1/50$}
\subcaption*{SNR = $\infty$}
\end{minipage}\,
\begin{minipage}{0.25\linewidth}
\includegraphics[width=.95\linewidth,trim=70 30 50 15, clip]{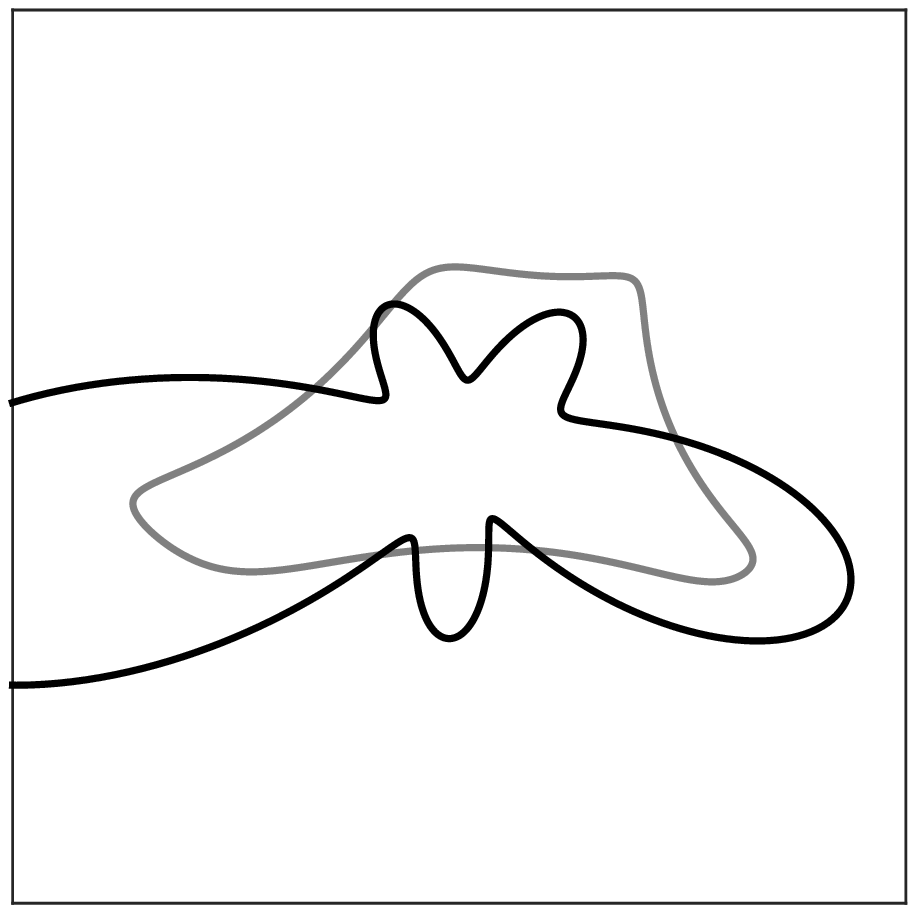}
\end{minipage}\,
\begin{minipage}{0.25\linewidth}
\includegraphics[width=.95\linewidth,trim=70 30 50 15, clip]{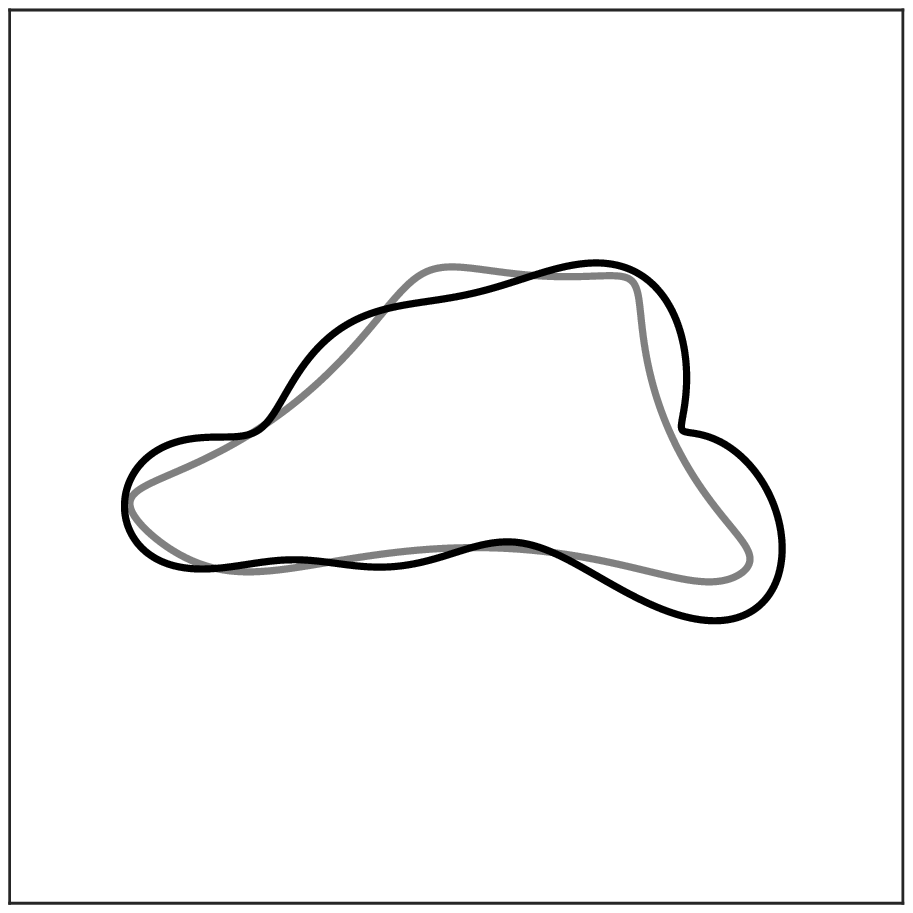}
\end{minipage}\,
\begin{minipage}{0.25\linewidth}
\includegraphics[width=.95\linewidth,trim=70 30 50 15, clip]{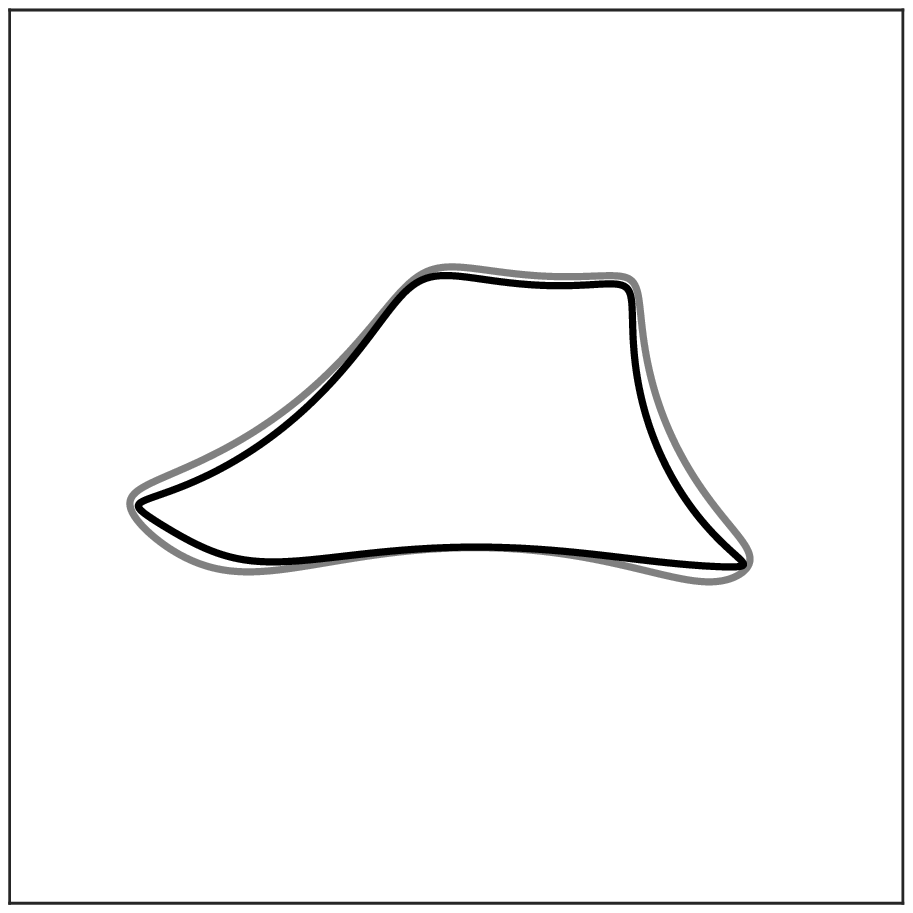}
\end{minipage}\,
\end{subfigure}\\
\hskip -1.5cm
\begin{subfigure}{\linewidth}
\raggedright\qquad
\begin{minipage}{0.11\linewidth}
\subcaption*{$\sigma=0$}
\subcaption*{SNR = $\infty$}
\end{minipage}\,
\begin{minipage}{0.25\linewidth}
\includegraphics[width=.95\linewidth,trim=70 30 50 15, clip]{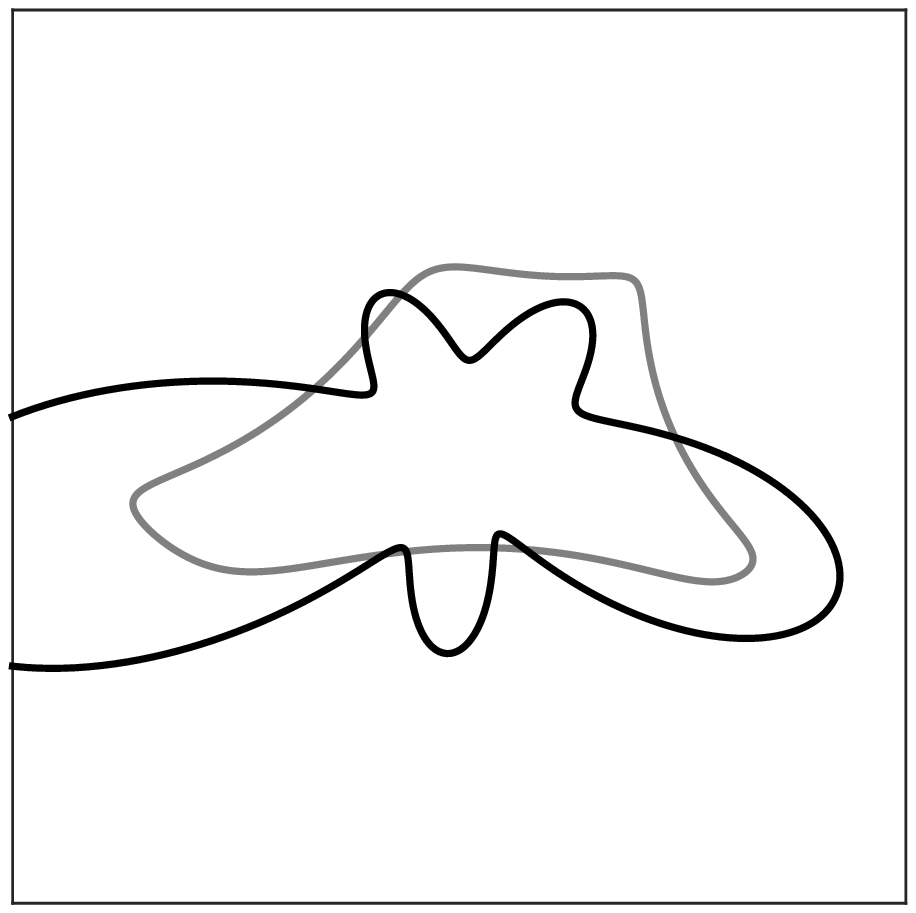}
\end{minipage}\,
\begin{minipage}{0.25\linewidth}
\includegraphics[width=.95\linewidth,trim=70 30 50 15, clip]{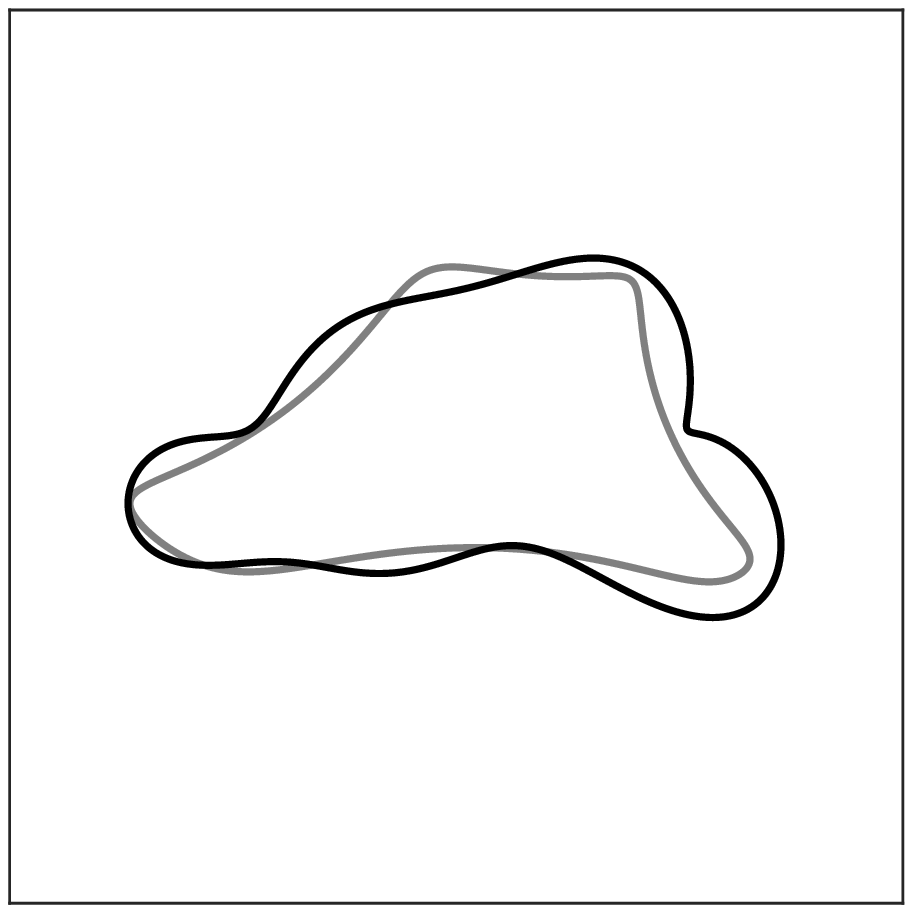}
\end{minipage}\,
\begin{minipage}{0.25\linewidth}
\includegraphics[width=.95\linewidth,trim=70 30 50 15, clip]{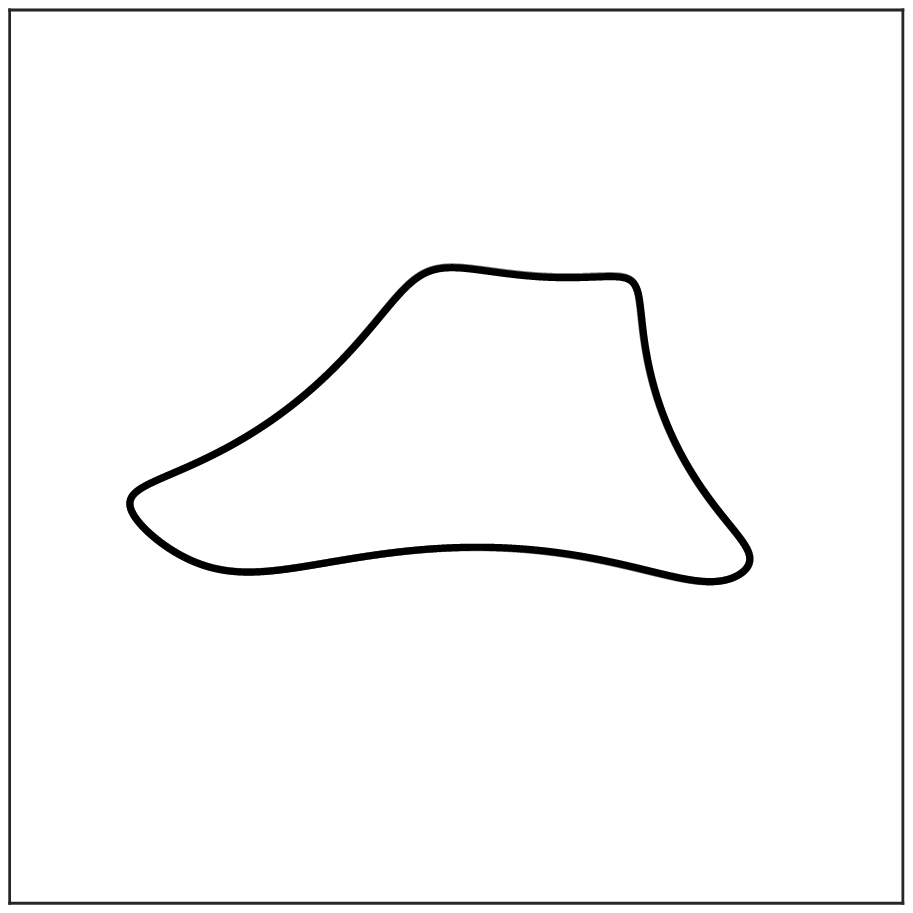}
\end{minipage}\,
\end{subfigure}
\caption{
Recovery of an asymmetric inclusion with various conductivities. The values of GPTs up to $\text{Ord}=6$ are used with SNR$=\infty$ (no noise).}
\label{fig:cap_initial}
\end{figure}
\begin{figure}[H]
\begin{subfigure}{\linewidth}
\raggedright\qquad
\captionsetup{justification=centering}
\begin{minipage}{0.11\linewidth}
\subcaption*{}
\end{minipage}\,
\begin{minipage}{0.25\linewidth}
\subcaption*{Perturbed disk\\ recovery}
\end{minipage}\,
\begin{minipage}{0.25\linewidth}
\subcaption*{Perturbed ellipse\\ recovery}
\end{minipage}\,
\begin{minipage}{0.25\linewidth}
\subcaption*{Conformal mapping\\ recovery}
\end{minipage}\,
\end{subfigure}\\
\hskip -1.5cm
\begin{subfigure}{\linewidth}
\raggedright\qquad
\begin{minipage}{0.11\linewidth}
\subcaption*{$\sigma=1/50$}
\subcaption*{SNR = 5}
\end{minipage}\,
\begin{minipage}{0.25\linewidth}
\includegraphics[width=.95\linewidth,trim=70 30 50 15, clip]{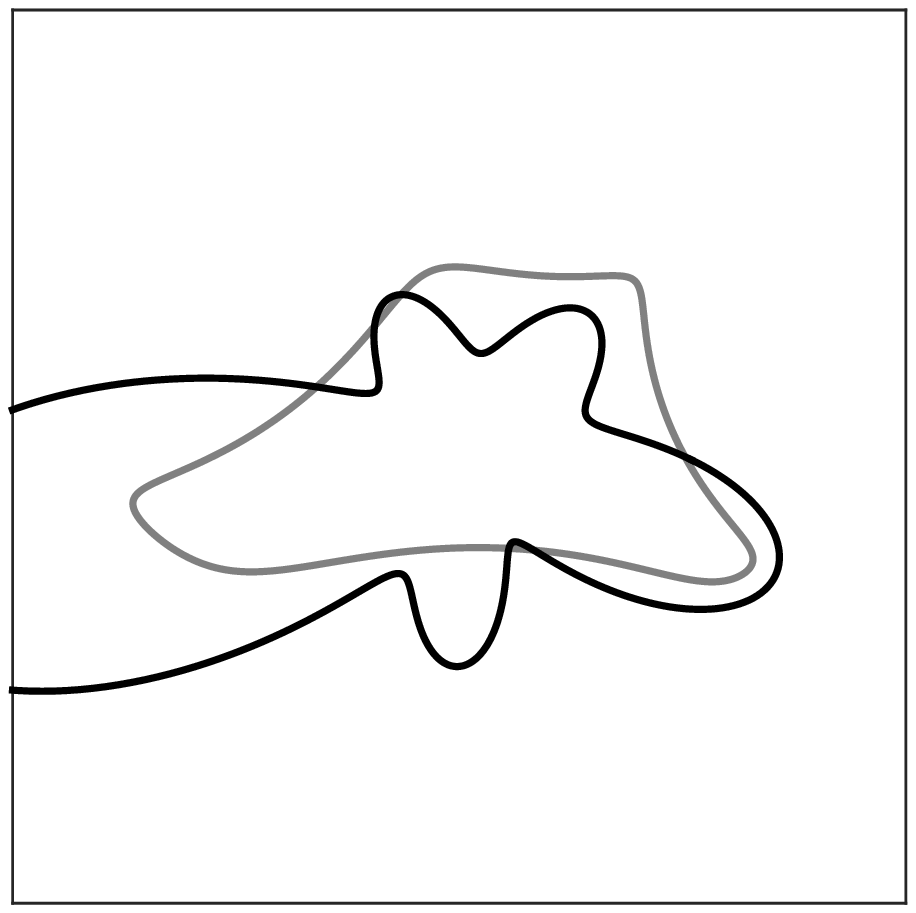}
\end{minipage}\,
\begin{minipage}{0.25\linewidth}
\includegraphics[width=.95\linewidth,trim=70 30 50 15, clip]{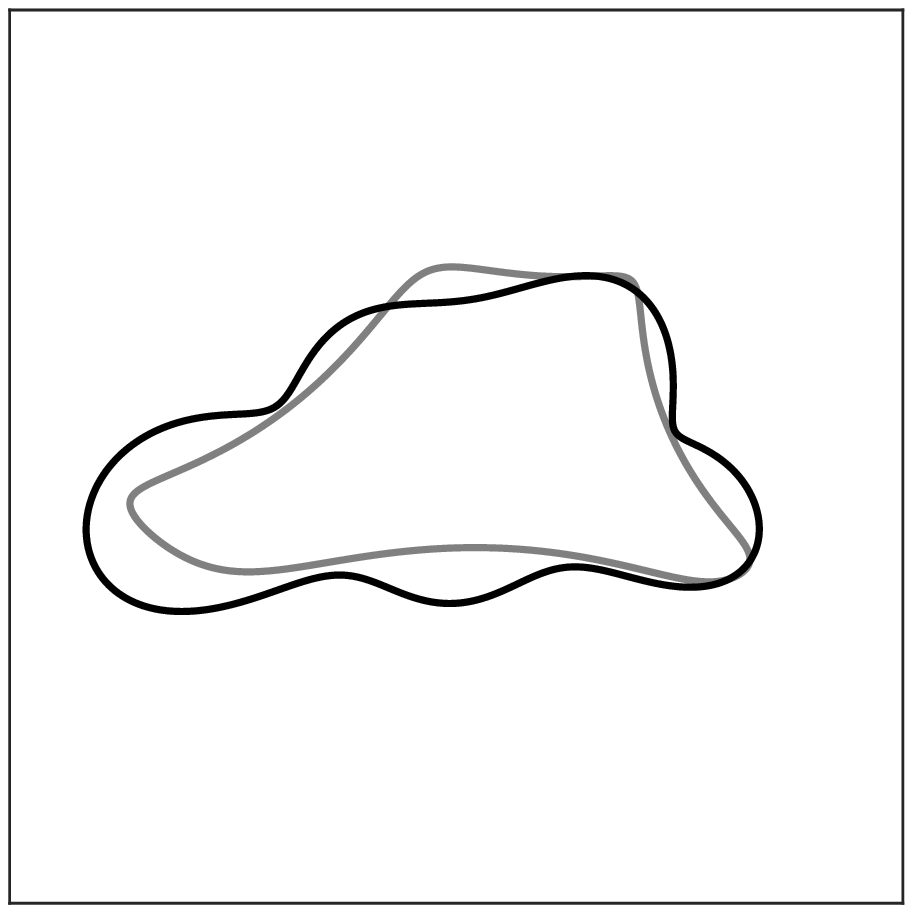}
\end{minipage}\,
\begin{minipage}{0.25\linewidth}
\includegraphics[width=.95\linewidth,trim=70 30 50 15, clip]{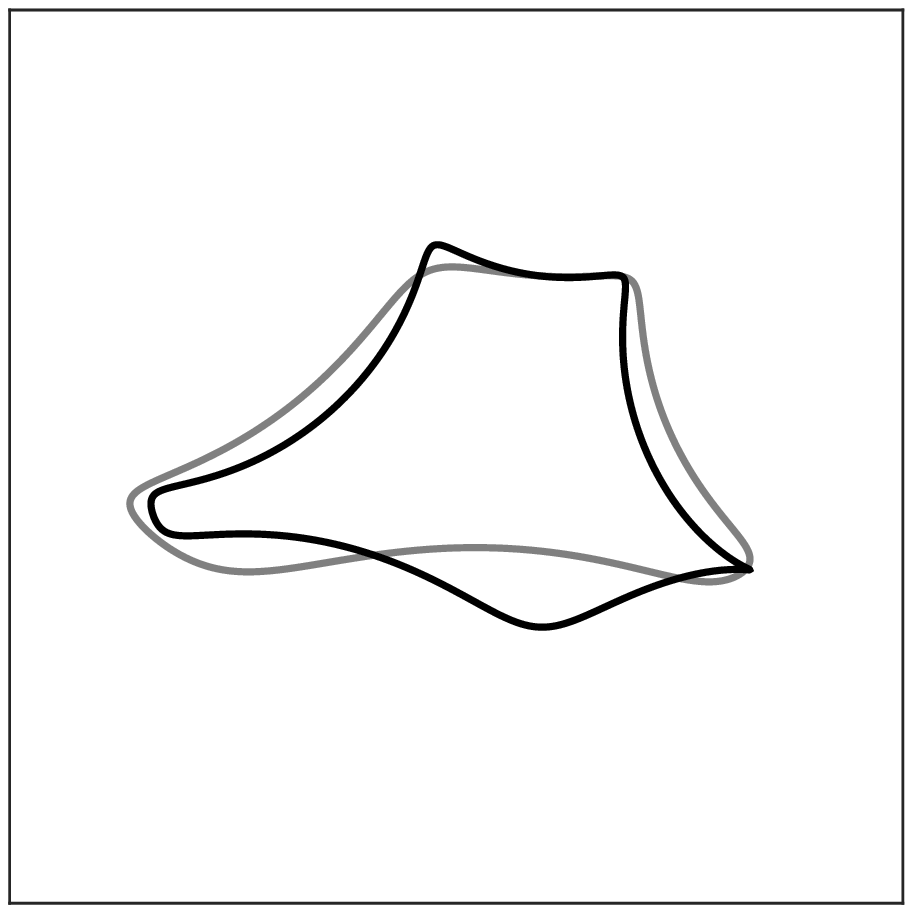}
\end{minipage}\,
\end{subfigure}\\
\hskip -1.5cm
\begin{subfigure}{\linewidth}
\raggedright\qquad
\begin{minipage}{0.11\linewidth}
\subcaption*{$\sigma=1/50$}
\subcaption*{SNR = 2}
\end{minipage}\,
\begin{minipage}{0.25\linewidth}
\includegraphics[width=.95\linewidth,trim=70 30 50 15, clip]{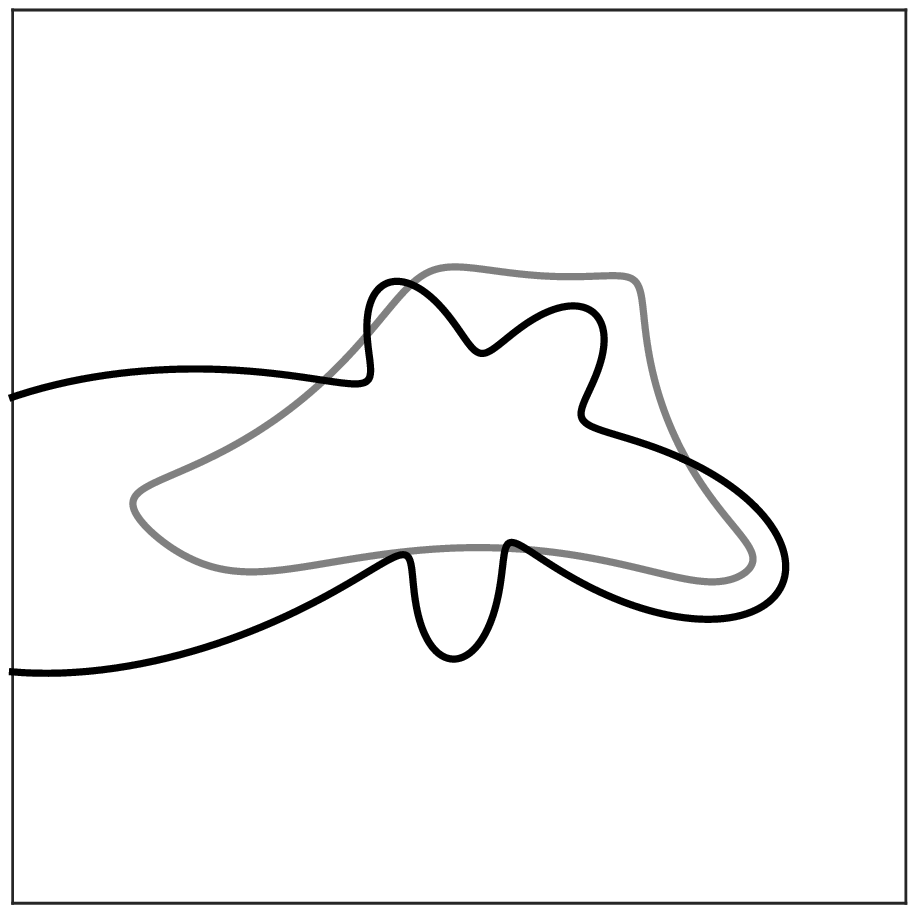}
\end{minipage}\,
\begin{minipage}{0.25\linewidth}
\includegraphics[width=.95\linewidth,trim=70 30 50 15, clip]{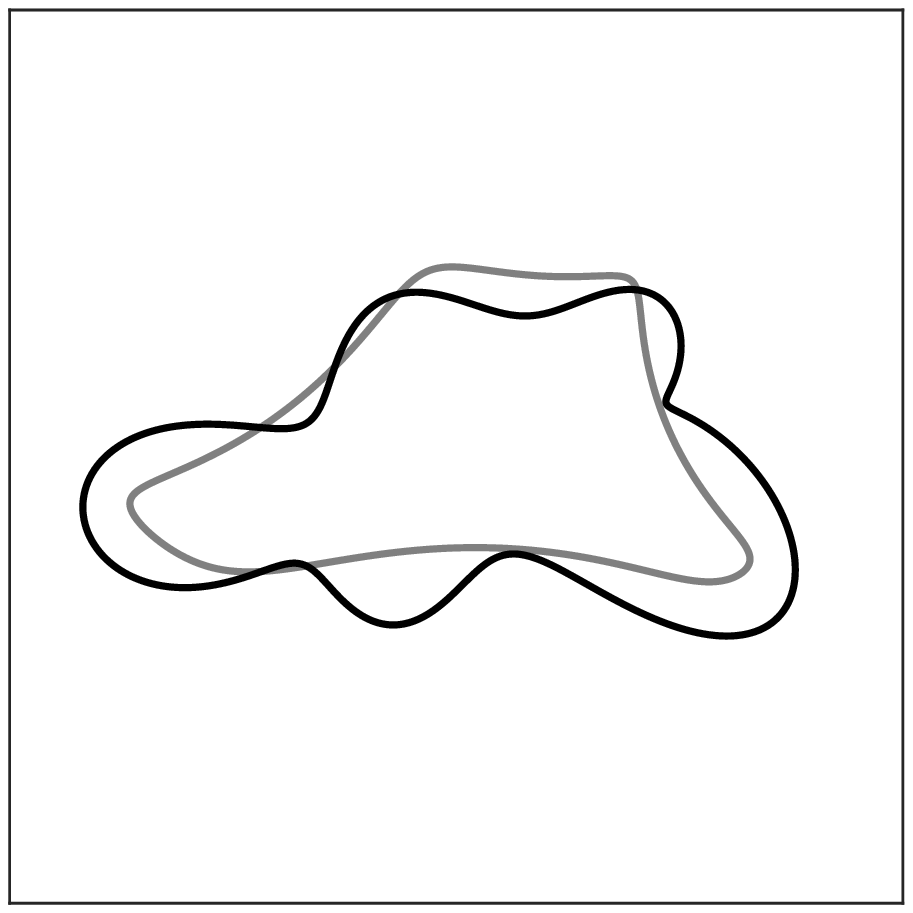}
\end{minipage}\,
\begin{minipage}{0.25\linewidth}
\includegraphics[width=.95\linewidth,trim=70 30 50 15, clip]{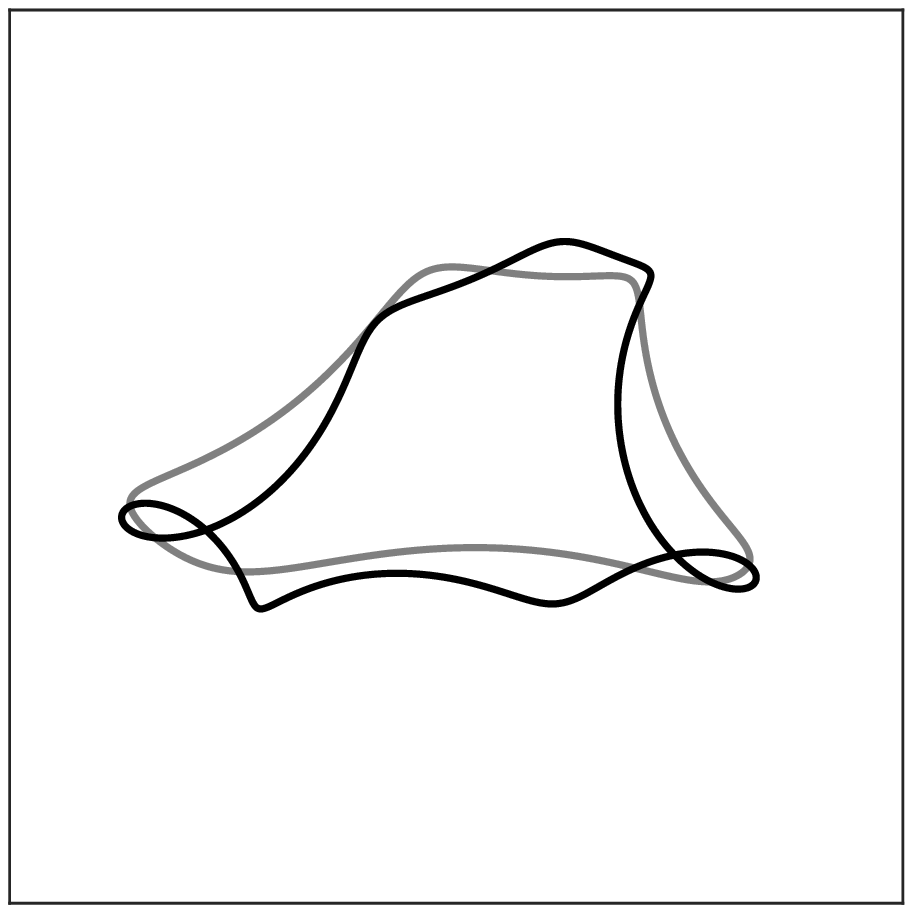}
\end{minipage}\,
\end{subfigure}
\caption{
The figures are numerically computed with noisy GPTs information with SNR=$5$, and $2$. The GPTs up to $\text{Ord}=6$ are used.}
\label{noise}
\end{figure}
\begin{figure}[h!]
\begin{subfigure}{\linewidth}
\raggedright\qquad
\captionsetup{justification=centering}
\begin{minipage}{0.11\linewidth}
\subcaption*{}
\end{minipage}\,
\begin{minipage}{0.25\linewidth}
\subcaption*{Perturbed disk\\ recovery}
\end{minipage}\,
\begin{minipage}{0.25\linewidth}
\subcaption*{Perturbed ellipse\\ recovery}
\end{minipage}\,
\begin{minipage}{0.25\linewidth}
\subcaption*{Conformal mapping\\ recovery}
\end{minipage}\,
\end{subfigure}\\
\hskip -1.5cm
\begin{subfigure}{\linewidth}
\raggedright\qquad
\begin{minipage}{0.11\linewidth}
\subcaption*{$\sigma=1/5$}
\end{minipage}\,
\begin{minipage}{0.25\linewidth}
\includegraphics[width=.95\linewidth,trim=70 30 50 15, clip]{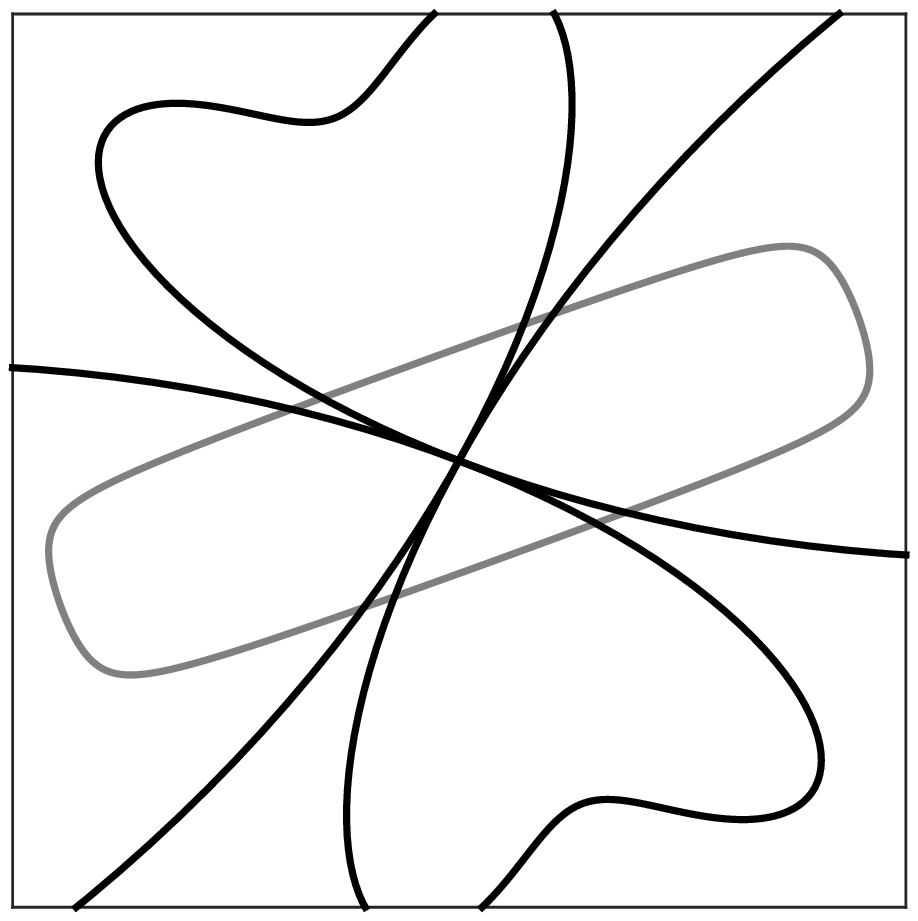}
\end{minipage}\,
\begin{minipage}{0.25\linewidth}
\includegraphics[width=.95\linewidth,trim=70 30 50 15, clip]{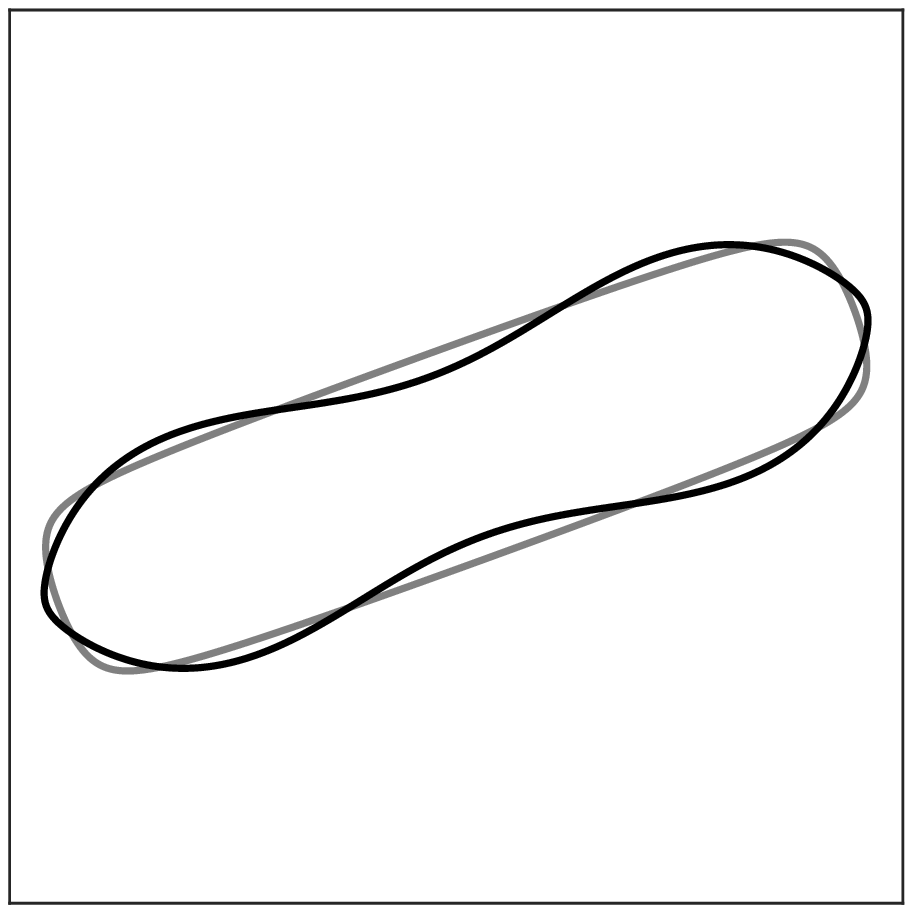}
\end{minipage}\,
\begin{minipage}{0.25\linewidth}
\includegraphics[width=.95\linewidth,trim=70 30 50 15, clip]{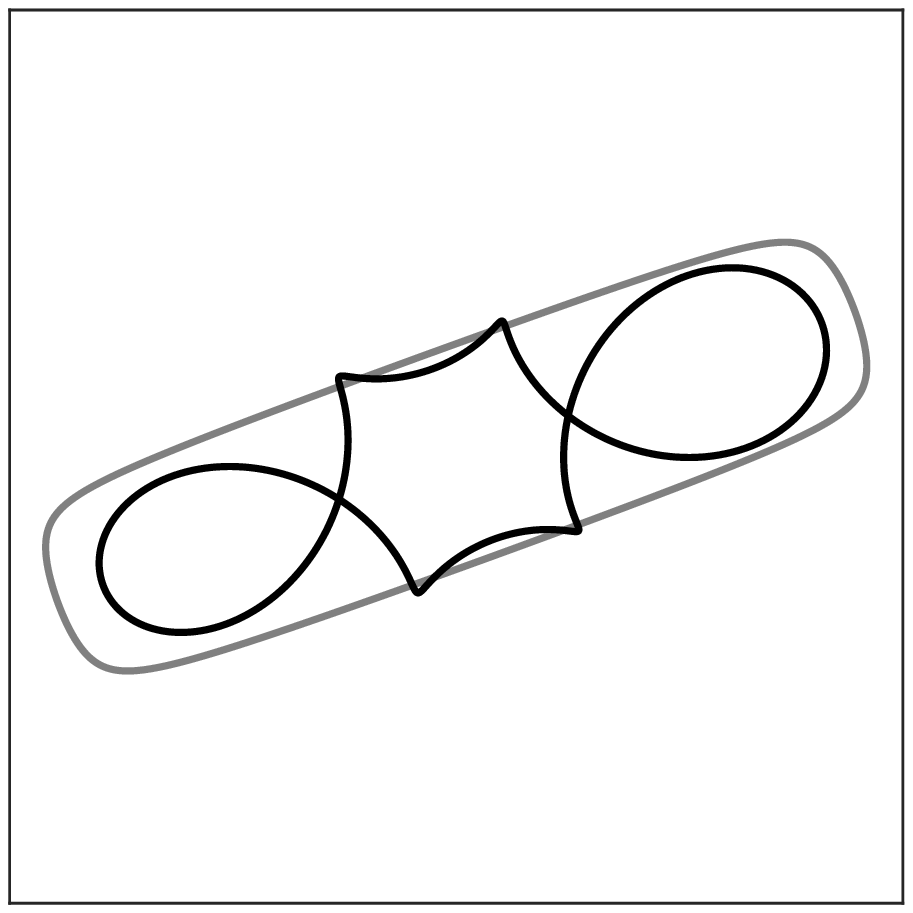}
\end{minipage}\,
\end{subfigure}\\
\hskip -1.5cm
\begin{subfigure}{\linewidth}
\raggedright\qquad
\begin{minipage}{0.11\linewidth}
\subcaption*{$\sigma=1/50$}
\end{minipage}\,
\begin{minipage}{0.25\linewidth}
\includegraphics[width=.95\linewidth,trim=70 30 50 15, clip]{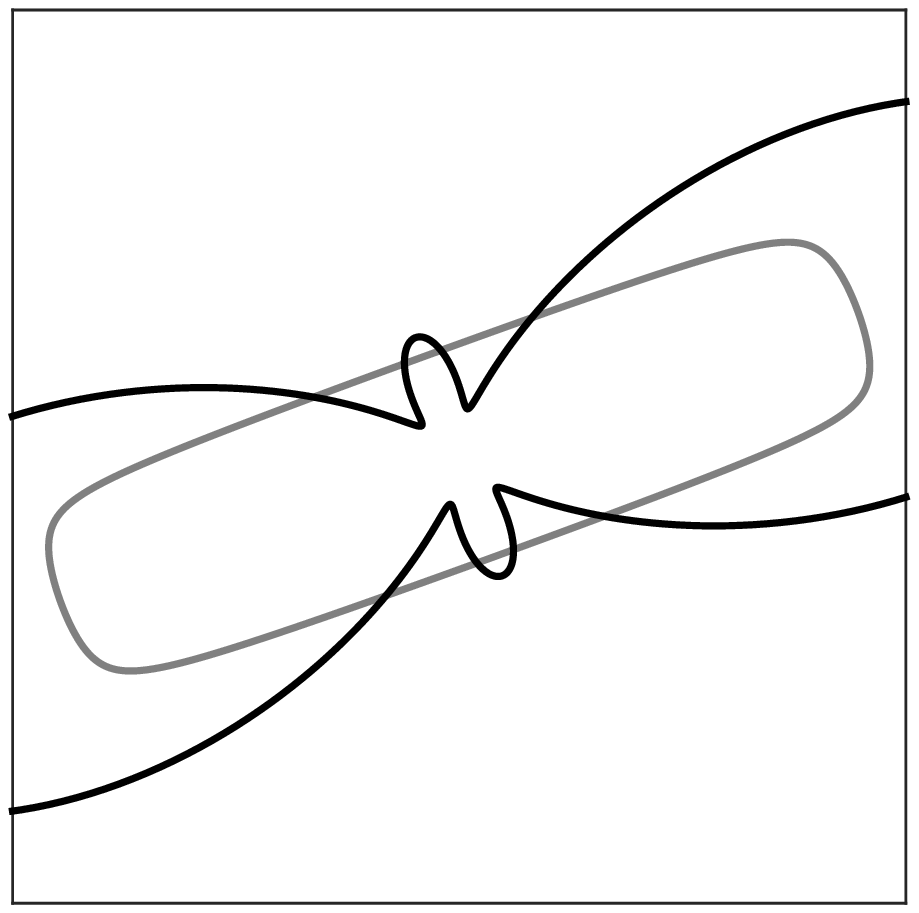}
\end{minipage}\,
\begin{minipage}{0.25\linewidth}
\includegraphics[width=.95\linewidth,trim=70 30 50 15, clip]{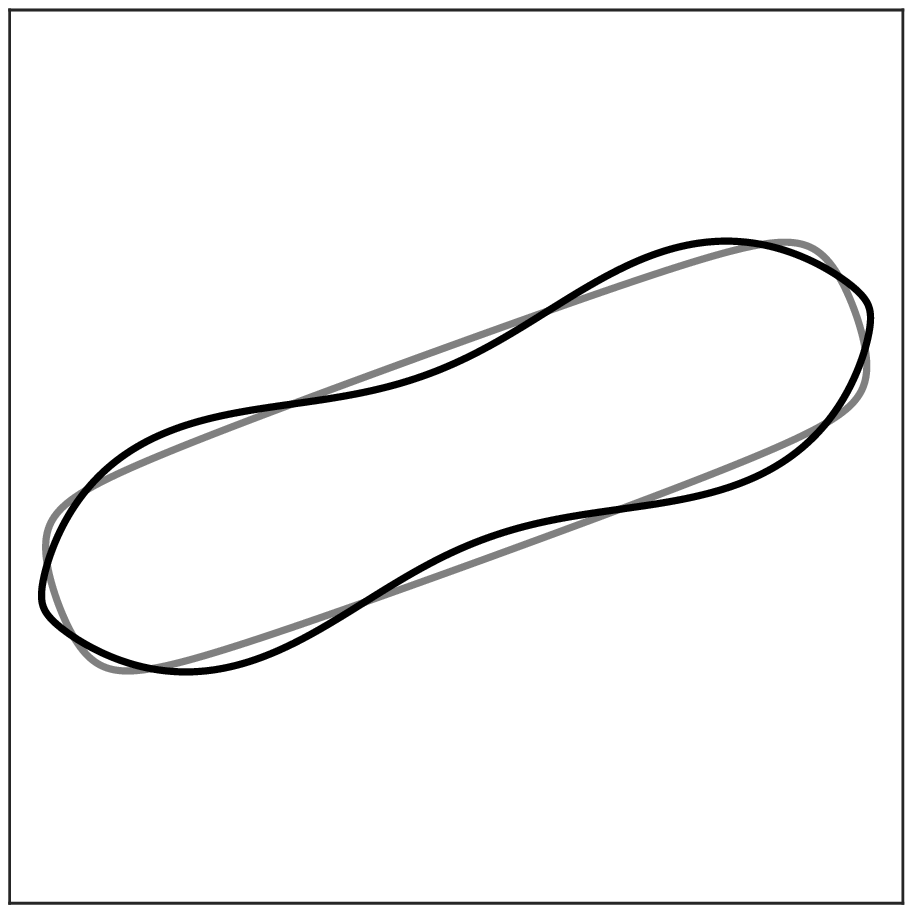}
\end{minipage}\,
\begin{minipage}{0.25\linewidth}
\includegraphics[width=.95\linewidth,trim=70 30 50 15, clip]{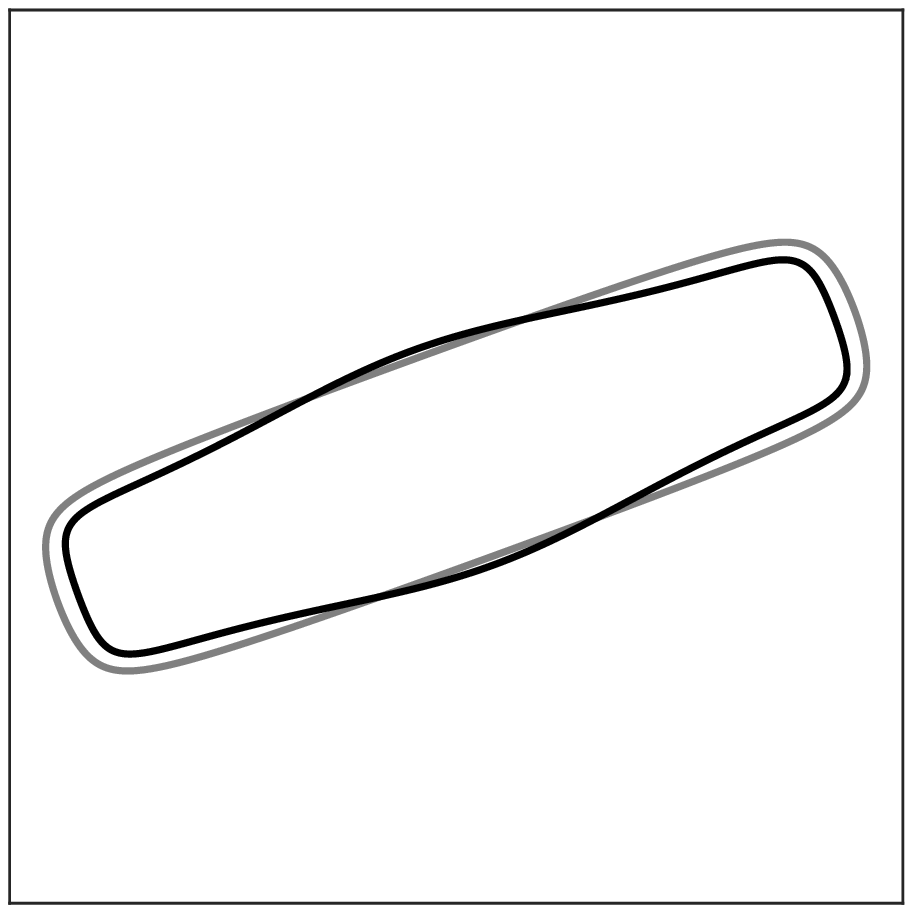}
\end{minipage}\,
\end{subfigure}\\
\hskip -1.5cm
\begin{subfigure}{\linewidth}
\raggedright\qquad
\begin{minipage}{0.11\linewidth}
\subcaption*{$\sigma=0$}
\end{minipage}\,
\begin{minipage}{0.25\linewidth}
\includegraphics[width=.95\linewidth,trim=70 30 50 15, clip]{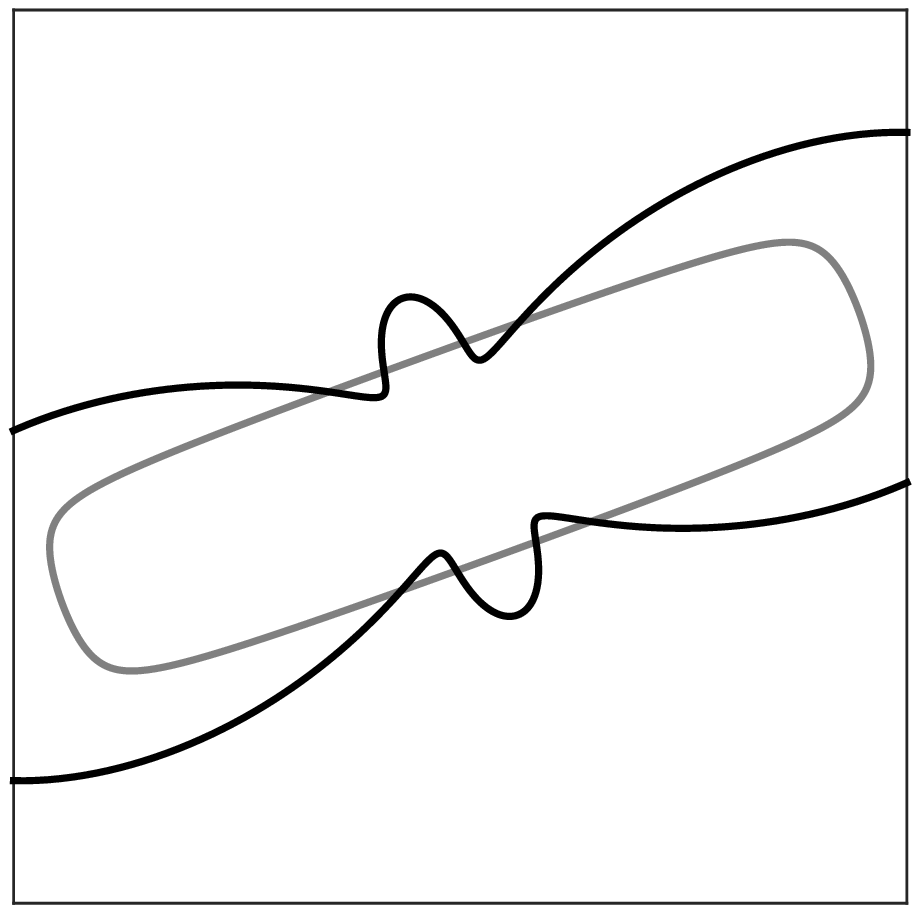}
\end{minipage}\,
\begin{minipage}{0.25\linewidth}
\includegraphics[width=.95\linewidth,trim=70 30 50 15, clip]{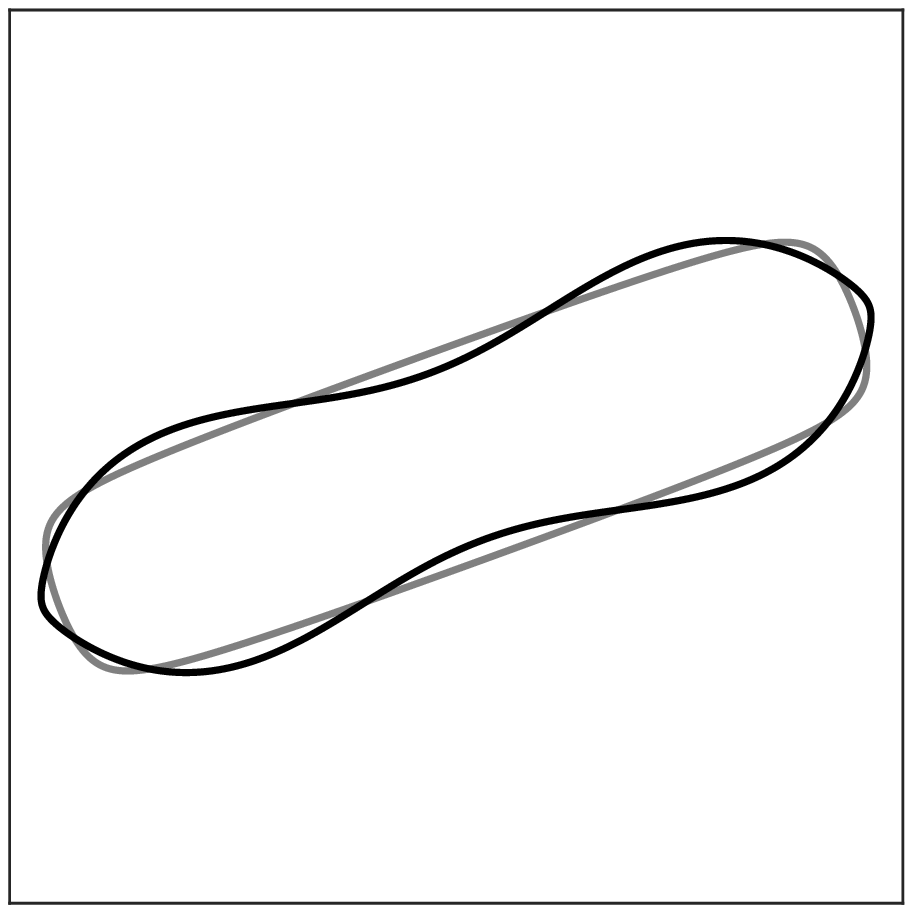}
\end{minipage}\,
\begin{minipage}{0.25\linewidth}
\includegraphics[width=.95\linewidth,trim=70 30 50 15, clip]{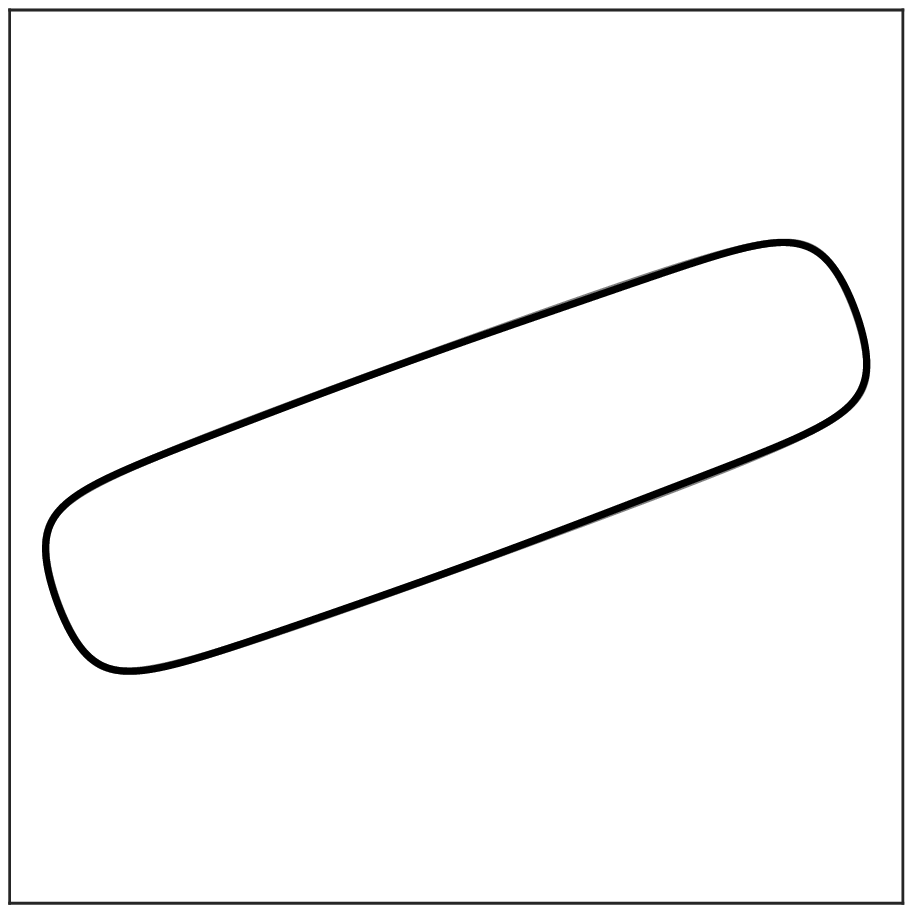}
\end{minipage}\,
\end{subfigure}
\caption{Recovery of a straight-shaped inclusion with various conductivities. The values of GPTs up to $\text{Ord}=6$ are used; see \eqnref{ordGPTs}. 
The perturbed ellipse recovery works well for all conductivity values. The conformal mapping recovery shows high accuracy for the extreme or near-extreme conductivity case.}
\label{fig:straight_initial}
\end{figure}
\begin{figure}[h!]
\begin{subfigure}{\linewidth}
\raggedright\qquad
\captionsetup{justification=centering}
\begin{minipage}{0.11\linewidth}
\subcaption*{}
\end{minipage}\,
\begin{minipage}{0.25\linewidth}
\subcaption*{Perturbed disk\\ recovery}
\end{minipage}\,
\begin{minipage}{0.25\linewidth}
\subcaption*{Perturbed ellipse\\ recovery}
\end{minipage}\,
\begin{minipage}{0.25\linewidth}
\subcaption*{Conformal mapping\\ recovery}
\end{minipage}\,
\end{subfigure}\\
\hskip -1.5cm
\begin{subfigure}{\linewidth}
\raggedright\qquad
\begin{minipage}{0.11\linewidth}
\subcaption*{$\sigma=5$}
\end{minipage}\,
\begin{minipage}{0.25\linewidth}
\includegraphics[width=.95\linewidth,trim=70 30 50 15, clip]{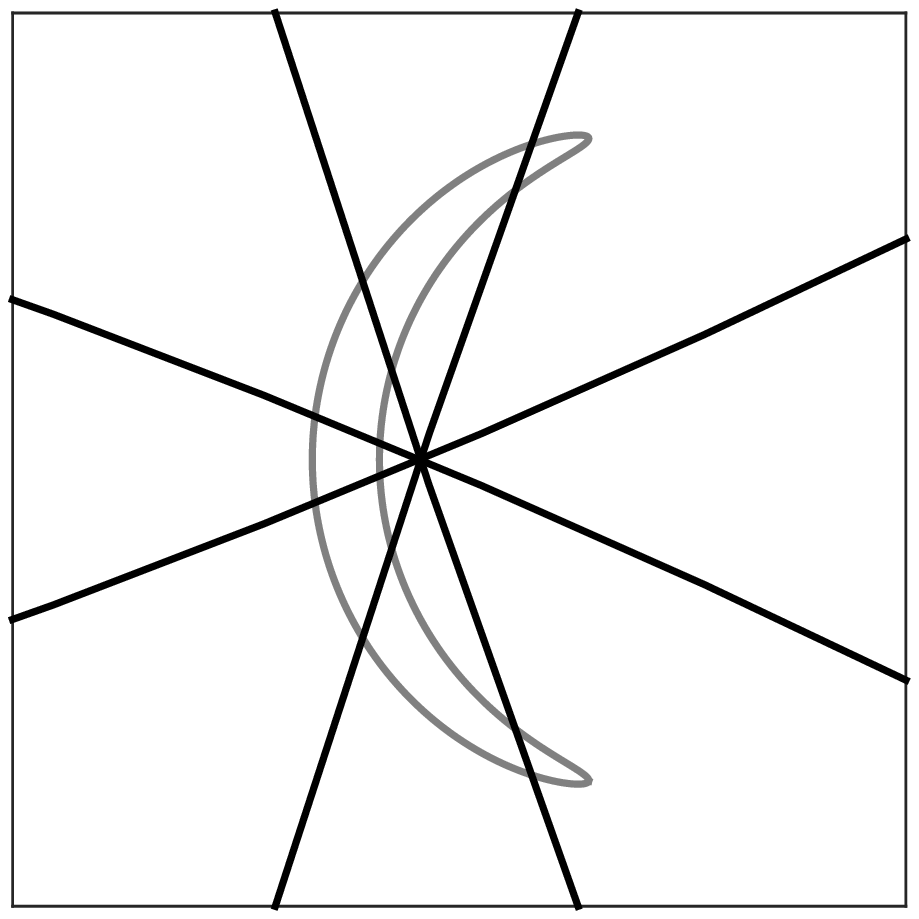}
\end{minipage}\,
\begin{minipage}{0.25\linewidth}
\includegraphics[width=.95\linewidth,trim=70 30 50 15, clip]{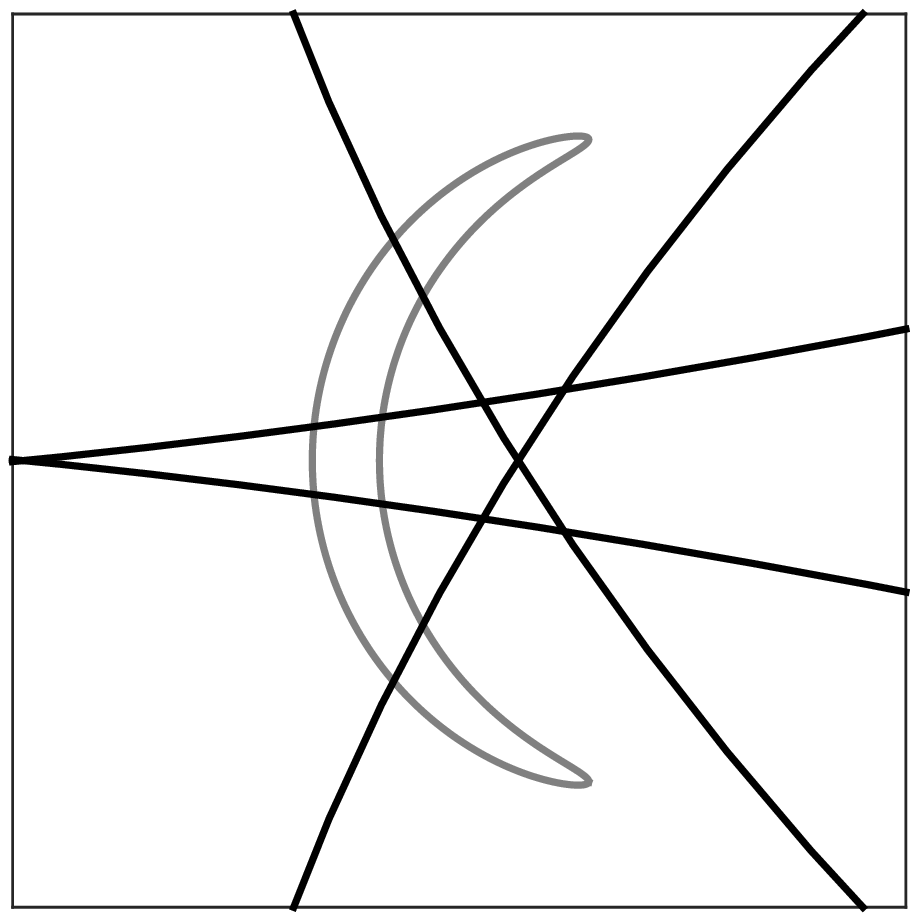}
\end{minipage}\,
\begin{minipage}{0.25\linewidth}
\includegraphics[width=.95\linewidth,trim=70 30 50 15, clip]{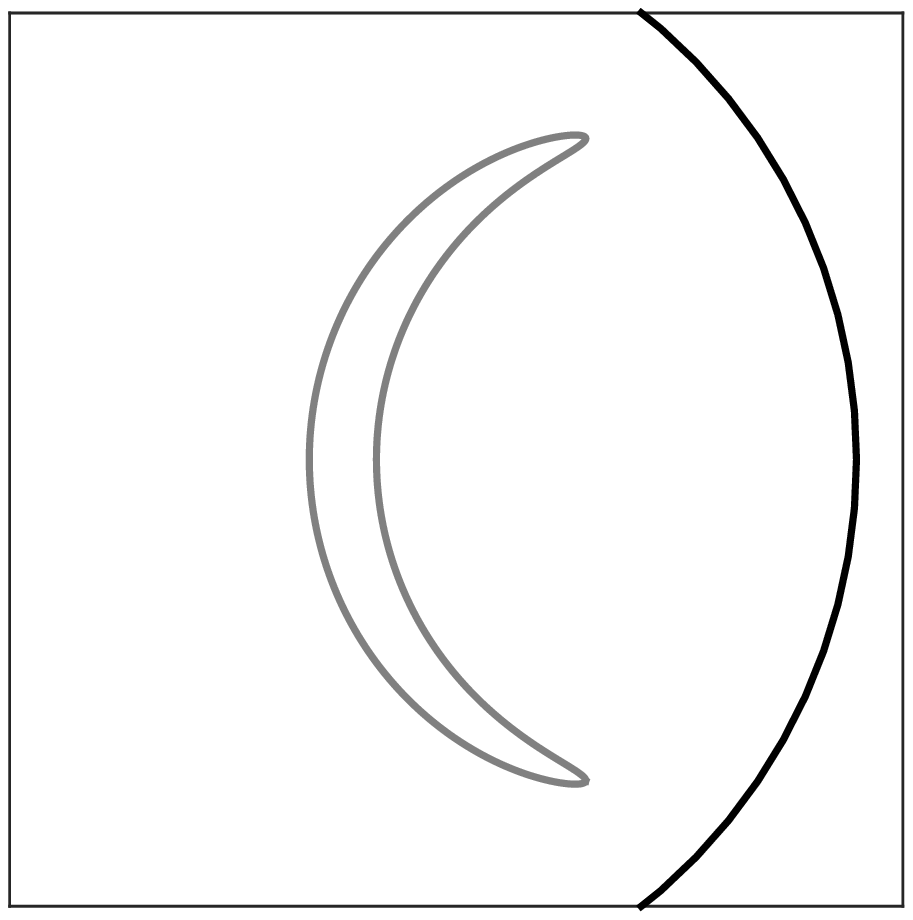}
\end{minipage}\,
\end{subfigure}\\
\hskip -1.5cm
\begin{subfigure}{\linewidth}
\raggedright\qquad
\begin{minipage}{0.11\linewidth}
\subcaption*{$\sigma=50$}
\end{minipage}\,
\begin{minipage}{0.25\linewidth}
\includegraphics[width=.95\linewidth,trim=70 30 50 15, clip]{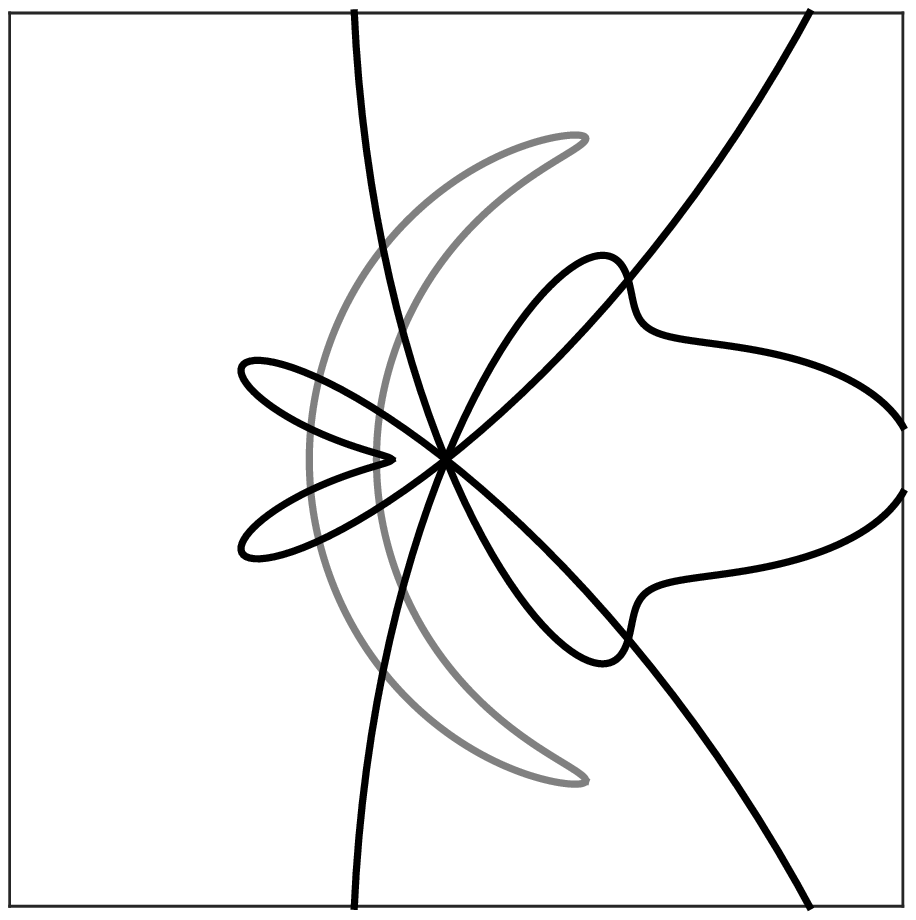}
\end{minipage}\,
\begin{minipage}{0.25\linewidth}
\includegraphics[width=.95\linewidth,trim=70 30 50 15, clip]{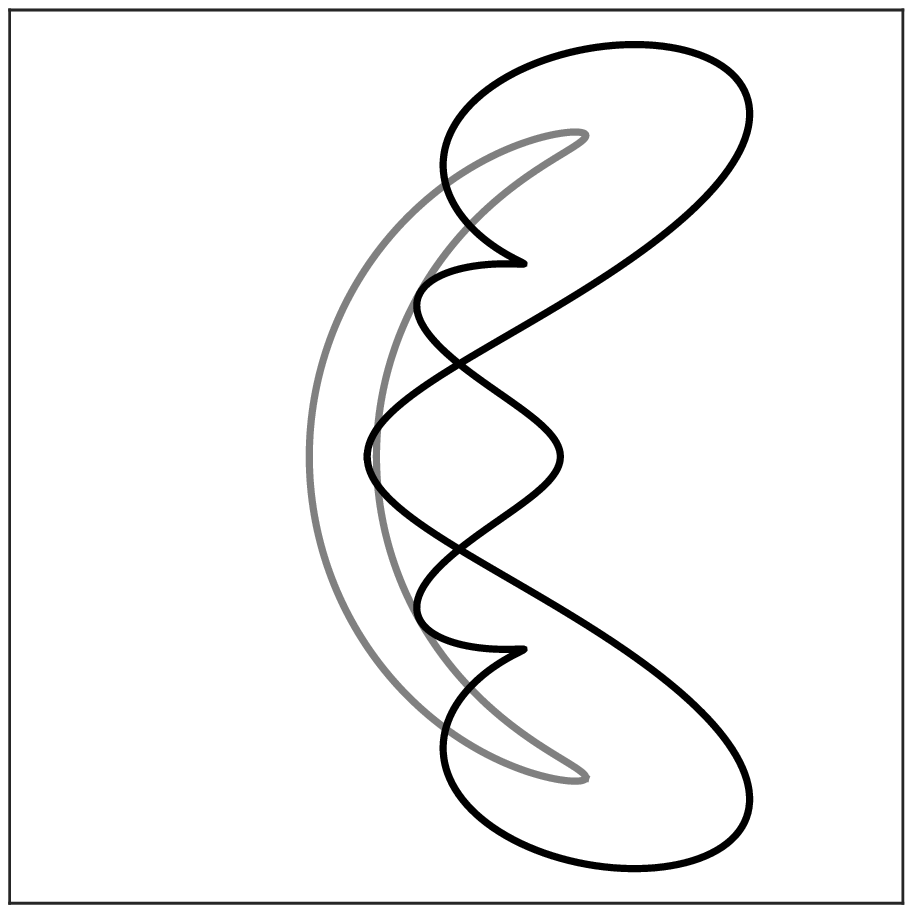}
\end{minipage}\,
\begin{minipage}{0.25\linewidth}
\includegraphics[width=.95\linewidth,trim=70 30 50 15, clip]{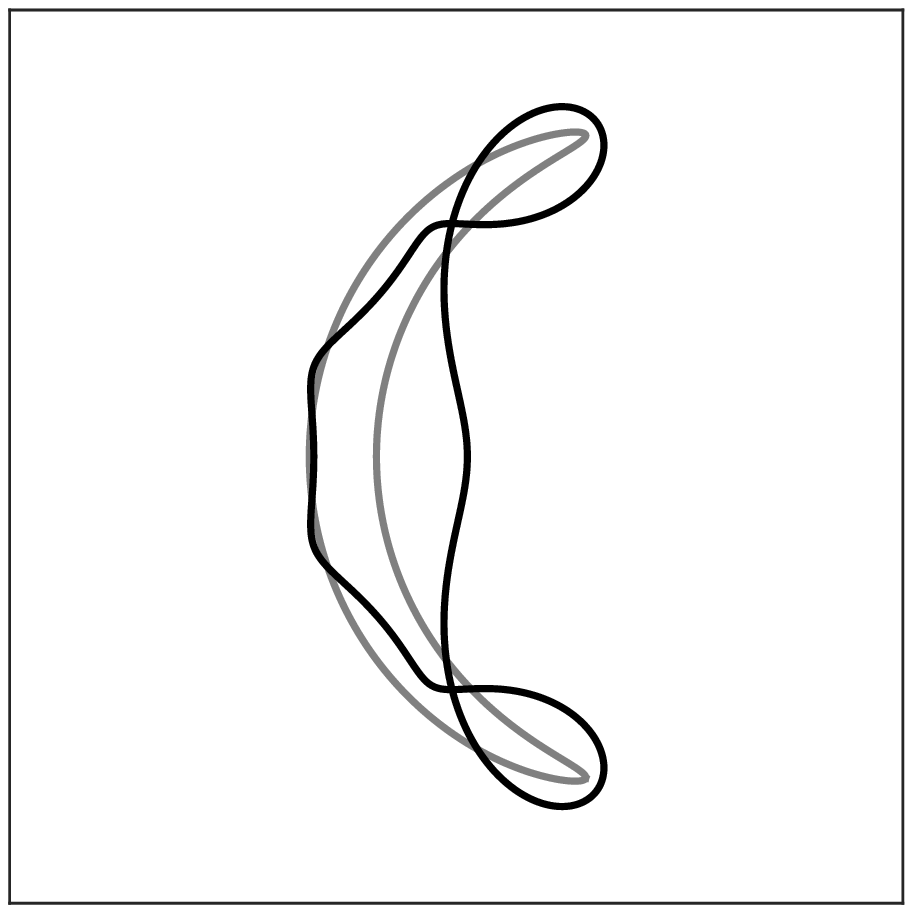}
\end{minipage}\,
\end{subfigure}\\
\hskip -1.5cm
\begin{subfigure}[b]{\linewidth}
\raggedright\qquad
\begin{minipage}{0.11\linewidth}
\subcaption*{$\sigma=\infty$}
\end{minipage}\,
\begin{minipage}{0.25\linewidth}
\includegraphics[width=.95\linewidth,trim=70 30 50 15, clip]{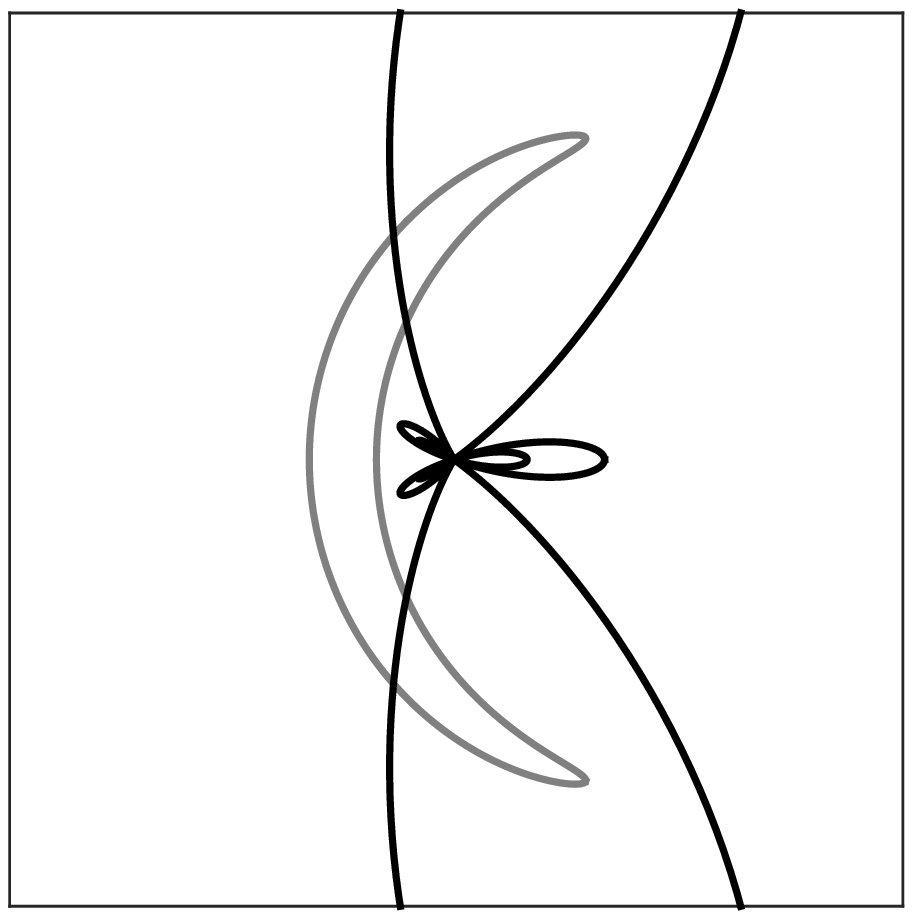}
\end{minipage}\,
\begin{minipage}{0.25\linewidth}
\includegraphics[width=.95\linewidth,trim=70 30 50 15, clip]{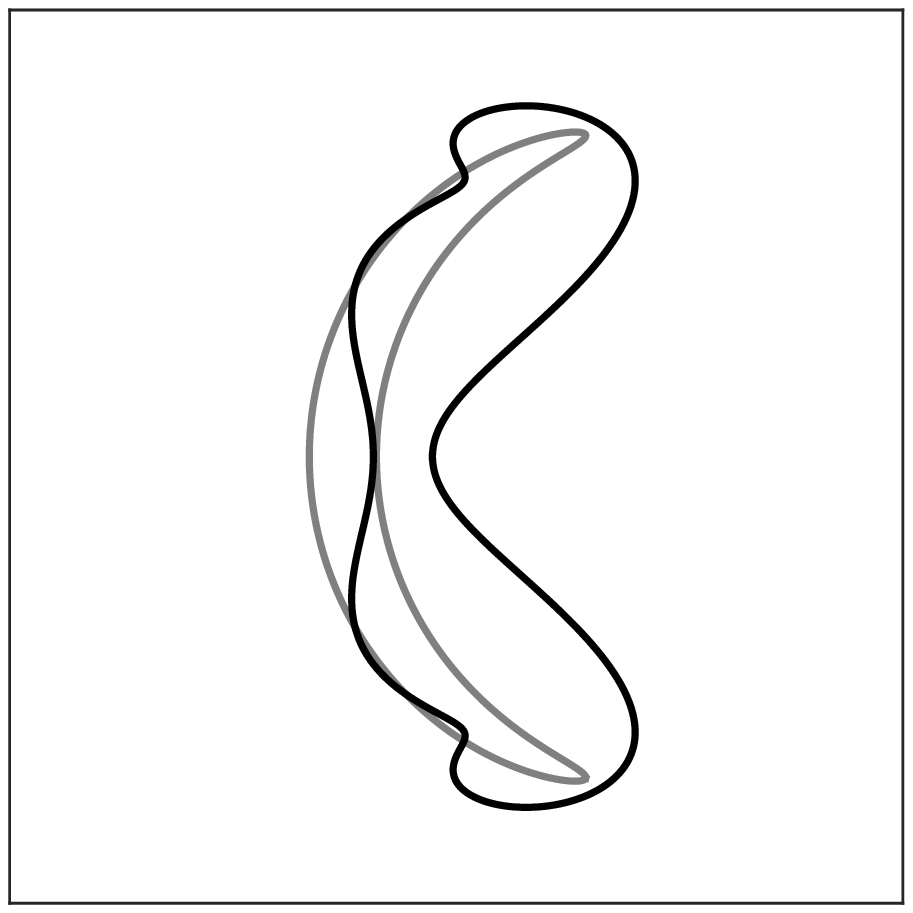}
\end{minipage}\,
\begin{minipage}{0.25\linewidth}
\includegraphics[width=.95\linewidth,trim=70 30 50 15, clip]{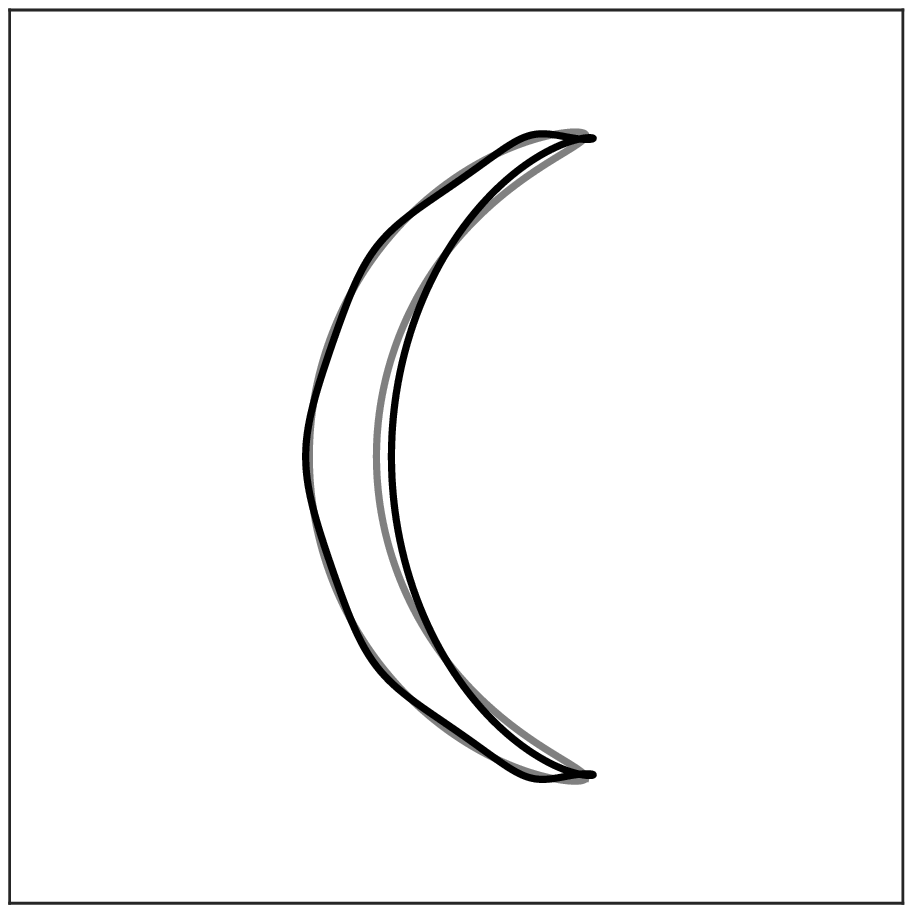}
\end{minipage}\,
\end{subfigure}
\caption{Recovery of a concave inclusion from the GPTs up to order $\text{Ord}=6$.
}
\label{fig:moon_initial}
\end{figure}

%\bibliographystyle{plain}
%\bibliography{2020_Choi_Kim_Lim_ASR}{}

\begin{thebibliography}{10}

\bibitem{Ahlfors:1963:QR}
Lars~V. Ahlfors.
\newblock Quasiconformal reflections.
\newblock {\em Acta Math.}, 109:291--301, 1963.

\bibitem{Ammari:2014:TIU}
Habib Ammari, Thomas Boulier, Josselin Garnier, Wenjia Jing, Hyeonbae Kang, and
  Han Wang.
\newblock Target identification using dictionary matching of generalized
  polarization tensors.
\newblock {\em Found. Comput. Math.}, 14(1):27--62, 2014.

\bibitem{Ammari:2018:MNF}
Habib Ammari, Doo~Sung Choi, and Sanghyeon Yu.
\newblock A mathematical and numerical framework for near-field optics.
\newblock {\em Proc. Roy. Soc. A.}, 474(2217), 2018.

\bibitem{Ammari:2013:STN}
Habib Ammari, Giulio Ciraolo, Hyeonbae Kang, Hyundae Lee, and Graeme~W. Milton.
\newblock Spectral theory of a {N}eumann-{P}oincar\'{e}-type operator and
  analysis of cloaking due to anomalous localized resonance.
\newblock {\em Arch. Ration. Mech. Anal.}, 208(2):667--692, 2013.

\bibitem{Ammari:2016:SPR}
Habib Ammari, Youjun Deng, and Pierre Millien.
\newblock Surface plasmon resonance of nanoparticles and applications in
  imaging.
\newblock {\em Arch. Ration. Mech. Anal.}, 220(1):109--153, 2016.

\bibitem{Ammari:2013:MSM}
Habib Ammari, Josselin Garnier, Wenjia Jing, Hyeonbae Kang, Mikyoung Lim, Knut
  S{\o}lna, and Han Wang.
\newblock {\em Mathematical and statistical methods for multistatic imaging},
  volume 2098 of {\em Lecture Notes in Mathematics}.
\newblock Springer, Cham, 2013.

\bibitem{Ammari:2014:GPT}
Habib Ammari, Josselin Garnier, Hyeonbae Kang, Mikyoung Lim, and Sanghyeon Yu.
\newblock Generalized polarization tensors for shape description.
\newblock {\em Numer. Math.}, 126(2):199--224, 2014.

\bibitem{Ammari:2003:PGP}
Habib Ammari and Hyeonbae Kang.
\newblock Properties of the {Generalized Polarization Tensors}.
\newblock {\em Multiscale Model. Simul.}, 1(2):335--348, 2003.

\bibitem{Ammari:2007:PMT}
Habib Ammari and Hyeonbae Kang.
\newblock {\em Polarization and moment tensors}, volume 162 of {\em Applied
  Mathematical Sciences}.
\newblock Springer, New York, 2007.
\newblock With applications to inverse problems and effective medium theory.

\bibitem{Ammari:2005:RCS}
Habib Ammari, Hyeonbae Kang, Eunjoo Kim, and Mikyoung Lim.
\newblock Reconstruction of closely spaced small inclusions.
\newblock {\em SIAM J. Numer. Anal.}, 42(6):2408--2428, 2005.

\bibitem{Ammari:2013:ENCa}
Habib Ammari, Hyeonbae Kang, Hyundae Lee, and Mikyoung Lim.
\newblock Enhancement of near cloaking using generalized polarization tensors
  vanishing structures. {P}art {I}: {T}he conductivity problem.
\newblock {\em Comm. Math. Phys.}, 317(1):253--266, 2013.

\bibitem{Ammari:2010:CIP}
Habib Ammari, Hyeonbae Kang, Mikyoung Lim, and Habib Zribi.
\newblock Conductivity interface problems. {P}art {I}: {S}mall perturbations of
  an interface.
\newblock {\em Trans. Amer. Math. Soc.}, 362(5):2435--2449, 2010.

\bibitem{Ammari:2012:GPT}
Habib Ammari, Hyeonbae Kang, Mikyoung Lim, and Habib Zribi.
\newblock The generalized polarization tensors for resolved imaging. {P}art
  {I}: {S}hape reconstruction of a conductivity inclusion.
\newblock {\em Math. Comp.}, 81(277):367--386, 2012.

\bibitem{Ammari:2018:IAD}
Habib Ammari, Mihai Putinar, Andries Steenkamp, and Faouzi Triki.
\newblock Identification of an algebraic domain in two dimensions from a finite
  number of its generalized polarization tensors.
\newblock {\em Math. Ann.}, 375(3-4):1337--1354, 2019.

\bibitem{Ammari:2018:RFD}
Habib Ammari, Matias Ruiz, Sanghyeon Yu, and Hai Zhang.
\newblock Reconstructing fine details of small objects by using plasmonic
  spectroscopic data. {P}art {II}: {T}he strong interaction regime.
\newblock {\em SIAM J. Imaging Sci.}, 11(3):1931--1953, 2018.

\bibitem{Beckermann:2018:BOP}
Bernhard Beckermann and Nikos Stylianopoulos.
\newblock Bergman orthogonal polynomials and the {G}runsky matrix.
\newblock {\em Constr. Approx.}, 47(2):211--235, 2018.

\bibitem{Bruehl:2003:DIT}
Martin Br\"{u}hl, Martin Hanke, and Michael~S. Vogelius.
\newblock A direct impedance tomography algorithm for locating small
  inhomogeneities.
\newblock {\em Numer. Math.}, 93(4):635--654, 2003.

\bibitem{Choi:2018:CEP}
Doo~Sung Choi, Johan Helsing, and Mikyoung Lim.
\newblock Corner effects on the perturbation of an electric potential.
\newblock {\em SIAM J. Appl. Math.}, 78(3):1577--1601, 2018.

\bibitem{Choi:2018:GME}
Doosung Choi, Junbeom Kim, and Mikyoung Lim.
\newblock Geometric multipole expansion and its application to neutral
  inclusions of general shape.
\newblock {\em arXiv preprint arXiv:1808.02446}, 2018.

\bibitem{Chui:1992:FSA}
C.~K. Chui, J.~St\"{o}ckler, and J.~D. Ward.
\newblock A {F}aber series approach to cardinal interpolation.
\newblock {\em Math. Comp.}, 58(197):255--273, 1992.

\bibitem{Curtiss:1964:HIF}
J.~H. Curtiss.
\newblock Harmonic interpolation in {F}ej\'{e}r points with the {F}aber
  polynomials as a basis.
\newblock {\em Math. Z.}, 86:75--92, 1964.

\bibitem{Curtiss:1966:SDP}
J.~H. Curtiss.
\newblock Solutions of the {D}irichlet problem in the plane by approximation
  with {F}aber polynomials.
\newblock {\em SIAM J. Numer. Anal.}, 3:204--228, 1966.

\bibitem{Duren:1983:UF}
Peter~L. Duren.
\newblock {\em Univalent functions}, volume 259 of {\em Grundlehren der
  Mathematischen Wissenschaften}.
\newblock Springer-Verlag, New York, 1983.

\bibitem{Ellacott:1983:CFS}
S.~W. Ellacott.
\newblock Computation of {F}aber series with application to numerical
  polynomial approximation in the complex plane.
\newblock {\em Math. Comp.}, 40(162):575--587, 1983.

\bibitem{Escauriaza:1992:RTW}
L.~Escauriaza, E.~B. Fabes, and G.~Verchota.
\newblock On a regularity theorem for weak solutions to transmission problems
  with internal {L}ipschitz boundaries.
\newblock {\em Proc. Amer. Math. Soc.}, 115(4):1069--1076, 1992.

\bibitem{Escauriaza:1993:RPS}
Luis Escauriaza and Jin~Keun Seo.
\newblock Regularity properties of solutions to transmission problems.
\newblock {\em Trans. Amer. Math. Soc.}, 338(1):405--430, 1993.

\bibitem{Faber:1903:UPE}
Georg Faber.
\newblock \"{U}ber polynomische {E}ntwickelungen.
\newblock {\em Math. Ann.}, 57(3):389--408, 1903.

\bibitem{Feng:2018:CGV}
Tingting Feng, Hyeonbae Kang, and Hyundae Lee.
\newblock Construction of {GPT}-vanishing structures using shape derivative.
\newblock {\em J. Comput. Math.}, 35(5):569--585, 2017.

\bibitem{Gao:2004:FSM}
C.-F Gao and N.~Noda.
\newblock Faber series method for two-dimensional problems of an arbitrarily
  shaped inclusion in piezoelectric materials.
\newblock {\em Acta Mech.}, 171(1):1--13, 2004.

\bibitem{Helsing:2013:SIE}
Johan Helsing.
\newblock Solving integral equations on piecewise smooth boundaries using the
  {RCIP} method: a tutorial.
\newblock {\em Abstr. Appl. Anal.}, pages Art. ID 938167, 20, 2013.

\bibitem{Helsing:2017:CSN}
Johan Helsing, Hyeonbae Kang, and Mikyoung Lim.
\newblock Classification of spectra of the {N}eumann-{P}oincar\'{e} operator on
  planar domains with corners by resonance.
\newblock {\em Ann. Inst. H. Poincar\'{e} Anal. Non Lin\'{e}aire},
  34(4):991--1011, 2017.

\bibitem{Jung:2018:NSS}
YoungHoon Jung and Mikyoung Lim.
\newblock A new series solution method for the transmission problem.
\newblock {\em arXiv preprint arXiv:1803.09458}, 2018.

\bibitem{Jung:2019:DEE}
YoungHoon Jung and Mikyoung Lim.
\newblock A decay estimate for the eigenvalues of the {N}eumann-{P}oincar\'{e}
  operator using the grunsky coefficients.
\newblock {\em Proc. Amer. Math. Soc.}, 148(2):591--600, 2020.

\bibitem{Kang:2015:CCM}
Hyeonbae Kang, Hyundae Lee, and Mikyoung Lim.
\newblock Construction of conformal mappings by generalized polarization
  tensors.
\newblock {\em Math. Methods Appl. Sci.}, 38(9):1847--1854, 2015.

\bibitem{Kang:2017:SRN}
Hyeonbae Kang, Mikyoung Lim, and Sanghyeon Yu.
\newblock Spectral resolution of the {N}eumann-{P}oincar\'{e} operator on
  intersecting disks and analysis of plasmon resonance.
\newblock {\em Arch. Ration. Mech. Anal.}, 226(1):83--115, 2017.

\bibitem{Kellogg:2012:FPT}
Oliver~Dimon Kellogg.
\newblock {\em Foundations of potential theory}.
\newblock Reprint from the first edition of 1929. Die Grundlehren der
  Mathematischen Wissenschaften, Band 31. Springer-Verlag, Berlin-New York,
  1967.

\bibitem{Khelifi:2014:BVP}
Abdessatar Khelifi and Habib Zribi.
\newblock Boundary voltage perturbations resulting from small surface changes
  of a conductivity inclusion.
\newblock {\em Appl. Anal.}, 93(1):46--64, 2014.

\bibitem{Kuehnau:1971:VKG}
Reiner K\"{u}hnau.
\newblock Verzerrungss\"{a}tze und {K}oeffizientenbedingungen vom
  {G}runskyschen {T}yp f\"{u}r quasikonforme {A}bbildungen.
\newblock {\em Math. Nachr.}, 48:77--105, 1971.

\bibitem{Luo:2009:FSM}
J.~C Luo and C.~F. Gao.
\newblock Faber series method for plane problems of an arbitrarily shaped
  inclusion.
\newblock {\em Acta Mech.}, 208(3):133, 2009.

\bibitem{Polya:1951:IIM}
G.~P\'{o}lya and G.~Szeg\"{o}.
\newblock {\em Isoperimetric {I}nequalities in {M}athematical {P}hysics}.
\newblock Annals of Mathematics Studies, no. 27. Princeton University Press,
  Princeton, N. J., 1951.

\bibitem{Pommerenke:1992:BBC}
Ch. Pommerenke.
\newblock {\em Boundary behaviour of conformal maps}, volume 299 of {\em
  Grundlehren der Mathematischen Wissenschaften}.
\newblock Springer-Verlag, Berlin, 1992.

\bibitem{Pommerenke:1975:UF}
Christian Pommerenke.
\newblock {\em Univalent functions}.
\newblock Vandenhoeck \& Ruprecht, G\"{o}ttingen, 1975.

\bibitem{Springer:1964:FEQ}
George Springer.
\newblock Fredholm eigenvalues and quasiconformal mapping.
\newblock {\em Acta Math.}, 111:121--142, 1964.

\bibitem{Verchota:1984:LPR}
Gregory Verchota.
\newblock Layer potentials and regularity for the {D}irichlet problem for
  {L}aplace's equation in {L}ipschitz domains.
\newblock {\em J. Funct. Anal.}, 59(3):572--611, 1984.

\bibitem{Yu:2017:SDA}
Sanghyeon Yu and Mikyoung Lim.
\newblock Shielding at a distance due to anomalous resonance.
\newblock {\em New Journal of Physics}, 19(3):033018, 2017.

\end{thebibliography}

\ifx \bblindex \undefined \def \bblindex #1{} \fi\ifx \bbljournal \undefined
  \def \bbljournal #1{{\em #1}\index{#1@{\em #1}}} \fi\ifx \bblnumber
  \undefined \def \bblnumber #1{{\bf #1}} \fi\ifx \bblvolume \undefined \def
  \bblvolume #1{{\bf #1}} \fi\ifx \noopsort \undefined \def \noopsort #1{} \fi

%%%%%%%%%%%%%%%%%%%%%%%%%%%%%%%%%%%%%%%%%%%%%%%%%%%%%%%%%%%%%%%%%%%%%%%%

%    Templates for common elements of a journal article; for additional
%    information, see the AMS-LaTeX instructions manual, instr-l.pdf,
%    included in the MCOM author package, and the amsthm user's guide,
%    linked from http://www.ams.org/tex/amslatex.html .

\end{document}